\documentclass[a4paper, 10pt]{article}
\usepackage{amsmath}
\numberwithin{equation}{section}
\usepackage{amssymb,esint,hyperref}
\usepackage{amscd}
\usepackage{xspace}
\usepackage{fancyhdr}
\usepackage{color}
\usepackage{verbatim}
\usepackage{cite}
\usepackage{stmaryrd}
\usepackage{bbm}
\usepackage{mathrsfs}
\usepackage{srcltx} 
\usepackage{marginnote}
\let\oldmarginpar\marginpar
\renewcommand\marginpar[1]{\-\oldmarginpar[\raggedleft\footnotesize #1]%
	{\raggedright\footnotesize\color{red} #1}} 
\newtheorem{theorem}{Theorem}[section]

\newtheorem{corollary}[theorem]{Corollary}

\newtheorem{lemma}[theorem]{Lemma}

\newtheorem{proposition}[theorem]{Proposition}
\newtheorem{remark}[theorem]{Remark}

\newenvironment{proof}[1][Proof]{\textbf{#1.} }{\hfill\rule{0.5em}{0.5em}}
{\catcode`\@=11\global\let\AddToReset=\@addtoreset
	\AddToReset{equation}{section}
	
	\AddToReset{theorem}{section}

		\title{The Peskin problem with $\dot B^1_{\infty,\infty}$ initial data}
\author{
	{\bf Ke Chen,\thanks{E-mail address: kchen18@fudan.edu.cn, Fudan University, 220 Handan Road, Yangpu, Shanghai, 200433, China.}~~Quoc-Hung Nguyen\thanks{E-mail address: qhnguyen@amss.ac.cn, Academy of Mathematics and Systems Science, Chinese Academy of Sciences, Beijing, 100190, China.}}}}
	
\begin{document}
		\maketitle
		\begin{abstract}
			In this paper we study the Peskin problem in 2D, which describes the dynamics of a 1D closed elastic structure immersed in a steady Stokes flow. We prove the local well-posedness for arbitrary initial configuration in $(C^2)^{\dot B^1_{\infty,\infty}}$ satisfying the well-stretched condition, and the global well-posedness when the initial configuration is sufficiently close to an equilibrium in $\dot B^1_{\infty,\infty}$. Here $(C^2)^{\dot B^1_{\infty,\infty}}$ is the closure of $C^2$ in  the Besov space $\dot B^1_{\infty,\infty}$. The global-in-time solution will converge to an equilibrium exponentially as $t\rightarrow+\infty$. This is the first well-posedness result for the Peskin problem with non-Lipschitz initial data. 
		\end{abstract}
		\section{Introduction and main results}
		Fluid structure interaction (FSI) problems in which a deformable structure interacts with a surrounding fluid are found in many areas of science and engineering. In this paper, we consider the problem of an elastic filament immersed in a two dimensional Stokes fluid. It is inspired by the numerical immersed boundary method introduced by Peskin \cite{PeskinFlow1972,Peskin1972Thesis} to study the flow patterens around heart valves.  The numerical study for such FSI problems has attracted a lot of interests, which gives birth to wide applications in physics, biology and medical sciences\cite{PeskinImmersed2002,MittalImmersed2005,HouNumerical2012}. The Peskin problem is named after Peskin in honor of his seminal contributions.

		Let $\Gamma$ be a simple closed curve which partitions $\mathbb{R}^2$ into two regions, the interior of the curve, $\Omega^+$ and the exterior $\Omega^-=\mathbb{R}^2\backslash\Omega^+$.
		Let $\Gamma$ be parameterized by vector valued function $X(t,s)=(X_1(t,s),X_2(t,s))\in\mathbb{R}^2$. Here $s\in\mathbb{T}:=\mathbb{R}/(2\pi\mathbb{Z})$ is the material coordinate and $t\geq 0$ denotes time. For fixed $s$, $X(t,s)$ moves with the local fluid velocity. Suppose further that the elastic structure has force density $F(X(t,s))$ with the form 
		\[
		F(X)=\partial_s(T(|\partial_s X|)\tau(X)),
		\]
		where $T$ is the tension and $\tau(X)=\frac{\partial_s X}{|\partial_s X|}$ is the unit tangent of the boundary $\Gamma$. Denote $u$ the fluid velocity and $p$ the pressure. The Peskin problem reads
		\begin{align*}
			\begin{cases}
				-\Delta u =-\nabla p~~~~~&\text{in} ~\mathbb{R}^2\backslash\Gamma(t),  \\
				\nabla \cdot u =0 ~~~~\quad\quad&\text{in} ~\mathbb{R}^2\backslash\Gamma(t),\\
				\llbracket u \rrbracket =0 ~~~~\quad\quad&\text{on} ~\Gamma(t),\\
				\llbracket\left(\nabla u+\nabla u^{T}-p \mathrm{Id}\right) n\rrbracket=\frac{F(X)}{|\partial_s X|}~~~&\text{on}~\Gamma(t),\\
				\partial_tX=u~~~~\quad\quad&\text{on}~ \Gamma(t).
			\end{cases}	
		\end{align*}
		Here $n$ is the outward unit normal to the free boundary $\Gamma(t)$ and $\llbracket\cdot\rrbracket$ denotes the jump across $\Gamma$:
		\[\llbracket U\rrbracket=U^+-U^-,\]
		where $U^{\pm}$ denotes the limiting value of $U$ evaluated on $\Gamma$ from the $\Omega^\pm$ side.
		
		Consider the particular case where each infinitesimal segment of the filament behaves like a Hookean spring with elasticity coefficient equal to 1, we have $T(x)=x$ and the force density can be written as $F(X)=\partial_s^2 X$.
		In this case the Peskin problem can be equivalently written as the following contour equations\cite{LinTongSolvability2019,MoriWell2019}
		\begin{align}
			\partial_{t} X(t,s) &=\int_{\mathbb{T}} \mathbf{G}(X(t,s)-X(t,\sigma)) \partial_s^2X(t,\sigma) \mathrm{d} \sigma, \label{peskin}\\
			\mathbf{G}(x)&=\frac{1}{4\pi}\left(-\log |x|         \mathrm{Id}+\frac{x\otimes x}{|x|^2}\right),\quad \quad x=(x_1,x_2)\in \mathbb{R}^2\backslash\{0\},\nonumber
		\end{align}
		where $\mathbf{G}$ is the fundamental solution of the $2D$ Stokes problem. It is easy to check that if $X(t,s)$ is a solution, then for any $\lambda>0$, $X_\lambda(t,s)=\lambda^{-1}X(\lambda t,\lambda s)$ is also a solution. Under this scaling, $\dot W^{1,\infty}, BMO^1$ and $\dot H^\frac{3}{2}$ are  critial spaces.
		
		The analytical study of the Peskin problem was initiated in \cite{LinTongSolvability2019,MoriWell2019}. 
		Lin and Tong \cite{LinTongSolvability2019} proved the local well-posedness for arbitrary $H^\frac{5}{2}$ data. Their proof relies on energy arguments and an application of the Schauder fixed point theorem. They also proved the global existence result and exponential decay towards equilibrium when the initial configuration is sufficiently close to the equilibrium. Tong \cite{TongRegularized2021} also established global well-posedness of a regularized Peskin problem and proved the convergence as the regularization parameter diminishes. Mori, Rodenberg and Spirn \cite{MoriWell2019} extended the results in \cite{LinTongSolvability2019}, they established a local well-posedness result for initial data in $C^{1,\gamma}$ with $\gamma>0$ (see also \cite{Rodenberg2018}).
		These spaces are subcritical under the scaling of the Peskin equation. For the well-posedness in critical spaces, Garcia-Juarez, Mori and Strain \cite{GarciaViscosityContrast2020} proved the global well-posedness result with initial data in the Wiener space $\mathcal{F}^{1,1}$ and sufficiently close to the stationary states. Their result holds for two fluids with different viscosity. More recently, Gancedo, Belinch\'{o}n and Scrobogna \cite{GancedoGlobal2020} considered a toy model of the Peskin problem and proved a global existence result in the critical Lipschitz space.
		
		There are also a lot of analytical studies on FSI problems considering an elastic structure interacts with a fluid (see \cite{Cheng2007Navier,Cheng2010The,AmSie,AmLiu,Li2020Stability}). The Peskin problem is essentially simpler than other FSI models mentioned above. It is interesting to study the behavior of the Peskin problem and to consider whether the results can be extended to more complicated models.
		
		The Peskin problem has many similarities with the Muskat problem.  The Muskat equation models the evolution of the interface between two different fluids in porous media whose dynamics are governed by Darcy's law.
		The Muskat equation in 2 dimension reads 
		\begin{align*}
			\partial_tz(s)=\frac{1}{\pi}\int_\mathbb{R}\frac{( z_1(s)-z_1(s-\alpha))}{|z(s)-z(s-\alpha)|^2}(\partial_sz(s)-\partial_s z(s-\alpha))d\alpha,
		\end{align*}
		where $z(s)=(z_1(s),z_2(s))$ is the interface. 
		The analysis of the Muskat equation can be traced back to the work of C\'{o}rdoba, C\'{o}rdoba and Gancedo\cite{Cordoba2011}, which proved the local existence in $H^k$ ($k\geq 3$) under the Rayleigh-Taylor condition and the arc-chord condition. See also \cite{1Constantin2010,1Gancedo2020Global,Alazard-Lazar,1HuyNguyen2020,1Matioc2018,HuyNguyen2021} for further developments on this problem. There are a large amount of studies (see \cite{ThomasHfirst,TH4,KeC1,DengLei2017,Cameron2019,Cameron2020,Cordoba-Lazar-H3/2} and references therein) considering the Muskat equation in the graph case (i.e. $z(s)=(s,f(s))$), which can be written as  
		\begin{align*}
			\partial_tf(s)=\frac{1}{\pi}\int_\mathbb{R}\frac{\alpha\partial_s\delta_\alpha f(s)}{\alpha^2+(\delta_\alpha f(s))^2}{d\alpha},
		\end{align*}
		where we denote $\delta_\alpha f(s)=f(s)-f(s-\alpha)$.
		We can rewrite the equation as 
		\begin{align*}
			\partial_t f(s)+\frac{\Lambda f(s)}{1+(f'(s))^2}=F(f(s)),
		\end{align*}
		where $F(f)$ is the remainder nonlinear term and $\Lambda=(-\Delta)^\frac{1}{2}$ is the fractional Laplacian in $\mathbb{R}$. The main part of the Muskat equation is nonlinear and degenerate when the initial data is not Lipschitz, which makes the problem more difficult (see \cite{Alazard2020endpoint,Alazard2020} for more discussion). Up to now, the well-posedness of the Muskat equation in $BMO^1$ is still open.
		%
		On the other hand, the Peskin problem reads 
		\begin{align*}
			\partial_t X+\frac{1}{4}\Lambda X+\frac{1}{4}\left(\begin{array}{cc}
				0 & -\mathcal{H} \\
				\mathcal{H} & 0
			\end{array}\right) X=N(X),
		\end{align*}
		where $\mathcal{H}$ is the Hilbert transform and $N(X)$ denotes remainder term. Fortunately,
		the main part is linear and non-degenerate, which makes it possible to establish the well-posedness in $\dot B^1_{\infty,\infty}$. We introduce the main ideas of this paper in the following.
		
		The main difficulty is to choose a function space to work in.	To solve this problem, we consider the following toy model:
		\begin{align}\label{toymodel}
			\partial_t f(t,s)+\frac{1}{4}\Lambda f(t,s)=|\Lambda^\sigma f(t,s)|^\frac{1}{\sigma}, \quad0<\sigma<1,
		\end{align}
	 Here we denote $\Lambda^\sigma=(-\Delta)^\frac{\sigma}{2}$. Note that $f$ will be like $\partial_s X$ in the Peskin problem. The solution of the above model has the formula $$
		f(t,s)=\int K(t,s-s')f_0(s')ds'+\int_0^t\int K(t-\tau,s-s')|\Lambda^\sigma f(\tau,s')|^\frac{1}{\sigma}ds'd\tau,$$
		where $K$ is the kernel associate to $\partial_t+\frac{1}{4}\Lambda$ (see Section 2 for more discussion). 
		By classical regularity argument, to control the nonlinear part of the solution, one needs $\|\Lambda^b f\|_{L_T^\frac{1}{b}L^\infty}<\infty$ for some $b\in [\sigma,1)$. However, for any $m\in\mathbb{Z}^+,~b\in(0,1),$ there holds
		$$
		\|\partial_s^mK(t,\cdot)\|_{L^1}\sim t^{-m}\quad \text{and}\quad\|\Lambda^b K(t,\cdot)\|_{L^1}\sim t^{-b}.$$
		Generally speaking, $\|\Lambda^b(K\ast f_0)\|_{L_T^\frac{1}{b}L^\infty}$ is not finite for $f_0\in L^\infty$ (even for  $f_0\in C^0$). To fix this, we observe that 
		$$
		\|\delta_\alpha \Lambda^{b-\varepsilon'}K(t,\cdot)\|_{L^1}\lesssim\min\left\{1,|\alpha|t^{-1}\right\}t^{-(b-\varepsilon')},
		$$
		for $0<\varepsilon'<\frac{b}{2}$. Moreover, there holds
		$$
		\left\|\min\{1,|\alpha|t^{-1}\}t^{-(b-\varepsilon')}\right\|_{L_T^\frac{1}{b}}\lesssim |\alpha|^{\varepsilon'},$$
		which implies
		$$
		\sup_\alpha \frac{\left\|\delta_\alpha \Lambda^{b-\varepsilon'} (K\ast f_0) \right\|_{L_T^\frac{1}{b}L^\infty}}{|\alpha|^{\varepsilon'}}\lesssim ||f_0||_{\dot B_{\infty,\infty}^0}.
		$$
Here $\dot B^0_{\infty,\infty}$ is the Besov space with index $(0, \infty,\infty)$.	We will explain more details of this estimate in Lemma \ref{lemequiv}. This motivates us to define a new norm in which we move the derivative in space outside the integration in time.  More precisely, we introduce a space $\mathcal{G}_T$ of functions in $[0,T]\times\mathbb{R}$ with norm
		\begin{equation*}
			\|h\|_{\mathcal{G}_T}= \sup_{0\leq \mu\leq \frac{2}{3}\atop 2\varepsilon'\leq b\leq\theta-\mu-\varepsilon' }\sup_{\alpha\in\mathbb{R}}\frac{\|t^\mu\delta_\alpha \Lambda^{b-\varepsilon'} h\|_{L_T^\frac{1}{b}L^\infty}}{|\alpha|^{\mu+\varepsilon'}},
		\end{equation*}
		where $\theta$ is a constant close to $1$ and $\varepsilon'\ll 1-\theta$. For any $T>0$, we also define a space $\tilde{\mathcal{G}}_T$ of functions in $\mathbb{R}$ with norm 
		$$
		\|g\|_{\tilde {\mathcal{G}}_T}=\|K(t,\cdot)\ast g\|_{\mathcal{G}_{T}}.
		$$
		We denote $\tilde {\mathcal{G}}=\tilde {\mathcal{G}}_{+\infty}$ for simplicity. We prove that 
	\begin{equation*}
		\tilde {\mathcal{G}}=\dot B^{0}_{\infty,\infty},
	\end{equation*}
in Lemma \ref{lemequiv}. We say $h\in \mathcal{G}_T^1$, $g\in \tilde{\mathcal{G}}_T^1$ if $h'\in \mathcal{G}_T$, $g'\in \tilde{\mathcal{G}}_T$
		respectively.
	We note that $\dot B^{0}_{\infty,\infty}$ is a critical space of the toy model \eqref{toymodel}, and 
		$\dot B^{1}_{\infty,\infty}$ is a critical space of the Peskin problem \eqref{peskin}.
		
		The known results of the Peskin problem are established under the so-called well-stretched assumption, which means  that 
		\begin{equation}\label{kap}
			\kappa(X_0)=\sup_{s_1\neq s_2}\frac{|s_1-s_2|}{|X_0(s_1)-X_0(s_2)|}<+\infty,
		\end{equation}
		where $|s_1-s_2|=\inf_{k\in \mathbb{Z}}|s_1-s_2-2k\pi|$ is the distance between $s_1$ and $s_2$ on the torus.
		In critical spaces, it is most difficult to prove the propogation of the well-stretched condition. To overcome this, we introduce a quantity 
		\begin{align}\label{defQ}	Q_h(T)=\sup_{t\in[0,T]}\sup_{\alpha,s\in(-\pi,\pi)}\left(\frac{|\alpha|^{\varepsilon'}}{|t|^{\varepsilon'}}\left|\frac{1}{|\Delta_\alpha h(t,s)|}-\frac{1}{|\Delta_\alpha h(0,s)|}\right|\right),
		\end{align}
		where  $\Delta_\alpha h$ is a slope defined in \eqref{Deltadef} and $\varepsilon'$ is a small positive constant. In fact, if we have $Q_X(T)$ finite and $X(T)\in C^{1+\varepsilon_0}$ at time $T$ for $\varepsilon_0>\varepsilon'$, then $X$ satisfies the well-stretched condition at time $T$ (see Lemma \ref{lemkappa}).\medskip\\
		
		We organize the paper as follows: In the remaining part of this section, we reformulate the problem and state the main results of the paper. In Section 2 we introduce some preliminary lemmas. We establish the regularity theory for the nonlinear parabolic equation in Section 3. Applying the results in Section 3, we estimate the nonlinear terms in Section 4. Finally, we finish the proof of the main theorems in Section 5.
		\subsection{Formulation}
		To simplify the notation, we suppress the time variable and denote 
		\begin{align}\label{Deltadef}
			& \Delta_\alpha X(s)=\frac{\delta_\alpha X(s)}{\alpha}, \quad\quad	\tilde\Delta_\alpha X(s)=\frac{\delta_\alpha X(s)}{\tilde\alpha}
			,\\&
			E^\alpha X(s)=X'(s-\alpha)-\tilde\Delta_\alpha X(s),\nonumber
		\end{align}
		where  $\delta_\alpha X(s)=X(s)-X(s-\alpha)$ and  $\tilde\alpha=\left(\frac{1}{2}\cot\left(\frac{\alpha}{2}\right)\right)^{-1}$.  Note that 
		$$
		\frac{1}{2}\cot\left(\frac{\alpha}{2}\right)=\frac{1}{\alpha}+\sum_{n=1}^\infty \left(\frac{1}{\alpha+2n\pi}+\frac{1}{\alpha-2n\pi}\right).
		$$
		Hence for any periodic function $f:\mathbb{T}\rightarrow\mathbb{R}$, there holds
		\begin{align}\label{TR}
			\int_{\mathbb{T}}f(\alpha)\frac{d\alpha}{\tilde \alpha}=\int _\mathbb{R} f(\alpha) \frac{d\alpha}{\alpha}.
		\end{align}
		The Hilbert transform of $f$ is defined as 
		\begin{align*}
			\mathcal{H} f (s)=\frac{1}{2\pi}\int_{\mathbb{T}} \cot \left(\frac{\alpha}{2}\right)f(s-\alpha)d\alpha= \frac{1}{\pi}\int_{\mathbb{R}} f(s-\alpha)\frac{d\alpha}{\alpha}.
		\end{align*}
		We  introduce the fractional Laplacian operator $\Lambda$ defined by $$
		\Lambda f(s)=\mathcal{H} f'(s)=\frac{1}{\pi}\int_{\mathbb{T}}\frac{\delta_\alpha f(s)}{4\sin^2(\alpha/2)}d\alpha.$$
		It is easy to check that $\widehat{\Lambda f}(\xi)=|\xi|\hat{f}(\xi)$.
		For any $\sigma\in(0,1)$, we also define the operator $\Lambda^\sigma$ by
		\begin{equation}\label{deffraclap}
			\begin{aligned}
				\widehat{\Lambda^\sigma f}(\xi)=|\xi|^\sigma\hat{f}(\xi).
			\end{aligned}
		\end{equation}
		There holds $\Lambda^\sigma f=C_\sigma \int_\mathbb{R}\frac{\delta_\alpha f}{|\alpha|^{1+\sigma}}d\alpha$. 
		By a change of variable and integration by parts in \eqref{peskin} we get 
		$$
		\partial_t X(s)=\int_{\mathbb{T}} \partial_\alpha \mathbf{G}(\delta_\alpha X(s))X'(s-\alpha )d\alpha=-\int_{\mathbb{T}} \partial_\alpha \mathbf{G}(\delta_\alpha X(s))\delta_\alpha X'(s)d\alpha,
		$$
		where we used the fact that $\int \partial_\alpha \mathbf{G}(\delta_\alpha X(s))d\alpha=0$. Further computation leads to
		\begin{align*}
			\partial_t X(s)
			&=\frac{1}{4\pi}\int_{\mathbb{T}}  \frac{\delta_\alpha X(s)\cdot X'(s-\alpha)}{|\delta_\alpha X(s)|^2}\delta_\alpha X'(s)d\alpha\\
			&\quad\quad\quad-\frac{1}{4\pi}\int_{\mathbb{T}}\frac{ X'(s-\alpha)\otimes \delta_\alpha X(s)+\delta_\alpha X(s)\otimes X'(s-\alpha)}{|\delta_\alpha X(s)|^2} \delta_\alpha X'(s)d\alpha\\
			&\quad\quad\quad+\frac{1}{2\pi}\int_{\mathbb{T}}\frac{\delta_\alpha X(s)\otimes \delta_\alpha X(s)}{|\delta_\alpha X(s)|^4}(\delta_\alpha X(s)\cdot X'(s-\alpha)) \delta_\alpha X'(s)d\alpha.
		\end{align*}
		Note that when $|\alpha|\ll 1$, one has 
		$$
		\frac{\delta_\alpha X(s)\cdot X'(s-\alpha)}{|\delta_\alpha X(s)|^2}\sim \frac{1}{2}\cot\left(\frac{\alpha}{2}\right).$$
		This motivates us to extract a Hilbert transform from the first term and use cancellations between the second and the last term. More precisely, one has the formula
		\begin{align}\label{eqpeskin}
			\partial_t X(s)+\frac{1}{4}\Lambda X(s)=N(X(s)),
		\end{align}
		where
		\begin{align*}
			N(X)
			&=\frac{1}{4\pi}\int_{\mathbb{R}}  \frac{\tilde\Delta_\alpha X\cdot E^\alpha X}{|\tilde\Delta_\alpha X|^2}\delta_\alpha X'\frac{d\alpha}{\alpha}-\frac{1}{4\pi}\int_{\mathbb{R}}\frac{ E^\alpha X\otimes \tilde\Delta_\alpha X+\tilde\Delta_\alpha X\otimes E^\alpha X}{|\tilde\Delta_\alpha X|^2} \delta_\alpha X'\frac{d\alpha}{\alpha}\\
			&\quad\quad\quad\quad+\frac{1}{2\pi}\int_{\mathbb{R}} \frac{\tilde\Delta_\alpha X\otimes \tilde\Delta_\alpha X}{|\tilde\Delta_\alpha X|^4}\left(\tilde\Delta_\alpha X\cdot E^\alpha X\right) \delta_\alpha X'\frac{d\alpha}{\alpha}.
		\end{align*}
		We also used \eqref{TR} to transfer the integral on $\mathbb{T}$ to $\mathbb{R}$.
		Note that without specified, all the integrals in the rest of the paper should be understood as  principal value integrals
		over $\mathbb{R}$. 	For simplicity in later estimates, we write the nonlinear terms as
		\begin{equation}\label{Ne}
			N(X(s))=\sum\int H(\tilde\Delta_\alpha X(s)) E^\alpha X_i(s)  \delta_\alpha X_j'(s)\frac{d\alpha}{\alpha}.
		\end{equation}
		where the sum is for some $i,j=1,2$ and $H(x)=\frac{x_{i_1}x_{i_2}x_{i_3}}{|x|^4}$, $i_1,i_2,i_3=1,2$. Moreover, it is easy to check that 
		\begin{equation}\label{integral0}
			\sum\int_{\mathbb{R}}H(\tilde\Delta_\alpha X(s)) E^\alpha X_i(s)\frac{d\alpha}{\alpha}=\int_{-\pi}^\pi \left(\partial_\alpha \mathbf{G}(\delta_\alpha X(s))-\frac{1}{2}\cot \frac{\alpha }{2}\right)d\alpha=0.
		\end{equation}
		We fix two constants in our proof 
		\begin{align*}
			\theta=1-10^{-10^{10}} ,\quad\quad\quad\varepsilon'=10^{-10}(1-\theta).
		\end{align*}
		We also introduce some notations that will be used throughout the paper. We use the notation $a\lesssim b$, which means that there exists an absolute constant $C>0$ such that $a\leq Cb$. With a slight abuse of notation, the value of the absolute constant $C$ may be different from line to line. The mixed norm $\|\cdot\|_{L^p_TL^q}$ means first take $L^q$ norm in space variable $x\in\mathbb{R}$ and then take $L^p$ norm in time variable $t\in[0,T].$
		\subsection{Formulation near the steady state}\label{Formulation near the steady state}
		%
		It is easy to see that the Peskin problem has translation, rotation and dilation invariance. Moreover,  the only stationary mild solutions of the Peskin problem are circles in which the material points are evenly spaced\cite{LinTongSolvability2019,MoriWell2019}:
		$$
		Z(s)=A e_r+Be_t+C_1e_x+C_2e_y,~~~A^2+B^2>0,
		$$
		where 
		$$
		e_{\mathrm{r}}=\left(\begin{array}{c}
			\cos (s) \\
			\sin (s)
		\end{array}\right),  {e}_{\mathrm{t}}=\left(\begin{array}{c}
			-\sin (s) \\
			\cos (s)
		\end{array}\right),  {e}_{x}=\left(\begin{array}{l}
			1 \\
			0
		\end{array}\right),  {e}_{y}=\left(\begin{array}{l}
			0 \\
			1
		\end{array}\right).
		$$
		For later reference, we denote $ \tilde{\mathcal{V}}$ the above set of circular equilibria and $\mathcal{V}$ the linear space spanned by the above 4 basis vectors.
		To state our results, we first introduce some notations. For $U,W\in L^2(\mathbb{T},\mathbb{R}^2)$, we define the standard $L^2$ inner product as: 
		$$
		\langle U,W\rangle:=\int_{\mathbb{T}} U(s)\cdot W(s)ds.$$
		Let $\mathcal{P}$ be the $L^2$ projection on to the space $\mathcal{V}$ and $\Pi$ its complementary projection:
		$$
		\mathcal{P}  {X}=\frac{1}{2 \pi} \sum_{\ell=\mathrm{r}, \mathrm{t}, x, y}\left\langle {X},  {e}_{\ell}\right\rangle  {e}_{\ell},\quad\quad \Pi  {X}= {X}-\mathcal{P}  {X}.$$
		We linearize the equation around stationary solutions. The linearized operator of the equation \eqref{eqpeskin} at $Z\in\tilde{\mathcal{V}}$ is given by
		\begin{align}\label{deflinearize}
			\mathcal{L}_ZX=\left.\frac{d}{d\epsilon}\left(\frac{1}{4}\Lambda(Z+\epsilon X)-N(Z+\epsilon X)\right)\right|_{\epsilon=0}=\frac{1}{4}\Lambda X-\mathfrak{D} N (Z)X,
		\end{align}where we denote $\mathfrak{D} N (Z)X=\left.\frac{d}{d\epsilon}N(Z+\epsilon X)\right|_{\epsilon=0}$. 
		It is easy to check that the linearized operator has translation and dilation invariance. Moreover, denote $\mathcal{O}_s=\left(\begin{array}{cc}
			\cos (s) & \sin (s) \\
			-\sin (s) & \cos (s)
		\end{array}\right)$. Let $\bar e_r= \mathcal{O}_{s_0}e_r$, one has 
		$$
		\mathcal{L}_{\bar e_r}=\mathcal{O}_{s_0}	\mathcal{L}_{e_r}\mathcal{O}_{s_0}^T.
		$$
		For simplicity, denote $\mathcal{L}=\mathcal{L}_{e_r}$.
		We can check that 
		$$
		\mathcal{L}  {w}=\frac{1}{4}\Lambda  {w}+\frac{1}{4}\left(\begin{array}{cc}
			0 & -\mathcal{H} \\
			\mathcal{H} & 0
		\end{array}\right)  {w}.$$
		Note that $\mathcal{O}_{s_0}	\mathcal{L}\mathcal{O}_{s_0}^T=\mathcal{L}$. Hence the linearized operator has rotation, translation and dilation invariance. More precisely, there holds
		$$
		\mathcal{L}_Z=\mathcal{L}_{e_r}=\mathcal{L}, ~~~~\text{for any} ~~Z\in \tilde{\mathcal{V}}.$$
		Consider the equation
		$$
		\partial_{t}w+	\frac{1}{4}\Lambda  {w}+\frac{1}{4}\left(\begin{array}{cc}
			0 & -\mathcal{H} \\
			\mathcal{H} & 0
		\end{array}\right)  {w} =F\in \mathbb{R}^2.$$
		We can write the equation in terms of Fourier series
		\begin{align*}
			&\partial_{t}\hat w_{1,n}+\frac{1}{4}|n| \hat w_{1,n}-\frac{1}{4}i\operatorname{sgn}(n)\hat w_{2,n}=\hat F_{1,n}, \\	&\partial_{t}\hat w_{2,n}+\frac{1}{4}|n| \hat w_{2,n}+\frac{1}{4}i\operatorname{sgn}(n)\hat w_{1,n}=\hat F_{2,n}.
		\end{align*}
		Let $v(s)= \mathcal{O}_s\Pi w(s)$. Then for any $n\in\mathbb{Z}\backslash\{0\}$ one has
		\begin{align*}
			&\hat v_{1,n}=\frac{\hat w_{1,n-1}+\hat w_{1,n+1}}{2}+\frac{\hat w_{2,n-1}-\hat w_{2,n+1}}{2i},\\
			&\hat v_{2,n}=\frac{\hat w_{1,n+1}-\hat w_{1,n-1}}{2i}+\frac{\hat w_{2,n-1}+\hat w_{2,n+1}}{2}.
		\end{align*}
		It is easy to check that 
		\begin{align*}
			\partial_t \hat v_{1,n}+\frac{1}{4}|n|\hat v_{1,n}=\frac{1}{2}(\hat F_{1,n-1}+\hat F_{1,n+1})+\frac{1}{2i}(\hat F_{2,n-1}-\hat F_{2,n+1}),\\
			\partial_t \hat v_{2,n}+\frac{1}{4}|n|\hat v_{2,n}=\frac{1}{2i}(\hat F_{1,n+1}-\hat F_{1,n-1})+\frac{1}{2}(\hat F_{2,n-1}+\hat F_{2,n+1}),
		\end{align*}
		which is equivalent to 
		\begin{align}\label{kkkkkernel}
			\partial_t v+\frac{1}{4}\Lambda v=\mathcal{O}_s \Pi F.
		\end{align} 
		Hence we obtain
		$$
		\mathcal{L}=\frac{1}{4}\mathcal{O}_s^{-1}\Lambda \mathcal{O}_s\Pi=\frac{1}{4}\Pi \mathcal{O}_s^{-1}\Lambda \mathcal{O}_s.
		$$
		From above we directly obtain
		\begin{align}\label{proptyL}
			\mathcal{P}\mathcal{L}=0,\quad\text{and}\quad\Pi \mathcal{L}=\mathcal{L}.
		\end{align}
		We can rewrite the Peskin equation \eqref{eqpeskin} as 
		\begin{align}\label{eqlinearize}
			\partial_t X +\mathcal{L} X= \mathfrak{N}(X),\quad  \mathfrak{N}(X) =N(X)+\mathcal{L} X- \frac{1}{4}\Lambda X.
		\end{align}
		For any stationary solution $W\in \tilde{\mathcal{V}}$, one has 
		\begin{align}\label{haha1}
			0=-\mathcal{L}W+\mathfrak{N}(W)=\mathfrak{N}(W).
		\end{align}
		Moreover, 
		by the definition \eqref{deflinearize} and \eqref{eqlinearize} we have for any $W\in \tilde{\mathcal{V}}$ and any $U$
		\begin{align}\label{haha2}
			\mathfrak{DN}[ {W}]  {U}=\mathfrak{D} N[ {W}]  {U}+\mathcal{L}  {U}-\frac{1}{4}\Lambda  {U}=0.
		\end{align}
		Let $X$ be a solution of \eqref{eqlinearize}.	Denote $Y=\Pi X$ and $Z=\mathcal{P}X$, then \eqref{proptyL} leads to 
		\begin{equation}\label{eqglo}
			\begin{aligned}
				\partial_t Y+\mathcal{L}Y&= \Pi\mathfrak{N}(Y+Z),\\
				\partial_t Z&=\mathcal{P} \mathfrak{N}(Y+Z).
			\end{aligned}
		\end{equation}
		
		\subsection{Main results}
		For any vector valued function $f(t,s)$, denote $$\kappa_f(t)=\sup_{\tau\in(0,t)}\kappa(f(\tau,\cdot)),$$ where $\kappa$ is defined in \eqref{kap}.  For simplicity, let  $
		\kappa_0=\liminf_{\vartheta\rightarrow 0}\kappa(X_0\ast\rho_\vartheta)
		$, where $\rho_\vartheta$ is the standard mollifier. We also denote $\kappa(t)=\kappa_X(t)$.
		We state the main results as follows.
	\begin{theorem}\label{thmlocal}(Local existence) For any $r>0$, there exists $\xi_0=\xi_0(r)>0$ such that for
	 any initial data $X_0\in L^\infty$ with $\kappa_0\leq r$, if $\|X_0'\|_{\tilde {\mathcal{G}}_{T^*}}\leq \xi_0$ for some $T^*\in(0,1)$, then  the Cauchy problem of \eqref{peskin} has a solution $X\in C([0,T^*];L^\infty)$ satisfying 
		\begin{align*}
			\|X'\|_{\mathcal{G}_{T^*}}\leq2 \xi_0,\quad\quad \kappa(T^*)\leq 2\kappa_0.
		\end{align*}
		Moreover, we have $\sup_{0<t\leq T^*} t^k\|X'(t)\|_{\dot C^k}\leq C_k\xi_0$ for any $k\in\mathbb{Z}^+$.
	\end{theorem}
Thanks to the above theorems and  Lemma \ref{propparabolic}, we deduce immediately the following local well-posedness results with  $(C^2)^{\dot B^1_{\infty,\infty}}$ initial data. Here  we denote  $(C^2)^{\dot B^1_{\infty,\infty}}$ as the closure of $C^2$ in $\dot B^1_{\infty,\infty}$. 
\begin{corollary}\label{corvmo}
	For any initial data $X_0\in  (C^2)^{\dot B^1_{\infty,\infty}}\cap L^\infty$ satisfying $\kappa_0<\infty$, and any $\xi_0\ll 1$, there exists $T^*>0$ such that the Cauchy problem of \eqref{peskin} has a solution $X\in C([0,T^*];L^\infty)$ satisfying 
	\begin{align*}
		\|X'\|_{\mathcal{G}_{T^*}}\leq2 \xi_0,\quad\quad \kappa(T^*)\leq 2\kappa_0.
	\end{align*}
	Moreover, we have $\sup_{0<t\leq T^*} t^k\|X'(t)\|_{\dot C^k}\leq C_k\xi_0$ for any $k\in\mathbb{Z}^+$.
\end{corollary}
Following is global existence of \eqref{peskin} in $\dot B^1_{\infty,\infty}=\tilde {\mathcal{G}}^1$.
		\begin{theorem}\label{thmglobal}(Global existence)
			For any $c_0\in(0,1)$, there exists $\xi_1>0$ such that if the initial data $X_0=Y_0+Z_0\in L^\infty$ satisfies $Z_0\in \mathcal{V}$,  $\|Z_0'\|_{L^\infty}\in [ c_0,c_0^{-1}]$; $\liminf_{\vartheta\rightarrow 0}\kappa(Y_0\ast\rho_\vartheta+Z_0)\leq c_0^{-1}$ and $\|Y_0'\|_{\tilde{\mathcal{G}}}\leq \xi_1$, then
			the Cauchy problem of \eqref{peskin} has a solution $X=Y+Z\in C([0,+\infty);L^\infty)$ satisfying for some $T_0>0$\\
			\textbf{1)}	
			$$
			\|Y'\|_{\mathcal{G}_{T_0}}\leq 2\xi_1,\quad\quad \sup_{\tau\in[0,T_0]}\left|\|Z' (\tau)\|_{L^\infty}-\|Z_0'\|_{L^\infty}\right|\leq \xi_1,\quad\quad \kappa(T_0)\leq 4c_0^{-1}.
			$$
			\textbf{2)}$$
			\sup_{0<t\leq T_0} t^k\|X'(t)\|_{\dot C^{k}}\leq C_k,\ \forall k\in\mathbb{Z}^+,~~~~\quad\quad\quad\quad\quad~~~~	\|Y(T_0)\|_{\dot C^\frac{3}{2}}\leq \xi_1^\frac{1}{4}.
			$$
			\textbf{3)} There exists a circle $Z_\infty\in \tilde{\mathcal{V}}$ such that for $t\geq T_0$, there holds
			$$
			\|Y(t)\|_{\dot C^\frac{3}{2}}\leq C\xi_1^\frac{1}{4}e^{-\frac{t}{4}}, \quad\quad\quad
			\|X(t)-Z_\infty\|_{\dot C^\frac{3}{2}}\leq C\xi_1^\frac{1}{4}e^{-\frac{t}{4}}.
			$$
		\end{theorem}
		\begin{proposition}(Uniqueness)\label{Uniqueness}
			Let $X,\bar X\in \mathcal{G}_T^1$ be solutions of equation \eqref{eqpeskin} on $[0,T]$ with $\kappa_X(T)+\kappa_{\bar X}(T)<+\infty$. If 
			$$
			\left(1+\kappa_X(T)+\kappa_{\bar X}(T)\right)^{2}(\|X'\|_{\mathcal{G}_T}+\|\bar X'\|_{\mathcal{G}_T})\ll 1 ,
			$$
			then there holds 
			$$
			\|X'-\bar X'\|_{\mathcal{G}_T}\lesssim\|X_0'-\bar X'_0\|_{\tilde{\mathcal{G}}_T}.$$
		\end{proposition}
		The above proposition implies the uniqueness of the solution in Theorem \ref{thmlocal}, Corollary \ref{corvmo} and Theorem \ref{thmglobal}.\\ 
		Our result is related to the work of Koch and Tartaru \cite{KochTar} about the well-posedness of Navier-Stokes in $BMO^{-1}$. They proved that for initial data $u_0$ with $\|u_0\|_{BMO^{-1}}\ll 1$,  the Navier-Stokes equation 
		\begin{equation}\label{NS}
				\begin{aligned}
				\partial_t u-\Delta u+u\cdot\nabla u+\nabla p=0,\\
				\operatorname{div} u=0,
			\end{aligned}
		\end{equation}
		has a unique solution $u$ in $\mathcal{X}$ so that 
		$$
		\|u\|_{\mathcal{X}}\lesssim  \|u_0\|_{BMO^{-1}}.
		$$
		Here the space $\mathcal{X}$ is equipped with the norm 
		$$
		\|u\|_{\mathcal{X}} {:=} \sup _{t>0} \sqrt{t}\|u(t,\cdot)\|_{L^{\infty}}+\sup _{x, R}\left( \int_{0}^{R} \fint_{B(x, \sqrt{R})}|u(t,y)|^{2} d y d t\right)^{\frac{1}{2}},$$
		where $\fint_{B(x, \sqrt{R})}=|B(x, \sqrt{R})|^{-1}\int_{B(x, \sqrt{R})}.$
		Koch and Tartaru \cite{KochTar} used the following characterization of the $BMO^{-1}$ norm (see also \cite{Stein,Ping}):
		\begin{align*}\left\|u_{0}\right\|_{BMO^{-1}} \sim\sup_{t>0}\sqrt{t}\|e^{t\Delta }u_0\|_{L^\infty}+\sup _{x, R}\left(\int_{0}^{R} \fint_{B(x, \sqrt{R})}\left|e^{t \Delta} u_{0}(y)\right|^{2} d y d t\right)^{\frac{1}{2}}.
		\end{align*}
	Moreover, the problem \eqref{NS} is strongly ill-posed in $\dot B^{-1}_{\infty,\infty}$, proved by J. Bourgain and  N. Pavlovic \cite{Bour}. In Theorem \ref{thmlocal} and Theorem \ref{thmglobal}, we prove that the Peskin problem is well-posed in $\dot B^{1}_{\infty,\infty}$.

		\section{Preliminaries}
		We denote $\dot C^k, k=0,1,2,\cdots$  the space of functions with $k$-th continuous derivative. Let $\gamma \in (0,1)$. A function $h\in \dot C^{k}$ is in the H\"{o}lder space $\dot C^{k+\gamma}$ if 
		$$\|h\|_{\dot C^{k+\gamma}}=
		\sup_{s\neq s'}\frac{|h^{(k)}(s)-h^{(k)}(s')|}{|s-s'|^\gamma}<\infty.
		$$
		We introduce the following H\"{o}lder estimates for periodic functions.
		\begin{lemma}\label{holder}	For any function $f:\mathbb{R}\rightarrow \mathbb{R}$, if $f$ is $2\pi$-periodic, there holds\\
			\textbf{1)} For any $0<l_1<l_2$, \begin{align*}
				\|f\|_{\dot C^{l_1}}\lesssim \|f\|_{\dot C^{l_2}}.  
			\end{align*}
			\textbf{2)} For any $0<\gamma<1$,\begin{align*}
				\sup_{\alpha,s}\frac{|\delta_\alpha f(s)|}{|\alpha|^\gamma}+	\sup_{\alpha,s}\frac{|\delta_\alpha f(s)|}{|\tilde\alpha|^\gamma}\lesssim \|f\|_{\dot C^\gamma}.
			\end{align*}
			\textbf{3)}	For any $0<\gamma<1$,\begin{align*}
				\sup_{\alpha,s}\frac{|E^\alpha f(s)|}{|\alpha|^\gamma}+	\sup_{\alpha,s}\frac{|E^\alpha f(s)|}{|\tilde\alpha|^\gamma}\lesssim \|f\|_{\dot C^{1+\gamma}}.
			\end{align*}
		\end{lemma}
		\begin{proof}
			\textbf{1)} For any $\alpha\in(-\pi,\pi)$ and $\gamma\in(0,1)$, one has  $\sup_{k\in\mathbb{Z}}\frac{1}{|\alpha+2k\pi|^\gamma}\leq \frac{1}{|\alpha|^\gamma}$. The function $f$ is periodic, hence
			it is easy to check that
			\begin{align*}
				\sup_{\alpha,s}\frac{|\delta_\alpha f(s)|}{|\alpha|^\gamma}=\sup_{\alpha,s\in(-\pi,\pi)}\frac{|\delta_\alpha f(s)|}{|\alpha|^\gamma}.
			\end{align*}
			Hence for $0<\gamma_1\leq\gamma_2\leq 1$,
			$$
			\|f\|_{\dot C^{\gamma_1}}\lesssim \|f\|_{\dot C^{\gamma_2}}.
			$$
			Furthermore, for any $s\in\mathbb{R}$ there exists $s_0$ such that $f'(s_0)=0$, $|s-s_0|\leq \pi$. Hence for any $\gamma\in(0,1)$
			$$
			|f'(s)|=|f'(s)-f'(s_0)|\lesssim \|f'\|_{\dot C^{\gamma}}.
			$$
			Hence $\|f'\|_{L^\infty}\lesssim \|f'\|_{\dot C^\gamma}$. We repeat the above procedure with $f$ replaced by $f', f'',\cdots$, we obtain \textbf{1)}.\\ 
			\textbf{2)}	Observe that $\sup_{\alpha\in(-\pi,\pi)}\frac{\alpha}{\tilde\alpha}\lesssim 1.$ Hence 
			\begin{align*}
				\sup_{\alpha,s}\frac{|\delta_\alpha f(s)|}{|\tilde\alpha|^\gamma}\lesssim \sup_{\alpha,s\in(-\pi,\pi)}\frac{|\delta_\alpha f(s)|}{|\tilde\alpha|^\gamma}\lesssim \|f\|_{\dot C^\gamma}.
			\end{align*}
			\textbf{3)} 	Recall the definition \eqref{Deltadef}, one has 
			\begin{align*}
				\sup_{\alpha,s}	\frac{|E^\alpha f(s)|}{|\alpha|^\gamma}\lesssim	\sup_{\alpha,s} \frac{1}{|\alpha|^\gamma}\left|f'(s-\alpha)-\frac{\delta_\alpha f(s)}{\alpha}\right|+	\sup_{\alpha,s}\frac{|\delta_\alpha f(s)|}{|\alpha|^\gamma}\lesssim \|f\|_{\dot C^{1+\gamma}},
			\end{align*}
			where we also used \textbf{1)}. Similarly we have 
			\begin{align*}
				\sup_{\alpha,s}\frac{|E^\alpha f(s)|}{|\tilde\alpha|^\gamma}\lesssim \|f\|_{\dot C^{1+\gamma}}.
			\end{align*}
			The proof is complete.
		\end{proof}
		\begin{lemma} Let $\theta_1\in(0,1)$. For any function $f:\mathbb{R}\rightarrow\mathbb{R}$ and any $0<\varepsilon_0<\frac{1}{2}\min\{\theta_1,1-\theta_1\}$, there hold
			\begin{equation}\label{interpfrac}
				||\Lambda^{\theta_1}f||_{L^\infty}\lesssim (||f||_{\dot C^{\theta_1-\varepsilon_0}}||f||_{\dot C^{\theta_1+\varepsilon_0}})^{\frac{1}{2}},
			\end{equation}	
			\begin{equation}\label{interpfrac2}
				||f||_{\dot C^{\theta_1}}\lesssim ||\Lambda^{\theta_1}f||_{L^\infty}.
			\end{equation}	

		\end{lemma}
		\begin{proof} 
			Recall the definition of the fractional Laplacian \eqref{deffraclap}, we have for any $\lambda_1>0$
$$
				|\Lambda^{\theta_1}_1f(s)|\lesssim\left(\int_{|z|\leq \lambda_1}+\int_{|z|\geq \lambda_1}\right)|\delta_zf(s)|\frac{dz}{|z|^{1+{\theta_1}}} |\lambda_1|^{\varepsilon_0}||f||_{\dot C^{\theta_1+\varepsilon_0}}+|\lambda_1|^{-\varepsilon_0}||f||_{\dot C^{\theta_1-\varepsilon_0}}.$$
			Choosing $\lambda_1=\left(||f||_{\dot C^{\theta_1+\varepsilon_0}}^{-1}||f||_{\dot C^{\theta_1-\varepsilon_0}}\right)^\frac{1}{2\varepsilon_0}$ we get \eqref{interpfrac}.\\
			To prove \eqref{interpfrac2}, we only need to prove $\|\Lambda^{-\theta_1}g\|_{\dot C^{\theta_1}}\lesssim \|g\|_{L^\infty}$. Observe that
$$\left|\delta_\alpha\Lambda^{-\theta_1}g(x)\right|\lesssim \left|\int g(y)\left(\delta_\alpha\frac{1}{|\cdot|^{1-{\theta_1}}}\right)(x-y)dy\right|\lesssim \|g\|_{L^\infty}\int \left|\frac{1}{|y|^{1-{\theta_1}}}-\frac{1}{|y-\alpha|^{1-{\theta_1}}}\right|dy \lesssim\|g\|_{L^\infty}|\alpha|^{\theta_1}.$$
			Hence we get \eqref{interpfrac2}.
		\end{proof}
		\begin{lemma}\label{lem2}
			For any function $f,g:\mathbb{R}^+\times{\mathbb{R}}\rightarrow\mathbb{R}$, $\gamma\in(0,1)$ and $p\in(1,+\infty)$, denote $\tilde f(t,\alpha)=\tilde\Delta_\alpha f(t,0)$, if $f$ is $2\mathbb{\pi}$-periodic in space, there hold\\
			\textbf{1)}	
			\begin{align*}
				\sup_{\alpha,y}\frac{\|\tilde f(\alpha)-\tilde f(y)\|_{L^p_T}}{|\alpha-y|^\gamma}\lesssim \sup_\alpha \frac{\|\delta_\alpha f'\|_{L^p_TL^\infty}}{|\alpha|^\gamma},
			\end{align*}
			\textbf{2)}	Let $\frac{1}{p}=\frac{1}{p_1}+\frac{1}{p_2}$ and $0<\gamma'<\gamma$, then
			\begin{align*}
				\sup_{\alpha,y}\frac{\left\| |\delta_\alpha g'(0)||\tilde f(\alpha)-\tilde f(y) |\right\|_{L^p_T}}{|\alpha-y|^\gamma}\lesssim\left(\sup_\alpha \frac{\|\delta_\alpha f'\|_{L^{p_1}_TL^\infty}}{|\alpha|^{\gamma'}}\right)\left(\sup_\alpha \frac{\|\delta_\alpha g'\|_{L^{p_2}_TL^\infty}}{|\alpha|^{\gamma-\gamma'}}\right),
			\end{align*}
			where $|\alpha-y|=\inf_{k\in \mathbb{Z}}|\alpha-y-2k\pi|$ denotes the distance of $\alpha$ and $y$ on $\mathbb{T}$.
		\end{lemma}
		\begin{proof}
			\textbf{1)}\quad	
			Without loss of generality, assume
			$|\alpha|<\pi$ and $|\alpha|\leq|y|$, we have
			\begin{align*}
				\tilde f(\alpha)-\tilde f(y)=&\left(\frac{1}{2}\cot\left(\frac{y}{2}\right)-\frac{1}{y}\right)\int_0^{-y}f'(s)ds-\left(\frac{1}{2}\cot\left(\frac{\alpha}{2}\right)-\frac{1}{\alpha}\right)\int_0^{-\alpha}f'(s)ds\\
				&+\frac{1}{y}\int_0^{-y}f'(s)ds-\frac{1}{\alpha}\int_0^{-\alpha}f'(s)ds.
			\end{align*}
			Hence by Minkowski's inequality we obtain
			\begin{align*}
				\|\tilde f(\alpha)-\tilde f(y)\|_{L^p_T}\lesssim& \left|\frac{1}{2}\cot\left(\frac{\alpha}{2}\right)-\frac{1}{\alpha}\right|\left|\int_{-\alpha}^{-y}\|f'(s)\|_{L^p_T}ds\right|\\
				&\quad\quad+\left|\frac{1}{2}\cot\left(\frac{\alpha}{2}\right)-\frac{1}{\alpha}-\frac{1}{2}\cot\left(\frac{y}{2}\right)+\frac{1}{y}\right|\left|\int_{0}^{-y}\|f'(s)\|_{L^p_T}ds\right|\\
				&\quad\quad+\frac{1}{|y|}\left|\int_{-\alpha}^{-y}\|f'(s)-f'(-\alpha)\|_{L^p_T}ds\right|+\left|\int_0^{-\alpha}\|f'(s)-f'(-\alpha)\|_{L^p_T}ds\right|\frac{|\alpha-y|}{|\alpha||y|}.
			\end{align*}	
			Note that $\left|\frac{1}{2}\cot\left(\frac{\alpha}{2}\right)-\frac{1}{\alpha}\right|\lesssim 1$ and $\left|\frac{1}{2}\cot\left(\frac{\alpha}{2}\right)-\frac{1}{\alpha}-\frac{1}{2}\cot\left(\frac{y}{2}\right)+\frac{1}{y}\right|\lesssim |\alpha-y|$. Moreover, because $f$ is periodic, there holds
			$$
			\|f'(s)\|_{L^p_T}\lesssim \sup_z\frac{\|\delta_z f'\|_{L^p_TL^\infty}}{|z|^\gamma}.$$
			Hence 
			\begin{align*}
				\|\tilde f(\alpha)-\tilde f(y)\|_{L^p_T}\lesssim& \left(|y-\alpha|+\frac{|y-\alpha||y|^\gamma}{|y|}+\frac{|\alpha|^\gamma|y-\alpha|}{|y|}\right)\sup_z\frac{\|\delta_z f'\|_{L^p_TL^\infty}}{|z|^\gamma}\\
				\lesssim&|y-\alpha|^\gamma\sup_z\frac{\|\delta_z f'\|_{L^p_TL^\infty}}{|z|^\gamma}.
			\end{align*}	
			Then we have the result.\\
			\textbf{2)}\quad By \textbf{1)} there holds 
			$$
			\|\tilde f(\alpha)-\tilde f(y)\|_{L^{p_1}_T}\lesssim |y-\alpha|(|y|+|\alpha|)^{\gamma'-1}\sup_z\frac{\|\delta_z f'\|_{L^{p_1}_TL^\infty}}{|z|^{\gamma'}}.
			$$
			Hence by H\"{o}lder's inequality we obtain
			\begin{align*}
				&	\left\| |\delta_\alpha g'(0)||\tilde f(\alpha)-\tilde f(y) |\right\|_{L^p_T}\\&\lesssim |\alpha|^{\gamma-\gamma'}|y-\alpha|(|y|+|\alpha|)^{{\gamma'-1}}\left(\sup_\alpha \frac{\|\delta_\alpha f'\|_{L^{p_1}_TL^\infty}}{|\alpha|^{\gamma'}}\right)\left(\sup_\alpha \frac{\|\delta_\alpha g'\|_{L^{p_2}_TL^\infty}}{|\alpha|^{\gamma-\gamma'}}\right)\\
				&\lesssim |y-\alpha|^\gamma \left(\sup_\alpha \frac{\|\delta_\alpha f'\|_{L^{p_1}_TL^\infty}}{|\alpha|^{\gamma'}}\right)\left(\sup_\alpha \frac{\|\delta_\alpha g'\|_{L^{p_2}_TL^\infty}}{|\alpha|^{\gamma-\gamma'}}\right).
			\end{align*}
			This implies \textbf{2)}	.	
		\end{proof}
		\begin{lemma}\label{lemI12}
			For any function $f,g:{\mathbb{R}}\rightarrow\mathbb{R}$, denote $\tilde g(\alpha)=\Delta_\alpha g(0)$. Then for any $\sigma\in(0,1)$ and $0<\varepsilon< 10^{-3}\min\{1-\sigma,\sigma\}$, there holds
			$$
			\left|\int f(\alpha)(\partial_\alpha \tilde g)(\alpha)d\alpha\right|\lesssim \|\Lambda^\sigma f\|_{L^\infty}\|\Lambda^{1-\sigma+\varepsilon}g\|_{L^\infty}^\frac{1}{2}\|\Lambda^{1-\sigma-\varepsilon}g\|_{L^\infty}^\frac{1}{2}.$$
		\end{lemma}
		We postpone the proof in the appendix. Applying H\"{o}lder's inequality and Young's inequality in the proof of Lemma \ref{lemI12}, we have
		\begin{remark}\label{remmovedeT}
			For any function $f,g:\mathbb{R}^+\times{\mathbb{R}}\rightarrow\mathbb{R}$, denote $\tilde g(\alpha)=\Delta_\alpha g(0)$. Then for any $\sigma\in(0,1)$, $p>1$ and $0<\varepsilon< 10^{-3}\min\{1-\sigma,\sigma\}$, there holds
			\begin{align*}
				&\left\|\int f(\alpha)(\partial_\alpha \tilde g)(\alpha)d\alpha\right\|_{L^p_T}\lesssim\sup_\alpha \int\min_{+,-}\left\{\|\delta_z f(\alpha)\|_{L^{p_{1\pm}}}\left\|\|\Lambda^{1-\sigma+\varepsilon}g\|_{L^\infty}^\frac{1}{2}\|\Lambda^{1-\sigma-\varepsilon}g\|_{L^\infty}^\frac{1}{2}\right\|_{L^{p_{2\mp}}}\right\}\frac{dz}{|z|^{1+\sigma}},
			\end{align*}
			where $p_{1\pm}, p_{2\pm}$ satisfies $\frac{1}{p}=\frac{1}{p_{1+}}+\frac{1}{p_{2-}}=\frac{1}{p_{1-}}+\frac{1}{p_{2+}}$. 
		\end{remark} 
		\begin{lemma}\label{moveabove}
			For any function $g:[0,T]\times\mathbb{R}\rightarrow \mathbb{R}$, and any $b,\sigma\in(0,1), p\in(1,\infty)$, if  $b+\sigma\leq1-\frac{\varepsilon'}{100}$,  there holds
			$$
			\sup_\alpha\frac{\|\delta_\alpha g\|_{L^p_TL^\infty}}{|\alpha|^{b+\sigma}}\lesssim\left(\sup_\alpha\frac{\|\delta_\alpha \Lambda^{\sigma-\varepsilon} g\|_{L^p_TL^\infty}}{|\alpha|^{b+\varepsilon}}\right)^\frac{1}{2}\left(\sup_\alpha\frac{\|\delta_\alpha \Lambda^{\sigma+\varepsilon} g\|_{L^p_TL^\infty}}{|\alpha|^{b-\varepsilon}}\right)^\frac{1}{2},$$
			where $0<\varepsilon<10^{-3}\min\left\{b,\sigma,\varepsilon'\right\}$.
		\end{lemma}
		\begin{proof}
			Denote $g_1=\Lambda^{\sigma-\varepsilon} g$, $g_2=\Lambda^{\sigma+\varepsilon} g$, $l_1=b+\sigma-\varepsilon$, $l_2=b+\sigma+\varepsilon$, then
$$
				g(x)=c\int\delta_z (\Lambda^{-b}g)(x)\frac{dz}{|z|^{1+b}}=c\int_{|z|\leq\lambda}\delta_z(\Lambda^{-l_1}g_1)(x)\frac{dz}{|z|^{1+b}}+c\int_{|z|\geq\lambda}(\delta_z\Lambda^{-l_2}g_2)(x)\frac{dz}{|z|^{1+b}}.$$
			Then 
$$
				\delta_\alpha g(x)=c\int_{|z|\leq\lambda}\int\delta_\alpha\left(\frac{1}{|\cdot|^{1-l_1}}\right)(x-y)\delta_zg_1(y)\frac{dydz}{|z|^{1+b}}+c\int_{|z|\geq\lambda}\int\delta_\alpha\left(\frac{1}{|\cdot|^{1-l_2}}\right)(x-y)\delta_zg_2(y)\frac{dydz}{|z|^{1+b}}.$$
			It is easy to check that
			$$
			\int \left|\delta_\alpha\left(\frac{1}{|\cdot|^{1-l_k}}\right)(y)\right|dy\lesssim |\alpha|^{l_k}, \quad\quad k=1,2.$$
			Hence by Minkowski inequality we obtain
			\begin{align*}
				\|\delta_\alpha g\|_{L^p_TL^\infty}&\lesssim |\alpha |^{l_1}\int_{|z|\leq\lambda}\|\delta_zg_1\|_{L^p_TL^\infty}\frac{dz}{|z|^{1+b}}+|\alpha |^{l_2}\int_{|z|\geq\lambda}\|\delta_zg_2\|_{L^p_TL^\infty}\frac{dz}{|z|^{1+b}}\\
				&\lesssim |\alpha |^{l_1}\lambda^\varepsilon\sup_z\frac{\|\delta_zg_1\|_{L^p_TL^\infty}}{|z|^{b+\varepsilon}}+ |\alpha |^{l_2}\lambda^{-\varepsilon}\sup_z\frac{\|\delta_zg_2\|_{L^p_TL^\infty}}{|z|^{b-\varepsilon}}\\
				&\lesssim |\alpha|^{b+\sigma} \left(\sup_z\frac{\|\delta_z \Lambda^{\sigma-\varepsilon} f\|_{L^p_TL^\infty}}{|z|^{b+\varepsilon}}\right)^\frac{1}{2}\left(\sup_z\frac{\|\delta_z \Lambda^{\sigma+\varepsilon} f\|_{L^p_TL^\infty}}{|z|^{b-\varepsilon}}\right)^\frac{1}{2},
			\end{align*}
			where we take $\lambda=\left(|\alpha |^{l_1}\sup_z\frac{\|\delta_zg_1\|_{L^p_TL^\infty}}{|z|^{b+\varepsilon}}\right)^{-\frac{1}{2\varepsilon}} \left(|\alpha |^{l_2}\sup_z\frac{\|\delta_zg_2\|_{L^p_TL^\infty}}{|z|^{b-\varepsilon}}\right)^{\frac{1}{2\varepsilon}}$. 
			Then we obtain the result.
		\end{proof}	
		
		\begin{lemma}\label{lemRiesz}
			For any function $f:[0,+\infty)\rightarrow\mathbb{R}$, let $B, T>0$, $p\in(1,\infty)$, then for any $r,\sigma\in[0,1]$ such that $\frac{1}{p}<r+\sigma<1$, 
			%
			there holds
			$$
			\left\|\int_0^t\frac{1}{(t-\tau)^r}\min\left\{1,\frac{B}{t-\tau}\right\}f(\tau)d\tau\right\|_{L^p_T}\lesssim B^\sigma\|f\|_{L^q_T}, $$
			where $q=\frac{p}{1+(1-r-\sigma)p}$.
		\end{lemma}
		\begin{proof}
			Note that 
			$
			\min\left\{1,\frac{B}{t-\tau}\right\}\lesssim \left(\frac{B}{t-\tau}\right)^\sigma.
			$
			Then we have $$
			\left|\int_0^t\frac{1}{(t-\tau)^r}\min\left\{1,\frac{B}{t-\tau}\right\}f(\tau)d\tau\right|\lesssim \mathbf{I}_{1-r-\sigma}(f\mathbf{1}_{[0,T]})(t),$$
			where $\mathbf{I}_a$ is the Riesz potential in $\mathbb{R}$ which satisfies $\|\mathbf{I}_a\|_{L^{q_1}\rightarrow L^{q_2}}\lesssim 1$ with $a+\frac{1}{q_2}=\frac{1}{q_1}$ and $q_1\in(1,\frac{1}{a})$.
			Then we get the result.
		\end{proof} 
		\vspace{0.3cm}\\
		Recall the definition of $Q_w(t)$ in \eqref{defQ}.	The following is a key lemma to prove the propogation of the well-streched condition. 
		\begin{lemma}\label{lemkappa}
			For any function $w(t,x):\mathbb{R}^+\times\mathbb{R}\rightarrow\mathbb{R}^2$. Denote $\kappa_w(t)=\sup_{\tau\in(0,t)}\sup_{\alpha,s}\frac{|\alpha|}{|\delta_\alpha w(\tau,s)|}$. If there exists $0<\varepsilon<\frac{1}{100}\min \{\kappa_w(0),\kappa_w(0)^{-1}\}$ such that 
			\begin{align*}
				Q_{w}(t)\leq \varepsilon, \quad\quad\text{and}\quad\quad \|w(t)\|_{\dot C^\frac{3}{2}}\leq \varepsilon t^{-\frac{1}{2}}~~\text{for any}~t\in[0,T].
			\end{align*}
			Then there holds $$
			\kappa_w(T)\leq 2\kappa_w(0).$$
		\end{lemma}
		\begin{proof}
			For any $t\in[0,T]$, by $	Q_{w}(t)\leq \varepsilon$ one has for any $\alpha$
			$$\sup_s	\frac{1}{|\Delta_\alpha w(s)|}\leq \kappa_w(0)+\left(\frac{t}{|\alpha|}\right)^{\varepsilon'}\varepsilon.$$
			If $|\alpha|\geq t$, then we obtain $\sup |\Delta_\alpha w|^{-1}\leq 2\kappa_w(0)$. It remains to consider $|\alpha|\leq t$.
			Note that for any $|\beta|\leq|\alpha|$
			$$
			|\Delta_\alpha w(t,s)-\Delta_\beta w(t,s)|\leq 2 \|w\|_{\dot C^\frac{3}{2}}|\alpha|^\frac{1}{2}\leq2\left(\frac{|\alpha|}{t}\right)^{\frac{1}{2}}\varepsilon.$$
			Hence $$
			|\Delta_\beta w(t,s)|\geq \left(\kappa_w(0)+\left(\frac{t}{|\alpha|}\right)^{\varepsilon'}\varepsilon\right)^{-1}-2\left(\frac{|\alpha|}{t}\right)^{\frac{1}{2}}\varepsilon.$$
			We can take $|\alpha|=t$ in the right hand side, which leads to 
			$$
			|\Delta_\beta w(t,s)|\geq\left(\kappa_w(0)+\varepsilon\right)^{-1}-2\varepsilon\geq \frac{1}{2}\kappa_w(0)^{-1}.
			$$
			Then we complete the proof.
		\end{proof}
	\vspace{0.3cm}\\
		The following are some elementary properties of the function $H(x)$ which will be used to estimate $N$ and $\mathfrak{N}$ in section 4.
		\begin{lemma}\label{leH} Let $H$ be as defined in \eqref{Ne}. Then for any $A_1, A_2, A_3, A_4\in\mathbb{R}^2\backslash\{0\}$, there holds, 
			\begin{align}\label{H1}
				&\left|H(A_1)\right|\lesssim \frac{1}{|A_1|},\quad\quad\quad
				\left|H(A_1)-H(A_2)\right|\lesssim \left(\frac{1}{|A_1|^2}+\frac{1}{|A_2|^2}\right)|A_1-A_2|,		\\
				&\label{H2}	\left|(H(A_{1})-H(A_{2}))-(H(A_{3})-H(A_{4}))\right|\\&\quad\lesssim \left(\sum_{j=1}^{4}\frac{1}{|A_j|^2}\right)|(A_1-A_2)-(A_3-A_4)|+\left(\sum_{j=1}^{4}\frac{1}{|A_j|^3}\right)|A_3-A_4|\left(|A_3-A_1|+|A_4-A_2|\right),\nonumber
			\\
				&\left|D_H(A_1,A_2)\right|\lesssim \left(\sum_{j=1}^{2}\frac{1}{|A_j|^3}\right)|A_1-A_2|^2,\label{H3}\\
		\label{H4}
				&\left|D_H(A_1,A_2)-D_H(A_3,A_4)\right|\lesssim \left(\sum_{j=1}^{4}\frac{1}{|A_j|^3}\right)|(A_1-A_2)-(A_3-A_4)|\left(|A_1-A_2|+|A_3-A_4|\right)\\
				&\quad\quad\quad\quad\quad\quad\quad\quad	\quad\quad\quad\quad\quad\quad +\left(\sum_{j=1}^{4}\frac{1}{|A_j|^4}\right)|A_3-A_4|^2\left(|A_3-A_1|+|A_4-A_2|\right),\nonumber
			\end{align}
			where we denote $D_H(A_1,A_2)=H(A_1)-H(A_2)-(A_1-A_2)\cdot\nabla H(A_2)$.
		\end{lemma}
		\section{Regularity of the Nonlinear Parabolic Equation}\label{sectionparabolic}
		Consider the regularity of the following parabolic equation
		\begin{equation}\label{paraboliceq}
			\begin{split}
				\partial_tf(t,x)+\frac{1}{4}\Lambda f(t,x)&=G(t,x),\quad\quad (t,x)\in \mathbb{R}^+\times \mathbb{R},\\
				f(0,x)&=f_0(x).
			\end{split}
		\end{equation}
		We have the kernel
		\begin{align*}
			K(t,x)=\mathcal{F}^{-1}\left(e^{-\frac{t|\xi|}{4}}\right)(x)=\frac{8t}{t^2+64\pi^2 x^2},
		\end{align*}
		where $\mathcal{F}^{-1}$ is the inverse Fourier transform in $\mathbb{R}$. 	
			The solution to system \eqref{paraboliceq} has the formula
		\begin{align*}
			f(t,x)&=\int K(t,x-y)f_0(y)dy+\int_0^t\int K(t-\tau,x-y)G(\tau,y)dyd\tau\\
			&:=f_L(t,x)+f_N(t,x).
		\end{align*}
		It is easy to check that $K(t,y)>0$ and $\int_{\mathbb{R}} K(t,y)dy=1$.
		Moreover, we have the following properties for the kernel
		\begin{align}\label{a1}
			\|\delta_\alpha\partial_x^kK(t,\cdot)\|_{L^1}&\lesssim\min\left\{1,|\alpha|t^{-1}\right\}t^{-k},\\
			\|\delta_\alpha\partial_x^k\mathcal{H}\Lambda^{\gamma} K(t,\cdot)\|_{L^1}&\lesssim\min\left\{1,|\alpha|t^{-1}\right\}t^{-(k+{\gamma})},\label{fractionald}
		\end{align}
		for $k=0,1,2$, $\gamma\in(0,1)$.

			\begin{lemma}\label{lemequiv}
			Let $\tilde{ \mathcal{G}}$ be the function space associate to the norm $\|\cdot\|_{\tilde {\mathcal{G}}}$, then 
			\begin{align*}
				\tilde {\mathcal{G}}=\dot B^0_{\infty,\infty}.
			\end{align*}
		\end{lemma}
		\begin{proof}
			Recall the definition of $\|\cdot\|_{\tilde {\mathcal{G}}}$
			$$\|h\|_{\tilde {\mathcal{G}}} :=\sup_{0\leq \mu\leq\frac{2}{3}\atop 2\varepsilon'\leq b\leq\theta-\mu-\varepsilon' }\sup_{\alpha\in\mathbb{R}}\frac{\|t^\mu\delta_\alpha \Lambda^{b-\varepsilon'} (K(t,\cdot)\ast h)\|_{L^\frac{1}{b}_\infty L^\infty}}{|\alpha|^{\mu+\varepsilon'}}.
			$$
			We have the following characterization of $\dot B^0_{\infty,\infty}$ (see \cite{BaChDa,Trie}):
			\begin{align*}
				\|h\|_{\dot B^0_{\infty,\infty}}:=\sup_{t>0}t\|\Lambda(K(t,\cdot)\ast h)\|_{L^\infty}\sim \sup_{t>0}t^\gamma\|\Lambda^\gamma(K(t,\cdot)\ast h)\|_{L^\infty},\ \ \ \forall \gamma>0.
			\end{align*}
		\textbf{Step 1.}	We claim that 
			for any $\gamma_1,\gamma_2>0$,
			\begin{align}\label{equiv}
				\sup_{t>0}t^{\gamma_1+\gamma_2}\|\Lambda^{\gamma_1}(K(t,\cdot)\ast h)\|_{\dot C^{\gamma_2}}\sim 	\|h\|_{\dot B^0_{\infty,\infty}}. 
			\end{align}
			It is easy to check that 
			\begin{align*}
				\sup_{t>0}t^{\gamma_1+\gamma_2}\|\Lambda^{\gamma_1}(K(t,\cdot)\ast h)\|_{\dot C^{\gamma_2}}\lesssim 	\|h\|_{\dot B^0_{\infty,\infty}}. 
			\end{align*}
			for any $\gamma_1,\gamma_2>0$. Moreover, for any $g$
			\begin{align*}
				||\Lambda^{\gamma_3}K(1,\cdot)\ast g||\lesssim ||g||_{\dot C^{\gamma_2}}~~\forall \gamma_3>\gamma_2.
			\end{align*}
			Then, for any $\gamma>\gamma_1+\gamma_2$, 
			\begin{align*}
				\|\Lambda^\gamma(K(1,\cdot)\ast h)\|_{L^\infty}\lesssim \|\Lambda^{\gamma_1}(K(\frac{1}{2},\cdot)\ast h)\|_{\dot C^{\gamma_2}}\lesssim	\sup_{t>0}t^{\gamma_1+\gamma_2}\|\Lambda^{\gamma_1}(K(t,\cdot)\ast h)\|_{\dot C^{\gamma_2}},
			\end{align*}
			where we used the fact that $K(t,x)=t^{-1}K(1,x/t)$. 
			Hence we obtain
			\begin{align*}
				\|h\|_{\dot B^0_{\infty,\infty}}\sim	\sup_{t>0}t^\gamma\|\Lambda^\gamma(K(t,\cdot)\ast h)\|_{L^\infty}\lesssim	\sup_{t>0}t^{\gamma_1+\gamma_2}\|\Lambda^{\gamma_1}(K(t,\cdot)\ast h)\|_{\dot C^{\gamma_2}},
			\end{align*}
			which implies \eqref{equiv}.\\
			\textbf{Step 2.}
			For simplicity, we denote 
			$$
			f(t,x)=K(t,\cdot)\ast f_0(x).
			$$
			Note that for any $t_1, t_2>0$,
			\begin{equation*}
				f(t_1+t_2,x)=K(t_1,\cdot)\ast f(t_2)(x).
			\end{equation*}
			We have
			\begin{align*}
				K(t,\cdot )\ast f_0=\frac{2}{t}\int_0^\frac{t}{2}K(t-\tau,\cdot)\ast f(\tau,\cdot)d\tau.
			\end{align*}
			Hence for any $\alpha\in \mathbb{R}\backslash\{0\}$,
			\begin{align}
				\|\delta_\alpha \Lambda^\frac{1}{4} f(t,\cdot)\|_{L^\infty}&\lesssim \frac{1}{t}\int_0^\frac{t}{2}\|K(t-\tau,\cdot)\ast \Lambda^\frac{1}{4} \delta_\alpha f(\tau,\cdot)\|_{L^\infty}d\tau\nonumber\\
				&\lesssim  \frac{1}{t}\int_0^\frac{t}{2}\|\Lambda^\frac{1}{4} \delta_\alpha f(\tau,\cdot)\|_{L^\infty}d\tau,\label{2}
			\end{align}
			where we use the fact that $\|K(t,\cdot)\|_{L^1}=1$ for any $t>0$.
			By H\"{o}lder's inequality we obtain 
			\begin{align*}
				\int_0^\frac{t}{2}\| \Lambda^\frac{1}{4} \delta_\alpha f(\tau,\cdot)\|_{L^\infty}d\tau\lesssim t^\frac{1}{2}\| \Lambda^\frac{1}{4} \delta_\alpha f(\tau,\cdot)\|_{L^2_tL^\infty}\lesssim t^\frac{1}{2} |\alpha|^\frac{1}{4}\|f\|_{\tilde {\mathcal{G}}}.
			\end{align*}
			Combining this with \eqref{2} and \eqref{equiv},   we obtain 
			\begin{align*}
				\|f_0\|_{\dot B^0_{\infty,\infty}}\lesssim 	\|f_0\|_{\tilde {\mathcal{G}}}.
			\end{align*}
			Now we prove that 
			\begin{align}\label{eq2}
				\|f_0\|_{\tilde {\mathcal{G}}}\lesssim 	\|f_0\|_{\dot B^0_{\infty,\infty}}.
			\end{align}
			Let $\mu, b$ be such that $0\leq \mu\leq\frac{2}{3},  2\varepsilon'\leq b\leq\theta-\mu-\varepsilon' $. We have 
			\begin{align*}
				\|\delta_\alpha \Lambda^{b-\varepsilon'} (K(t,\cdot)\ast h)\|_{L^\infty}\overset{\eqref{equiv}}\lesssim \frac{\|h\|_{\dot B^0_{\infty,\infty}}}{t^{b-\varepsilon'}}\min\left\{1,\frac{|\alpha|}{t}\right\}.
			\end{align*}
			Hence 
			\begin{align*}
				&\|t^\mu\delta_\alpha \Lambda^{b-\varepsilon'} (K(t,\cdot)\ast h)\|_{L^\frac{1}{b}_\infty L^\infty}\\
				&\lesssim  \|h\|_{\dot B^0_{\infty,\infty}}\left(\int_0^{|\alpha|}t^\frac{\mu-b+\varepsilon'}{b}dt\right)^b+\|h\|_{\dot B^0_{\infty,\infty}}|\alpha|\left(\int_{|\alpha|}^\infty t^\frac{\mu-b+\varepsilon'-1}{b}dt\right)^b\\
				&\lesssim |\alpha|^{\mu+\varepsilon'}\|h\|_{\dot B^0_{\infty,\infty}},
			\end{align*}
			which implies \eqref{eq2}. This completes the proof.
		\end{proof}
		\begin{remark}\label{Rem1}
			From Lemma \ref{moveabove}, one can check that for any function $h:[0,T]\times\mathbb{R}\rightarrow\mathbb{R}$ and any $\mu,b,b_1$ such that  $0\leq\mu\leq\frac{2}{3},~2\varepsilon'\leq b\leq\theta-\mu-\varepsilon',~\mu+b+\varepsilon'\leq b_1\leq 1+b-\varepsilon
			'$, there holds
			\begin{align*}
				\sup_{\alpha\in\mathbb{R}}
				\frac{\|t^\mu\delta_\alpha \Lambda^{b_1} h\|_{L_T^\frac{1}{b}L^\infty}}{|\alpha|^{\mu+b-b_1+1}}	\lesssim\|\partial_x h\|_{\mathcal{G}_T}.
			\end{align*}
			%
		\end{remark}
		\begin{lemma}\label{estMf22}
			Consider equation \eqref{paraboliceq} with $G(t,x)=\partial_x N(t,x)$,
			then the solution $f$ satisfies for any $T>0$
			\begin{align*}
				\|f\|_{\mathcal{G}_T}\leq \|f_0\|_{\tilde{\mathcal{G}}_T}+ C \mathcal{M}(T),
			\end{align*}
			where we denote 
			\begin{align}\label{defM}
				\mathcal{M}(T):=	\sup_{0\leq\mu\leq\frac{2}{3}\atop\theta-\varepsilon'\leq \mu+a\leq1-\varepsilon'}\sup_\beta \frac{\|t^\mu\delta_\beta N(t,x)\|_{L^\frac{1}{a}_TL^\infty}}{|\beta|^{\mu+a}}.
			\end{align}
		\end{lemma}
		\begin{proof}
			By definition, we have $\|f_L\|_{\mathcal{G}_T}\leq \|f_0\|_{\tilde{\mathcal{G}}_T}$. It remains to prove that 
			\begin{align}\label{nonp}
				\|f_N\|_{\mathcal{G}_T}\leq C \mathcal{M}(T).
			\end{align}
			We have 
			\begin{align*}
				f_N(t,x)&=\int_0^t\int K(t-\tau,x-y) \partial_xN(\tau,y)dyd\tau=\int_0^t\int\mathcal{H}\Lambda^{1-\theta}K(t-\tau,x-y)\Lambda^\theta N(\tau,y)dyd\tau.
			\end{align*}
			Fix $b,\mu$ such that $0\leq\mu\leq\frac{2}{3}$ and $ 2\varepsilon'\leq b\leq\theta-\mu-\varepsilon'$.
			One has
			$$
			\left|\delta_\alpha \Lambda^{b-\varepsilon'}f_N(t,x)\right|\lesssim	\int_0^t\iint|\delta_\alpha\mathcal{H}\Lambda^{1-\theta+b-\varepsilon'}K(t-\tau,x-y)| |\delta_\beta  N(\tau,y)|\frac{d\beta dyd\tau}{|\beta|^{1+\theta}} .$$
			Recalling \eqref{fractionald}, one has
			$$
			\int|\delta_\alpha\mathcal{H}\Lambda^{1-\theta+b-\varepsilon'}K(x-y,t-\tau)|dy\lesssim \frac{1}{(t-\tau)^{1-\theta+b-\varepsilon'}}\min\left\{1,\frac{|\alpha|}{t-\tau}\right\}.$$
			Hence we obtain $$
			\|t^\mu\delta_\alpha \Lambda^{b-\varepsilon'}f_N\|_{L^\infty}\lesssim \iint_0^t\frac{t^\mu+(t-\tau)^\mu}{(t-\tau)^{1-\theta+b-\varepsilon'}}\min\left\{1,\frac{|\alpha|}{t-\tau}\right\}\|\delta_\beta  N(\tau,\cdot)\|_{L^\infty}\frac{d\tau d\beta}{|\beta|^{1+\theta}},$$
			where we also use the fact that $t^\mu\lesssim \tau^\mu+(t-\tau)^\mu$ for any $\tau\in(0,t)$.
			Applying Lemma \ref{lemRiesz} one obtains 
			\begin{align*}
				&\left\|\int_0^t\frac{\tau^\mu\|\delta_\beta  N(\tau,\cdot)\|_{L^\infty}}{(t-\tau)^{1-\theta+b-\varepsilon'}}\min\left\{1,\frac{|\alpha|}{t-\tau}\right\}d\tau\right\|_{L^\frac{1}{b}_T}\lesssim |\alpha|^{\sigma_\pm} \|\tau^\mu\delta_\beta N\|_{L^{q_{\pm}}_TL^\infty},\\
				&\left\|\int_0^t\frac{\|\delta_\beta  N(\tau,\cdot)\|_{L^\infty}}{(t-\tau)^{1-\theta+b-\varepsilon'-\mu}}\min\left\{1,\frac{|\alpha|}{t-\tau}\right\}d\tau\right\|_{L^\frac{1}{b}_T}\lesssim |\alpha|^{\sigma_\pm} \|\delta_\beta N\|_{L^{p_{\pm}}_TL^\infty},
			\end{align*}
			where $\sigma_\pm=\mu+\varepsilon'\pm\frac{\varepsilon'}{2}$, $p_\pm=(\mu+\varepsilon'+\theta-\sigma_\pm)^{-1}$ and $q_\pm=(\varepsilon'+\theta-\sigma_\pm)^{-1}$.
			Then we have 
			\begin{align*}
				\|t^\mu\delta_\alpha \Lambda^{b-\varepsilon'}f_N\|_{L^\frac{1}{b}_TL^\infty}\lesssim& \int\min_{+,-}\left\{ |\alpha|^{\sigma_\pm}\|\delta_\beta N\|_{L^{p_{\pm}}_TL^\infty}\right\}\frac{d\beta}{|\beta|^{1+\theta}}+\int\min_{+,-}\left\{ |\alpha|^{\sigma_\pm}\|\tau^\mu\delta_\beta N\|_{L^{q_{\pm}}_TL^\infty}\right\}\frac{d\beta}{|\beta|^{1+\theta}}\\
				\lesssim& |\alpha|^{\mu+\varepsilon'}\sum_{+,-}\left(\sup_\beta \frac{\|\delta_\beta N\|_{L^{p_{\mp}}_TL^\infty}}{|\beta|^{\theta\pm\frac{\varepsilon'}{2}}}+\sup_\beta \frac{\|\tau^\mu\delta_\beta N\|_{L^{q_{\mp}}_TL^\infty}}{|\beta|^{\theta\pm\frac{\varepsilon'}{2}}}\right)\\
				\lesssim& |\alpha|^{\mu+\varepsilon'}\mathcal{M}(T),
			\end{align*}
			which implies \eqref{nonp}. This completes the proof.
		\end{proof}
		\begin{lemma}\label{lempropsmoothe}
			For any $\mu\in[0,\frac{2}{3}]$ and any function $N:\mathbb{R}^+\times\mathbb{R}\rightarrow\mathbb{R}$, there holds
			\begin{align*}
				\sup_{t\in[0,T]}\int_0^t \frac{\tau^\mu}{(t-\tau)^\mu}\int\|\delta_\beta N (\tau,\cdot)\|_{L^\infty}\frac{d\beta}{|\beta|^{1+\theta}}d\tau\lesssim T^{1-\theta} \mathcal{M}(T),
			\end{align*}
			where $\mathcal{M}(T)$ is defined in \eqref{defM}.
		\end{lemma}
		\begin{proof}
			By H\"{o}lder's inequality we have for any $\mu\in[0,\frac{2}{3}]$,
			\begin{align*}
				\int_0^t \frac{\tau^\mu}{(t-\tau)^\mu}\int\|\delta_\beta N (\tau,\cdot)\|_{L^\infty}&\frac{d\beta}{|\beta|^{1+\theta}}d\tau\lesssim T^{1-\theta-\varepsilon'}\sup_\beta\frac{\|\tau^\mu\delta_\beta N\|_{L_T^\frac{1}{\theta-\mu+\varepsilon'}L^\infty}}{|\beta|^{\theta+\varepsilon'}}\int_{|\beta|\leq T}\frac{d\beta}{|\beta|^{1-\varepsilon'}}\\
				&\quad+T^{1-\theta+\varepsilon'}\sup_\beta\frac{\|\tau^\mu\delta_\beta N\|_{L_T^\frac{1}{\theta-\mu-\varepsilon'}L^\infty}}{|\beta|^{\theta-\varepsilon'}}\int_{|\beta|\geq T}\frac{d\beta}{|\beta|^{1+\varepsilon'}} \\
				&\quad\lesssim T^{1-\theta}\mathcal{M}(T).
			\end{align*}
			The proof is complete.
		\end{proof}
	\vspace{0.3cm}\\
		To study \eqref{eqglo}, we consider the system
		\begin{equation}\label{eqwithL}
			\partial_{t} w+\mathcal{L} w= F\in \mathbb{R}^2,
		\end{equation}
		where the operator $\mathcal{L}$ is defined in Section \ref{Formulation near the steady state}.
		Recalling \eqref{kkkkkernel}, let $v=\mathcal{O}_x w$, then each component of $v$ satisfies the equation \eqref{paraboliceq} with nonlinear terms $\mathcal{O}_x F(t,x)$. Hence 
$$
			v(t,x)=\int K(t,x-y)v(0,y)dy+\int_0^t\int K(t-\tau,x-y)\mathcal{O}_y F(\tau,y)dyd\tau.$$
		We obtain
		\begin{align}\label{formulaY}
			w(t,x)=\int \tilde K(t,x-y) w(0,y)dy+\int_0^t\int \tilde K(t-\tau,x-y) F(\tau,y)dyd\tau,
		\end{align}
		where we denote $\tilde K(t,x)=K(t,x)\mathcal{O}_x^T$.
		
		For system \eqref{eqwithL}, we have the following result, which is an analogy to Lemma \ref{estMf22}.
		\begin{lemma}\label{propdeltaNY}
			Consider the system \eqref{eqwithL} with $w(0,x)=w_0(x)$ and $F(t,x)=\partial_x N(t,x)$,
			then for any $T>0$, we have 
$$
				\|w\|_{\mathcal{G}_T}\leq \|w_0\|_{\tilde {\mathcal{G}}_T}+C	\sup_{0\leq\mu\leq\frac{2}{3}\atop\theta-\varepsilon'\leq \mu+a\leq1-\varepsilon'}\sup_\beta \frac{\|t^\mu\delta_\beta N(t,x)\|_{L^\frac{1}{a}_TL^\infty}}{|\beta|^{\mu+a}}.$$
		\end{lemma}

		\section{Establish the main estimates}
		\subsection{Estimate the nonlinear terms}
		\begin{proposition}\label{propdeltaN} Let $N$ be as defined in \eqref{eqpeskin}, then for any $T>0$,
			there holds 
$$
				\sup_{0\leq\mu\leq\frac{2}{3}\atop\theta-\varepsilon'\leq \mu+a\leq1-\varepsilon'}\sup_\beta \frac{\|t^\mu\delta_\beta N(t,x)\|_{L^\frac{1}{a}_TL^\infty}}{|\beta|^{\mu+a}}\lesssim\left(1+\kappa(T)\right)^2\|X'\|_{\mathcal{G}_T}^2(1+\|X'\|_{\mathcal{G}_T})^2.$$
		\end{proposition}
		\begin{proof}
			Fix $a$ and $\mu$ such that $0\leq\mu\leq\frac{2}{3}$, $\theta-\varepsilon'\leq \mu+a\leq1-\varepsilon'$. Recall that the nonlinear term can be written as
			\begin{align*}
				N(t,s)=\sum\int H(\tilde\Delta_\alpha X(s)) E^\alpha X_{i}(s)\delta_\alpha X_{j}'(s)\frac{d\alpha}{\alpha}.
			\end{align*}
			We have 
			\begin{align*}
				|\delta_\beta N(t,s)|\lesssim&\left|\sum\int H(\tilde\Delta_\alpha X(s)) \delta_\beta E^\alpha X_{i}(s)\delta_\alpha X_{j}'(s-\beta)\frac{d\alpha}{\alpha} \right|\\& +\left|\sum\int H(\tilde\Delta_\alpha X(s))  E^\alpha X_{i}(s)\delta_\beta X_{j}'(s-\alpha)\frac{d\alpha}{\alpha} \right|
				\\&+\left|\sum\int\delta_\beta H(\tilde\Delta_\alpha X)(s) E^\alpha X_{i}(s-\beta)\delta_\alpha X_{j}'(s-\beta)\frac{d\alpha}{\alpha}\right| \\
				:=&J_1+J_2+J_3,
			\end{align*}
			where we also used the fact that 
		$$\sum\int H(\tilde\Delta_\alpha X(s))  E^\alpha X_{i}(s)\delta_\beta\delta_\alpha X_{j}'(s)\frac{d\alpha}{\alpha}\overset{\eqref{integral0}}=-\sum\int H(\tilde\Delta_\alpha X(s))  E^\alpha X_{i}(s)\delta_\beta X_{j}'(s-\alpha)\frac{d\alpha}{\alpha}.$$
			Note that 
			$$
			E^\alpha X(s)= \alpha \partial_\alpha\left(\Delta_\alpha X(s)\right)+\delta_\alpha X(s)\left(\frac{1}{\alpha}-\frac{1}{\tilde\alpha}\right).
			$$
			We have
			\begin{align*}
				J_{1}&\lesssim\left|\int H(\tilde\Delta_\alpha X(s))\delta_\alpha X_{j}'(s-\beta) \partial_\alpha[\Delta_\alpha  \delta_\beta X_{i}(s)]{d\alpha}\right|\\
				&\quad\quad\quad+\left|\int H(\tilde\Delta_\alpha X(s))\delta_\alpha X_{j}'(s-\beta) \delta_\beta\delta_\alpha X(s)\left(\frac{1}{\alpha}-\frac{1}{\tilde\alpha}\right)\frac{d\alpha}{\alpha}\right|\\
				&:=J_{1,1}+J_{1,2}.
			\end{align*}
			We first estimate $J_{1,1}$. 
			By \eqref{H1}, one has
			\begin{align}
				&\sup_\alpha\|H(\tilde\Delta_\alpha X(t,\cdot))\|_{L^\infty}=\sup_{\alpha\in(-\pi,\pi)}\|H(\tilde\Delta_\alpha X(t,\cdot))\|_{L^\infty}\lesssim \kappa(t),\label{Hkappa}\\
				&\|H(\tilde\Delta_\alpha X)-H(\tilde\Delta_{\alpha-z} X)\|_{L^\infty}\lesssim \kappa(t)^2\|\tilde\Delta_\alpha X-\tilde\Delta_{\alpha-z} X\|_{L^\infty}.\nonumber
			\end{align}
			Let $f(t,\alpha)=H(\tilde\Delta_\alpha X(s))\delta_\alpha X'(t,s-\beta)$, $g(t,s)=t^\mu\delta_\beta X(t,s)$, there holds
			\begin{align*}
				&|\delta_z f(t,\cdot)(\alpha)|\lesssim\kappa(t)\|\delta_z X'\|_{L^\infty}+ \kappa(t)^2\|\tilde\Delta_\alpha X-\tilde\Delta_{\alpha-z} X\|_{L^\infty}\|\delta_\alpha X'\|_{L^\infty}.
			\end{align*}
			For any $p, q$ such that $4\varepsilon'\leq\frac{1}{p}\leq \theta-\varepsilon'$ and   $2\varepsilon'\leq\frac{1}{q}+\mu\leq\frac{3}{4}-\frac{\varepsilon'}{2},q\leq 30$, denote
			\begin{align*}
				&V_{q,\mu}(\beta)= 	\|t^\mu\delta_\beta \Lambda^{\frac{3+\varepsilon'}{4}}X\|^\frac{1}{2}_{L^q_TL^\infty}\|t^\mu\delta_\beta \Lambda^{\frac{3-\varepsilon'}{4}}X\|^\frac{1}{2}_{L^q_TL^\infty},\quad\quad\\
				&U_p=\sup_{\alpha,z}\frac{\left\|\|\tilde\Delta_\alpha X-\tilde\Delta_{\alpha-z} X\|_{L^\infty}\|\delta_\alpha X'\|_{L^\infty}\right\|_{L^p_T}}{|z|^\frac{1}{p}}+\sup_z\frac{\|\delta_z X'\|_{L_T^pL^\infty}}{|z|^\frac{1}{p}}.
			\end{align*}
			By Lemma \ref{lem2} and Remark \ref{Rem1}, it is easy to check that 
			\begin{align}\label{VandU}
				V_{q,\mu}(\beta)&\lesssim|\beta|^{\frac{1}{q}+\frac{1}{4}+\mu} \|X'\|_{\mathcal{G}_T},\quad\quad\quad U_p\lesssim \|X'\|_{\mathcal{G}_T}(1+\|X'\|_{\mathcal{G}_T}).
			\end{align}
			For any $q>1$, we have 
			$$
			\sup_{\alpha,s}\|\delta_z f(t,\cdot)(\alpha)\|_{L^q_T}\lesssim|z|^\frac{1}{q} \left(1+\kappa(T)\right)^2U_q.
			$$ 
			Applying Remark \ref{remmovedeT}  with above $f$,  $g$ and parameters $\sigma=\frac{1}{4}$, $p=\frac{1}{a}$, $p_{1\pm}=\frac{4}{1\pm\varepsilon'}$, $p_{2\pm}=\frac{4}{4a\pm\varepsilon'-1}$, we obtain
			%
			\begin{align}
				\|t^\mu J_{1,1}\|_{L^\frac{1}{a}_TL^\infty}\lesssim& \sup_{\alpha,s}\int\min_{+,-}\left\{ \|\delta_z f(t,\cdot)(\alpha)\|_{L^\frac{4}{1\pm\varepsilon'}_T}V_{\frac{4}{4a\mp\varepsilon'-1},\mu}(\beta)\right\}\frac{dz}{|z|^\frac{5}{4}}\nonumber\\		
				\lesssim& (1+\kappa(T))^2\int \min_{+,-}\left\{|z|^{\pm\frac{\varepsilon'}{4}}U_\frac{4}{1\pm\varepsilon'}V_{\frac{4}{4a\mp\varepsilon'-1},\mu}(\beta)\right\}\frac{dz}{|z|}
				\nonumber\\
				\overset{\eqref{VandU}}\lesssim& |\beta|^{\mu+a}\left(1+\kappa(T)\right)^2\|X'\|_{\mathcal{G}_T}^2(1+\|X'\|_{\mathcal{G}_T})^2.\label{J11}
			\end{align}
			For $J_{1,2}$, by \eqref{TR} one has for any $2\pi$ -periodic function $h$
			\begin{equation}\label{z2}
				\int \frac{h(\alpha)}{\tilde\alpha}\frac{d\alpha}{\alpha}=\int_{-\pi}^{\pi}h(\alpha)\frac{d\alpha}{|\tilde\alpha|^2}.
			\end{equation}
			Hence 
			\begin{align*}
				J_{1,2}\lesssim \int| H(\tilde\Delta_\alpha X(s))\delta_\alpha X_{j}'(s-\beta) \delta_\beta\delta_\alpha X(s)|\frac{d\alpha}{|\alpha|^2}\overset{\eqref{Hkappa}}\lesssim \kappa(t) \int|\delta_\alpha X_{j}'(s-\beta) \delta_\beta\delta_\alpha X(s)|\frac{d\alpha}{|\alpha|^2}.
			\end{align*}
			By \eqref{interpfrac2}, we have 
			\begin{align}\label{12haha}
				\sup_{\alpha,s} \frac{|\delta_\beta\delta_\alpha X(s)|}{|\alpha|^\frac{3}{4}}\lesssim\|\delta_\beta\Lambda^\frac{3}{4} X\|_{L^\infty}. 	
			\end{align}
			Hence
			$$
			J_{1,2}\lesssim \kappa(t)\int \|\delta_\alpha X'\|_{L^\infty}\frac{d\alpha}{|\alpha|^\frac{5}{4}} \|\delta_\beta\Lambda^\frac{3}{4} X\|_{L^\infty}.
			$$
			Applying H\"{o}lder's inequality and Minkowski inequality  we obtain
			\begin{align*}
				\|t^\mu J_{1,2}\|_{L^\frac{1}{a}_TL^\infty}
				&\lesssim \kappa(T)\int\min_{+,-}\left\{\|\delta_\alpha X'\|_{L_T^\frac{4}{1\pm\varepsilon'}L^\infty}\|t^\mu\delta_\beta\Lambda^\frac{3}{4} X\|_{L_T^\frac{4}{4a-1\mp\varepsilon'}L^\infty} \right\}\frac{d\alpha}{|\alpha|^\frac{5}{4}}\\
				&\lesssim \kappa(T)\sum_{+,-}|\beta|^{\pm\frac{\varepsilon'}{4}}\sup_\alpha \frac{\|\delta_\alpha X'\|_{L_T^\frac{4}{1\pm\varepsilon'}L^\infty}}{|\alpha|^\frac{1\pm\varepsilon'}{4}}\|t^\mu\delta_\beta\Lambda^\frac{3}{4} X\|_{L_T^\frac{4}{4a-1\mp\varepsilon'}L^\infty}.
			\end{align*}
			By Remark \ref{Rem1}, there holds
			\begin{align}\label{blabla}
				\|t^\mu\delta_\beta\Lambda^\frac{3}{4} X\|_{L_T^\frac{4}{4a-1\pm\varepsilon'}L^\infty}\lesssim |\beta|^{\mu+a\pm\frac{\varepsilon'}{4}}\|X'\|_{\mathcal{G}_T}.
			\end{align}
			Hence we get 
			\begin{align*}
				\|t^\mu J_{1,2}\|_{L^\frac{1}{a}_TL^\infty}\lesssim |\beta|^{\mu+a}\left(1+\kappa(T)\right)^2\|X'\|_{\mathcal{G}_T}^2.
			\end{align*}
			Combining this with \eqref{J11} we have
			\begin{align}\label{J1}
				\|t^\mu J_{1}\|_{L^\frac{1}{a}_TL^\infty}\lesssim|\beta|^{\mu+a}\left(1+\kappa(T)\right)^2\|X'\|_{\mathcal{G}_T}^2(1+\|X'\|_{\mathcal{G}_T})^2.
			\end{align}
			Then we estimate $J_2$. Note that $X'(s-\alpha)=\partial_\alpha(\delta_\alpha X(s))=\partial_\alpha(\alpha\Delta_\alpha X(s))$, hence
			we have 
			\begin{align*}
				J_{2}
				&\lesssim\left|\int H(\tilde\Delta_\alpha X)  E^\alpha X_{i}\partial_\alpha[\Delta_\alpha\delta_\beta X_{j}]{d\alpha}\right|+\left|\int H(\tilde\Delta_\alpha X)  E^\alpha X_{i}\Delta_\alpha\delta_\beta X_{j}\frac{d\alpha}{\alpha}\right|.
			\end{align*}
			We can follow the estimates of $J_{1,1}$ and $J_{1,2}$ to estimate the above two terms, we conclude that
			\begin{align*}
				\|t^\mu J_{2}\|_{L^\frac{1}{a}_TL^\infty}\lesssim|\beta|^{\mu+a}\left(1+\kappa(T)\right)^2\|X'\|_{\mathcal{G}_T}^2(1+\|X'\|_{\mathcal{G}_T})^2.
			\end{align*}
			For $J_3$, by \eqref{TR} we have 
			\begin{align*}
				J_3=\left|\int_{-\pi}^{\pi}\delta_\beta (H(\tilde\Delta_\alpha X))(s) E^\alpha X_{i}(s-\beta)\delta_\alpha X_{j}'(s-\beta)\frac{d\alpha}{\tilde\alpha}\right|.
			\end{align*}
			By \eqref{H1} we have
			for any $\alpha\in(-\pi,\pi)$
			\begin{align*}
				|\delta_\beta (H(\tilde\Delta_\alpha X))(s) |\lesssim \kappa(t)^2\|\delta_\beta\Delta_\alpha X\|_{L^\infty}\overset{\eqref{12haha}}\lesssim |\alpha|^{-\frac{1}{4}} \kappa(t)^2\|\delta_\beta\Lambda^\frac{3}{4} X\|_{L^\infty}.
			\end{align*}
			Hence 
			\begin{align*}
				J_3\lesssim \kappa(t)^2\|\delta_\beta\Lambda^\frac{3}{4} X\|_{L^\infty}\int\|E^\alpha X\|_{L^\infty}\|\delta_\alpha X'\|_{L^\infty}\frac{d\alpha}{|\alpha|^\frac{5}{4}},
			\end{align*}
			where we also used \eqref{12haha}.
			Applying H\"{o}lder's inequality, Minkowski's inequality and \eqref{blabla} we obtain
			\begin{align*}
				\|t^\mu J_{3}\|_{L^\frac{1}{a}_TL^\infty}&\lesssim\kappa(T)^2\int\min_{+,-}\left\{|\beta|^{a+\mu\pm\varepsilon'}\|E^\alpha X\delta_\alpha X'\|_{L_T^\frac{4}{1\mp\varepsilon'}L^\infty}\right\}\frac{d\alpha}{|\alpha|^\frac{5}{4}}.
			\end{align*}
			Note that 
			\begin{align*}
				&\int_{|\alpha|\leq |\beta|}\|E^\alpha X\delta_\alpha X'\|_{L_T^\frac{4}{1\mp\varepsilon'}L^\infty}\frac{d\alpha}{|\alpha|^\frac{5}{4}}\lesssim \left(\sup_{\alpha}\frac{\|\delta_\alpha X'\|_{L_T^\frac{8}{1+\varepsilon'}L^\infty}}{|\alpha|^\frac{1+\varepsilon'}{8}}\right)^2|\beta|^\frac{\varepsilon'}{4},\nonumber\\
				&	\int_{|\alpha|\geq |\beta|}\|E^\alpha X\delta_\alpha X'\|_{L_T^\frac{4}{1\mp\varepsilon'}L^\infty}\frac{d\alpha}{|\alpha|^\frac{5}{4}}\lesssim \left(\sup_{\alpha}\frac{\|\delta_\alpha X'\|_{L_T^\frac{8}{1-\varepsilon'}L^\infty}}{|\alpha|^\frac{1-\varepsilon'}{8}}\right)^2|\beta|^\frac{-\varepsilon'}{4}.
			\end{align*}
			Hence we obtain 
			\begin{align}\label{J3}
				\|t^\mu J_{3}\|_{L^\frac{1}{a}_TL^\infty}&\lesssim|\beta|^{\mu+a}\kappa(T)^2\|X'\|_{\mathcal{G}_T}^3.
			\end{align}
			We conclude from \eqref{J1}-\eqref{J3} that 
			$$
			\sup_\beta \frac{\|t^\mu\delta_\beta N(t,s)\|_{L^\frac{1}{a}_TL^\infty}}{|\beta|^{\mu+a}}\lesssim\left(1+\kappa(T)\right)^2\|X'\|_{\mathcal{G}_T}^2(1+\|X'\|_{\mathcal{G}_T})^2.
			$$
			This completes the proof.
		\end{proof}
		\begin{proposition}
			\label{propparabolicY} Let $\mathfrak{N}$ be as defined in \eqref{eqlinearize}, then for any $T\in[0,1]$,
			there holds 
			\begin{align*}
				&\sup_{0\leq\mu\leq\frac{2}{3}\atop\theta-\varepsilon'\leq \mu+a\leq1-\varepsilon'}\sup_\beta \frac{\|t^\mu\delta_\beta \mathfrak{N}(t,x)\|_{L^\frac{1}{a}_TL^\infty}}{|\beta|^{\mu+a}}\\
				&\quad\quad\lesssim\Big(1+\kappa(T)+\sup_{t\in[0,T]}(\|Z'(t)\|_{L^\infty}^{-1})\Big)^{5}(1+ \|Z'\|_{L^\infty_TL^\infty}+\|Y'\|_{\mathcal{G}_T})^{5}\|Y'\|_{\mathcal{G}_T}^2.
			\end{align*}
		\end{proposition}
		\begin{proof}
			Note that by properties \eqref{haha1} and \eqref{haha2} one has
			\begin{align*}
				&\mathfrak{N}(Y+Z)=\mathfrak{N}(Y+Z)-\mathfrak{N}(Z)=\int_0^1\frac{d}{d r }\mathfrak{N}( r  Y+Z)d r =\int_0^1\mathfrak{DN}[ r  Y+Z]Yd r,\\&
				\mathfrak{DN}[ r  Y+Z]Y=\mathfrak{DN}[ r  Y+Z]Y-\mathfrak{DN}[Z]Y=\mathfrak{D}N[ r  Y+Z]Y-\mathfrak{D}N[Z]Y.
			\end{align*}
			Then we obtain
			$$
			\mathfrak{N}(Y+Z)=N(Y+Z)-N(Z)-\mathfrak{D}N[Z]Y.$$
			Recall the formula 
			$$
			N(X(s))=\sum\int H(\tilde\Delta_\alpha X(s)) E^\alpha X_i(s)  \delta_\alpha X_j'(s)\frac{d\alpha}{\alpha}.$$
			By a direct computation we obtain
			\begin{align*}
				&	N(X)-N(Z)=\sum\int H(\tilde\Delta_\alpha X) E^\alpha X_i  \delta_\alpha Y_j'\frac{d\alpha}{\alpha}+\sum\int H(\tilde\Delta_\alpha X) E^\alpha Y_i  \delta_\alpha Z_j'\frac{d\alpha}{\alpha}\\
				&\quad\quad\quad\quad\quad\quad\quad\quad\quad\quad+\sum\int [H(\tilde\Delta_\alpha X)-H(\tilde\Delta_\alpha Z)] E^\alpha Z_i \delta_\alpha Z_j'\frac{d\alpha}{\alpha},\\&
				\mathfrak{D}N[Z]Y=\sum\int H(\tilde\Delta_\alpha Z) E^\alpha Z_i  \delta_\alpha Y_j'\frac{d\alpha}{\alpha}+\sum\int H(\tilde\Delta_\alpha Z) E^\alpha Y_i  \delta_\alpha Z_j'\frac{d\alpha}{\alpha}\\
				&\quad\quad\quad\quad\quad\quad+\sum\int \mathfrak{D}\tilde H[Z]Y E^\alpha Z_i  \delta_\alpha Z_j'\frac{d\alpha}{\alpha},
			\end{align*}
			where we denote $\tilde H(Z)=H(\tilde \Delta_\alpha Z)$, hence $\mathfrak{D}\tilde H[Z]Y=\left.\frac{d}{d\epsilon}H(\tilde\Delta(Z+\epsilon Y))\right|_{\epsilon=0}$.
			Then we have
			\begin{align*}
				\mathfrak{N}(X)&=	\sum\int H(\tilde\Delta_\alpha X)E^\alpha Y_i \delta_\alpha Y_j'\frac{d\alpha}{\alpha}+\sum\int (H(\tilde\Delta_\alpha X)-H(\tilde\Delta_\alpha Z))E^\alpha Z_i \delta_\alpha Y_j'\frac{d\alpha}{\alpha}\\
				&\quad\quad+\sum\int (H(\tilde\Delta_\alpha X)-H(\tilde\Delta_\alpha Z))E^\alpha Y_i \delta_\alpha Z_j' \frac{d\alpha}{\alpha} \\&\quad\quad+\sum\int (H(\tilde\Delta_\alpha X)-H(\tilde\Delta_\alpha Z)-\mathfrak{D}\tilde H[Z]Y)E^\alpha Z_i \delta_\alpha Z_j' \frac{d\alpha}{\alpha}\\
				&:=\int\sum_{k=1}^4\mathbf{R}_k\frac{d\alpha}{\alpha}.
			\end{align*}
			Denote \begin{align*}
				&\tilde{\mathbf{R}}_1=\delta_\beta \mathbf{R}_1-\sum H(\tilde\Delta_\alpha X)E^\alpha Y_i \delta_\beta Y_j',\\
				&\tilde{\mathbf{R}}_2=\delta_\beta \mathbf{R}_2-\sum(H(\tilde\Delta_\alpha X)-H(\tilde\Delta_\alpha Z))E^\alpha Z_i \delta_\beta Y_j',\\
				&\tilde{\mathbf{R}}_k=\delta_\beta \mathbf{R}_k, \quad\quad\quad k=3,4.
			\end{align*}
			By \eqref{integral0} we have
			\begin{align*}
				&	\int \left(H(\tilde\Delta_\alpha X)E^\alpha Y_i +(H(\tilde\Delta_\alpha X)-H(\tilde\Delta_\alpha Z))E^\alpha Z_i\right) \delta_\beta Y_j'\frac{d\alpha}{\alpha}\\
				&\quad\quad\quad\quad\quad\quad=\int (H(\tilde\Delta_\alpha X)E^\alpha X_i-H(\tilde\Delta_\alpha Z)E^\alpha Z_i) \delta_\beta Y_j'\frac{d\alpha}{\alpha}=0.
			\end{align*}
			Hence we have
			$$
			\delta_\beta	\mathfrak{N}(X)=\int\sum_{k=1}^4\delta_\beta\mathbf{R}_k\frac{d\alpha}{\alpha}=\int\sum_{k=1}^4\tilde{\mathbf{R}}_k\frac{d\alpha}{\alpha}.$$
			Denote 
			\begin{align*}
				&\mathbf{P}_1=-\sum H(\tilde\Delta_\alpha X)E^\alpha Y_i\delta_\beta Y_j'(\cdot-\alpha),\quad\mathbf{P}_2=\sum H(\tilde\Delta_\alpha X)\delta_\beta E^\alpha Y_i\delta_\alpha Y_j'(\cdot-\beta),\\
				&\mathbf{P}_3=-\sum [H(\tilde\Delta_\alpha X)-H(\tilde\Delta_\alpha Z)] E^\alpha Z_i\delta_\beta Y_j'(\cdot-\alpha),\\ &\mathbf{P}_4=\sum (H(\tilde\Delta_\alpha X)-H(\tilde\Delta_\alpha Z))\delta_\beta E^\alpha Y_i \delta_\alpha Z_j' .
			\end{align*}
			Note that for any function $f_1,f_2,f_3$,
			$$
			|\delta_\beta (f_1f_2f_3)-f_1f_2\delta_\beta f_3-f_1\delta_\beta f_2 f_3(\cdot-\beta)|= |\delta_\beta f_1||f_2(\cdot-\beta)||f_3(\cdot-\beta)|,$$
			hence we have	
			\begin{align*}
				&|\tilde{\mathbf{R}}_1-\mathbf{P}_1-\mathbf{P}_2|=\left|	\delta_\beta\mathbf{R}_1-\sum H(\tilde\Delta_\alpha X)E^\alpha Y_i \delta_\beta\delta_\alpha Y_j'-\sum H(\tilde\Delta_\alpha X)\delta_\beta E^\alpha Y_i\delta_\alpha Y_j'(\cdot-\beta)\right|\\
				&\quad\quad\quad\quad\quad\quad~~=|\delta_\beta H(\tilde\Delta_\alpha X) |\left|E^\alpha Y(\cdot-\beta)\right|\left|\delta_\alpha Y'(\cdot-\beta)\right|,\\
				&|\tilde{\mathbf{R}}_2-\mathbf{P}_3|=\left|	\delta_\beta\mathbf{R}_2-\sum(H(\tilde\Delta_\alpha X)-H(\tilde\Delta_\alpha Z))E^\alpha Z_i \delta_\beta \delta_\alpha Y_j'\right|\\
				&\leq \left(|H(\tilde\Delta_\alpha X)-H(\tilde\Delta_\alpha Z)||\delta_\beta E^\alpha Z|+|\delta_\beta (H(\tilde\Delta_\alpha X)-H(\tilde\Delta_\alpha Z))||E^\alpha Z(\cdot-\beta)|\right)|\delta_\alpha Y'(\cdot-\beta)|,\\
				&	|\tilde{\mathbf{R}}_3-\mathbf{P}_4|=\left|	\delta_\beta\mathbf{R}_3-\sum (H(\tilde\Delta_\alpha X)-H(\tilde\Delta_\alpha Z))\delta_\beta E^\alpha Y_i \delta_\alpha Z_j'\right|\\
				&\lesssim\left(|H(\tilde\Delta_\alpha X)-H(\tilde\Delta_\alpha Z)||\delta_\beta\delta_\alpha Z'|+|\delta_\beta(H(\tilde\Delta_\alpha X)-H(\tilde\Delta_\alpha Z))||\delta_\alpha Z'(\cdot-\beta)|\right)|E^\alpha Y(\cdot-\beta)|,
					\\&|\tilde{\mathbf{R}}_4|=\left|	\delta_\beta\mathbf{R}_4\right|\lesssim |H(\tilde\Delta_\alpha X)-H(\tilde\Delta_\alpha Z)-\mathfrak{D}\tilde H[Z]Y|\left(|E^\alpha Z_i|| \delta_\beta\delta_\alpha Z_j'|+|\delta_\beta E^\alpha Z_i|| \delta_\alpha Z_j'(\cdot-\beta)|\right)\\
				&\quad\quad+ |\delta_\beta(H(\tilde\Delta_\alpha X)-H(\tilde\Delta_\alpha Z)-\mathfrak{D}\tilde H[Z]Y)||(E^\alpha Z_i \delta_\alpha Z_j')(\cdot-\beta)|.
			\end{align*}
			Moreover, by Lemma \ref{leH}  we have 
			\begin{align*}
				&\|\delta_\beta H(\tilde\Delta_\alpha X) \|_{L^\infty}\overset{\eqref{H1}}\lesssim \kappa(t)^2\|\delta_\beta\tilde\Delta_\alpha X\|_{L^\infty},\\
				&\|H(\tilde\Delta_\alpha X)-H(\tilde\Delta_\alpha Z)-\mathfrak{D}\tilde H[Z]Y\|_{L^\infty}\overset{\eqref{H3}}\lesssim \left(1+\kappa(t)+\|Z'(t)\|^{-1}_{L^\infty}\right)^5\|\tilde\Delta_\alpha Y\|_{L^\infty}^2,\\
				&\|\delta_\beta (H(\tilde\Delta_\alpha X)-H(\tilde\Delta_\alpha Z))\|_{L^\infty} \overset{\eqref{H2}}\lesssim\left(\|\delta_\beta \tilde\Delta_\alpha Y\|_{L^\infty}+\|\tilde\Delta_\alpha Y\|_{L^\infty}(\|\delta_\beta \tilde\Delta_\alpha Z\|_{L^\infty}+\|\delta_\beta \tilde\Delta_\alpha Y\|_{L^\infty})\right)\\
				&\quad\quad\quad\quad\quad\quad\times\left(1+\kappa(t)+\|Z'(t)\|^{-1}_{L^\infty}\right)^5,\\
				&\|\delta_\beta(H(\tilde\Delta_\alpha X)-H(\tilde\Delta_\alpha Z)-\mathfrak{D}\tilde H[Z]Y)\|_{L^\infty}\\
				&\quad\quad\quad\overset{\eqref{H4}}\lesssim \left(\|\delta_\beta \tilde\Delta_\alpha Y\|_{L^\infty}+\|\tilde\Delta_\alpha Y\|_{L^\infty}(\|\delta_\beta \tilde\Delta_\alpha Z\|_{L^\infty}+\|\delta_\beta \tilde\Delta_\alpha Y\|_{L^\infty})\right)\\
				&\quad\quad\quad\quad\quad\quad\times\|\tilde\Delta_\alpha Y\|_{L^\infty}\left(1+\kappa(t)+\|Z'(t)\|^{-1}_{L^\infty}\right)^5.
			\end{align*}
			From the above estimates we obtain
			\begin{align*}
				&|\tilde{\mathbf{R}}_1-\mathbf{P}_1-\mathbf{P}_2|\lesssim \kappa(t)^2\|\delta_\beta \tilde\Delta_\alpha X\|_{L^\infty}\|E^\alpha Y\|_{L^\infty}\|\delta_\alpha Y'\|_{L^\infty},\\
				&|\tilde{\mathbf{R}}_2-\mathbf{P}_3|+|\tilde{\mathbf{R}}_3-\mathbf{P}_4|\lesssim\bigg\{\|\tilde\Delta_\alpha Y\|_{L^\infty}\left(\|\delta_\beta E^\alpha Z\|_{L^\infty}\|\delta_\alpha Y'\|_{L^\infty}+\| E^\alpha Y\|_{L^\infty}\|\delta_\beta\delta_\alpha Z'\|_{L^\infty}\right)\\
				&\quad+\left(\|\delta_\beta \tilde\Delta_\alpha Y\|_{L^\infty}(1+\|\tilde\Delta_\alpha Y\|_{L^\infty})+\|\tilde\Delta_\alpha Y\|_{L^\infty}\|\delta_\beta \tilde\Delta_\alpha Z\|_{L^\infty}\right)\\
				&\quad\quad\times\left(\|E^\alpha Z\|_{L^\infty}\|\delta_\alpha Y'\|_{L^\infty}+\|E^\alpha Y\|_{L^\infty}\|\delta_\alpha Z'\|_{L^\infty}\right)\bigg\}\left(1+\kappa(t)+\|Z'(t)\|^{-1}_{L^\infty}\right)^5,\\
				&|\tilde{\mathbf{R}}_4|\lesssim\left(1+\kappa(t)+\|Z'(t)\|^{-1}_{L^\infty}\right)^5\left\{\|\tilde\Delta_\alpha Y\|_{L^\infty}\| E^\alpha Z\|_{L^\infty}\|\delta_\alpha Z'\|_{L^\infty}\right.\\
				&\quad\quad\quad\quad\quad\quad\quad\times\left(\|\delta_\beta \tilde\Delta_\alpha Y\|_{L^\infty}(1+\|\tilde\Delta_\alpha Y\|_{L^\infty})+\|\tilde\Delta_\alpha Y\|_{L^\infty}\|\delta_\beta \tilde\Delta_\alpha Z\|_{L^\infty}\right)\\
				&\quad\quad\quad\quad+\left.\|\tilde\Delta_\alpha Y\|_{L^\infty}^2(\| E^\alpha Z\|_{L^\infty}\|\delta_\beta\delta_\alpha Z'\|_{L^\infty}+\|\delta_\beta E^\alpha Z\|_{L^\infty}\|\delta_\alpha Z'\|_{L^\infty})\right\}.
			\end{align*}
			Note that 
			$
			\int (\tilde{\mathbf{R}}_1-\mathbf{P}_1-\mathbf{P}_2)\frac{d\alpha}{\alpha}=\int_{-\pi}^{\pi} (\tilde{\mathbf{R}}_1-\mathbf{P}_1-\mathbf{P}_2)\frac{d\alpha}{\tilde\alpha}.$
			Hence 
			\begin{align*}
				\left|\int \left(\tilde{\mathbf{R}}_1-\mathbf{P}_1-\mathbf{P}_2\right)\frac{d\alpha}{\alpha}\right|&\lesssim \kappa(t)^2\int_{-\pi}^\pi\|\delta_\beta \delta_\alpha X\|_{L^\infty}\|E^\alpha Y\|_{L^\infty}\|\delta_\alpha Y'\|_{L^\infty}\frac{d\alpha}{|\tilde\alpha|^2}\\
				&\lesssim \kappa(t)^2\int\|\delta_\beta \delta_\alpha X\|_{L^\infty}\|E^\alpha Y\|_{L^\infty}\|\delta_\alpha Y'\|_{L^\infty}\frac{d\alpha}{|\alpha|^2}.
			\end{align*}
			By \eqref{12haha} we obtain
			$$	\left|\int \left(\tilde{\mathbf{R}}_1-\mathbf{P}_1-\mathbf{P}_2\right)\frac{d\alpha}{\alpha}\right|\lesssim \kappa(t)^2\|\delta_\beta \Lambda^\frac{3}{4} X\|_{L^\infty}\int \|E^\alpha Y\|_{L^\infty}\|\delta_\alpha Y'\|_{L^\infty}\frac{d\alpha }{|\alpha|^\frac{5}{4}}.$$
			Applying H\"{o}lder's inequality and Minkowski inequality  we obtain
			\begin{align}
				&\left\|t^\mu \int \left(\tilde{\mathbf{R}}_1-\mathbf{P}_1-\mathbf{P}_2\right)\frac{d\alpha}{\alpha}\right\|_{L^\frac{1}{a}_TL^\infty}\label{errorter}\\
				&\lesssim\kappa(T)^2\int\min_{+,-}\left\{\|t^\mu\delta_\beta\Lambda^\frac{3}{4} X\|_{L_T^\frac{4}{4a-1\pm\varepsilon'}L^\infty}\|E^\alpha Y\delta_\alpha Y'\|_{L_T^\frac{4}{1\mp\varepsilon'}L^\infty}\right\}\frac{d\alpha}{|\alpha|^\frac{5}{4}}\nonumber\\
				&\lesssim|\beta|^{\mu+a} \kappa(T)^2\|Y'\|_{\mathcal{G}_T}^2(\|Y'\|_{\mathcal{G}_T}+\|Z'\|_{L^\infty_TL^\infty}),\nonumber
			\end{align}
			where in the last inequality we follow the estimates of $J_3$ in Proposition \ref{propdeltaN}. We also use the fact that $\|X'\|_{\mathcal{G}_T}\lesssim\|Y'\|_{\mathcal{G}_T}+\|Z'\|_{L^\infty_TL^\infty} $. 
			Then we estimate $	\left|\int(\tilde{\mathbf{R}}_2-\mathbf{P}_3)+(\tilde{\mathbf{R}}_3-\mathbf{P}_4)+\tilde{\mathbf{R}}_4\frac{d\alpha}{\alpha}\right|$. From the above discussion, it suffices to consider $\alpha\in(-\pi,\pi)$. Then $\|\tilde\Delta_\alpha f\|_{L^\infty}\lesssim \|\Delta _\alpha f\|_{L^\infty}$ for any  function $f$.
			Note that $Y$ is periodic,  by Lemma \ref{holder} we have for any $0<\gamma_1<\gamma_2$
			$$
			\|Y\|_{\dot C^{\gamma_1}}\lesssim \|Y\|_{\dot C^{\gamma_2}}.$$
			Moreover, we have $Z(t,\cdot)\in\mathcal{V}$, hence for any $\gamma>0$,
			\begin{align}\label{zzzzz}
				\|Z\|_{\dot C^{\gamma}}=c_\gamma \|Z'\|_{L^\infty}.
			\end{align}
			Hence 
			$$
			\int\|\delta_\beta \Delta_\alpha Y\|_{L^\infty}\|E^\alpha Z\|_{L^\infty}\|\delta_\alpha Y'\|_{L^\infty}\frac{d\alpha}{|\alpha|}\lesssim|\beta|^{a+\mu} \|Y'\|_{L^\infty}^2\|Z'\|_{L^\infty}.$$
			Other terms can be estimated similarly, we conclude that
			\begin{align*}
				&\left|\int\left((\tilde{\mathbf{R}}_2-\mathbf{P}_3)+(\tilde{\mathbf{R}}_3-\mathbf{P}_4)+\tilde{\mathbf{R}}_4\right)\frac{d\alpha}{\alpha}\right|\\
				&\lesssim |\beta|^{a+\mu} \left(1+\kappa(t)+\|Z'(t)\|_{L^\infty}^{-1}\right)^5\left(1+\|Z'(t)\|_{L^\infty}+\|Y'(t)\|_{L^\infty}\right)^3\|Y'(t)\|_{L^\infty}^2.
			\end{align*}
			By definition, it is easy to check that for any $T\in[0,1]$, $\gamma_1\geq 0$, and $ 2\varepsilon'\leq\gamma_2\leq\theta-\varepsilon'$, there holds \begin{align}\label{hhhh}\|Y'\|_{L^\frac{1}{\gamma_2}_T\dot C^{\gamma_1}}\lesssim\|Y'\|_{\mathcal{G}_T}.
			\end{align}
			Hence by H\"{o}lder's inequality we obtain
			\begin{align}
				\label{P5}	&\left\|t^\mu \int\left((\tilde{\mathbf{R}}_2-\mathbf{P}_3)+(\tilde{\mathbf{R}}_3-\mathbf{P}_4)+\tilde{\mathbf{R}}_4\right)\frac{d\alpha}{\alpha}\right\|_{L^\frac{1}{a}_TL^\infty}\\
				&\lesssim |\beta|^{a+\mu} \Big(1+\kappa(T)+\sup_{t\in[0,T]}(\|Z'(t)\|_{L^\infty}^{-1})\Big)^5\left(1+\|Z'\|_{L^\infty_TL^\infty}+\|Y'\|_{\mathcal{G}_T}\right)^3\|Y'\|_{\mathcal{G}_T}^2.\nonumber
			\end{align}
			It remains to estimate the main terms $P_k=\int\mathbf{P}_k\frac{d\alpha}{\alpha}$, $k=1,2,3,4$. We first estimate
			\begin{align*}
				&	P_1=-\int H(\tilde\Delta_\alpha X)E^\alpha Y_i\delta_\beta Y_j'(\cdot-\alpha)\frac{d\alpha}{\alpha},\quad P_2=\int H(\tilde\Delta_\alpha X)\delta_\beta E^\alpha Y_i\delta_\alpha Y_j'(\cdot-\beta)\frac{d\alpha}{\alpha}.
			\end{align*}
			Note that \begin{align*}
				&\delta_\beta Y'(s-\alpha)=\partial_\alpha(\alpha\Delta_\alpha \delta_\beta Y(s)),\\
				& \delta_\beta E^\alpha Y(s)= \alpha \partial_\alpha\left(\Delta_\alpha\delta_\beta Y(s)\right)+\delta_\alpha\delta_\beta Y(s)\left(\frac{1}{\alpha}-\frac{1}{\tilde\alpha}\right),
			\end{align*} hence one has 
			\begin{align*}
				&|P_1|\lesssim \left|\int H(\tilde\Delta_\alpha X)E^\alpha Y_i\partial_\alpha(\Delta_\alpha \delta_\beta Y_j)d\alpha\right|+\int|H(\tilde\Delta_\alpha X)E^\alpha Y_i\delta_\alpha\delta_\beta Y_j|\frac{d\alpha}{|\alpha|^2},\\
				&|P_2|\lesssim \left|\int H(\tilde\Delta_\alpha X)\partial_\alpha\left(\Delta_\alpha\delta_\beta Y_i(s)\right)\delta_\alpha Y_j'(\cdot-\beta)d\alpha\right|+\int\left| H(\tilde\Delta_\alpha X)\delta_\alpha\delta_\beta Y_i(s)\delta_\alpha Y_j'(\cdot-\beta)\right|\frac{d\alpha}{|\alpha|^2}.
			\end{align*}
			Here the last inequality follows from \eqref{z2}. 
			To estimate the above terms, we can follow the estimates of $J_{1}$ in Proposition \ref{propdeltaN}. 
			We conclude that 
			\begin{equation}\label{P13}
				\|t^\mu{P}_{1}\|_{L^\frac{1}{a}_TL^\infty}+	\|t^\mu{P}_{2}\|_{L^\frac{1}{a}_TL^\infty}\lesssim |\beta|^{\mu+a}\left(1+\kappa(T)\right)^2\|Y'\|_{\mathcal{G}_T}^2(1+\|Y'\|_{\mathcal{G}_T}+\|Z'\|_{L^\infty_TL^\infty}).
			\end{equation}
			Then we estimate
			$$
			{P}_{3}=-\int [H(\tilde\Delta_\alpha X)-H(\tilde\Delta_\alpha Z)] E^\alpha Z_i\delta_\beta Y_j'(\cdot-\alpha)\frac{d\alpha}{ \alpha}.$$
			Note that $Y'(s-\alpha)=\partial_\alpha ( \alpha \Delta_\alpha Y(s))$. Hence 
			\begin{align*}
				|P_{3}|\lesssim& \left|\int [H(\tilde\Delta_\alpha X(s)) -H(\tilde\Delta_\alpha Z(s))] E^\alpha Z_i(s)\partial_\alpha ( \Delta_\alpha \delta_\beta Y_j(s))d\alpha\right|\\
				&+\int\left| [H(\tilde\Delta_\alpha X(s)) -H(\tilde\Delta_\alpha Z(s))] E^\alpha Z_i(s) \delta_\alpha \delta_\beta Y_j(s)\right|\frac{d\alpha}{|\alpha|^2}\\
				=&P_{3,1}+P_{3,2}.
			\end{align*}
			Let $f_s(t,\alpha)=[H(\tilde\Delta_\alpha X(s)) -H(\tilde\Delta_\alpha Z(s))] E^\alpha Z_i(s)$, $g(t,s)=t^\mu \delta_\beta Y_j(t,s)$. Then by \eqref{H2} there holds
				\begin{align}\nonumber
					\sup_s|\delta_zf_s(t,\cdot)(\alpha)|&\lesssim \Big\{\left(\|\tilde\Delta_\alpha Y-\tilde\Delta_{\alpha-z} Y\|_{L^\infty}(1+\|Y'\|_{L^\infty})+\| Y'\|_{L^\infty}\|\tilde\Delta_\alpha Z-\tilde\Delta_{\alpha-z} Z\|_{L^\infty}\right)\|E^\alpha Z\|_{L^\infty}\\
					&\quad\quad+\| Y'\|_{L^\infty}\|\delta_z Z'\|_{L^\infty}\Big\}
					\times\left(1+\kappa(t)+\|Z'(t)\|_{L^\infty}^{-1}\right)^5,\label{deltazf}
				\end{align}
			where we also use the fact that $\|\tilde \Delta_\alpha Y\|_{L^\infty}\lesssim \|Y'\|_{L^\infty}$. 
			For any $p, q$ such that $4\varepsilon'\leq\frac{1}{p}\leq \theta-\varepsilon'$ and   $2\varepsilon'\leq\frac{1}{q}+\mu\leq\frac{3}{4}-\frac{\varepsilon'}{2}, q\leq 30$, denote
			\begin{align*}
				&\tilde V_{q,\mu}(\beta)=\|t^\mu\delta_\beta \Lambda^{\frac{3+\varepsilon'}{4}}Y\|^\frac{1}{2}_{L^q_TL^\infty}\|t^\mu\delta_\beta \Lambda^{\frac{3-\varepsilon'}{4}}Y\|^\frac{1}{2}_{L^q_TL^\infty},
				\\&
				\tilde U_p=
				\sup_{\alpha,z}\frac{\left\|\|\tilde\Delta_\alpha Y-\tilde\Delta_{\alpha-z} Y\|_{L^\infty}\|E^\alpha Z\|_{L^\infty}\right\|_{L^{p}_T}}{|z|^\frac{2}{p}}.
			\end{align*}
			By  Remark \ref{Rem1}, it is easy to check that 
			\begin{align}\label{tildeV}
				\tilde V_{q,\mu}(\beta)&\lesssim|\beta|^{\frac{1}{q}+\frac{1}{4}+\mu} \|Y'\|_{\mathcal{G}_T}.
			\end{align}
			By 	Lemma \ref{lem2} and \eqref{zzzzz}we have
			\begin{align}\label{tildeU}
				\tilde U_p\lesssim\sup_\alpha \frac{\|\delta_\alpha Y'\|_{L^{p}_TL^\infty}}{|\alpha|^\frac{1}{p}}\sup_\alpha \frac{\|\delta_\alpha Z'\|_{L^\infty_TL^\infty}}{|\alpha|^\frac{1}{p}}\lesssim \|Y'\|_{\mathcal{G}_T}\|Z'\|_{L^\infty_T L^\infty}.
			\end{align}
			Applying Remark \ref{remmovedeT} with above $f_s$, $g$ and parameters $\sigma=\frac{1}{4}$, $p=\frac{1}{a}$, $p_{1\pm}=\frac{4}{1\pm\varepsilon'}$, $p_{2\pm}=\frac{4}{4a\pm\varepsilon'-1}$, we have 
			\begin{align*}
				&\|t^\mu P_{3,1}\|_{L^\frac{1}{a}_TL^\infty}\lesssim\sup_{\alpha,s}\int\min_{+,-}\left\{ \|\delta_zf_s(t,\cdot)(\alpha)\|_{L^\frac{4}{1\pm\varepsilon'}_T}\tilde V_{\frac{4}{4a\mp\varepsilon'-1},\mu}(\beta)\right\}\frac{dz}{|z|^\frac{5}{4}}.
			\end{align*}
			By Lemma \ref{lem2} and \eqref{zzzzz}, it is easy to check that for any $p>1$
			\begin{align*}
				\int\left(\|\tilde\Delta_\alpha Z-\tilde\Delta_{\alpha-z} Z\|_{L^p_TL^\infty}+\|\delta_z Z'\|_{L^p_TL^\infty}\right)\frac{dz}{|z|^\frac{5}{4}}\lesssim \|Z'\|_{L^\infty_TL^\infty}.
			\end{align*}
			Combining this with \eqref{hhhh} one has
			\begin{align*}
				M_1^\pm:&=\left\|\left(\| Y'\|_{L^\infty}\|\tilde\Delta_\alpha Z-\tilde\Delta_{\alpha-z} Z\|_{L^\infty}\|E^\alpha Z\|_{L^\infty}+\| Y'\|_{L^\infty}\|\delta_z Z'\|_{L^\infty}\right)\right\|_{L_T^\frac{4}{1\pm\varepsilon'}}\\
				&\lesssim  |z|^\frac{1\pm\varepsilon'}{4} \|Y'\|_{\mathcal{G}_T}(1+\|Z'\|_{L^\infty_TL^\infty})^2.
			\end{align*}
			Moreover, by H\"{o}lder's inequality, \eqref{tildeU} and \eqref{hhhh} we have 
			\begin{align*}
				M_2^\pm:&=\sup_\alpha \left\|\|\tilde\Delta_\alpha Y-\tilde\Delta_{\alpha-z} Y\|_{L^\infty}(1+\|Y'\|_{L^\infty})\|E^\alpha Z\|_{L^\infty}\right\|_{L_T^\frac{4}{1\pm\varepsilon'}}\\
				&\lesssim |z|^\frac{1\pm\varepsilon'}{4}\tilde U_{\frac{8}{1\pm\varepsilon'}}(1+\|Y'\|_{L_T^\frac{8}{1\mp\varepsilon'}L^\infty})\lesssim |z|^\frac{1\pm\varepsilon'}{4}\|Y'\|_{\mathcal{G}_T}\|Z'\|_{L^\infty_T L^\infty}(1+\|Y'\|_{\mathcal{G}_T}).
			\end{align*}
			By \eqref{deltazf} we have 
			$$
			\sup_{\alpha,s} \|\delta_zf_s(t,\cdot)(\alpha)\|_{L^\frac{4}{1\pm\varepsilon'}_T}\lesssim (1+\kappa(T)+\sup_{t\in[0,T]}(\|Z'(t)\|_{L^\infty}^{-1}))^5(M_1^\pm+	M_2^\pm).	
			$$
			Hence we obtain
			\begin{align*}
				&\|t^\mu P_{3,1}\|_{L^\frac{1}{a}_TL^\infty}\lesssim \int \min_{+,-}\{(M_1^\pm+M_2^\pm)\tilde V_{\frac{4}{4a-1\mp\varepsilon'},\mu}(\beta)\}\frac{dz}{|z|^\frac{5}{4}}(1+\kappa(T)+\sup_{t\in[0,T]}(\|Z'(t)\|_{L^\infty}^{-1}))^5\\
				&	\lesssim\left( \int\min_{+,-}\left\{|z|^{\pm\varepsilon'}\tilde V_{\frac{4}{4a-1\mp\varepsilon'},\mu}(\beta)\right\}\frac{dz}{|z|}\right)(1+\|Y'\|_{\mathcal{G}_T})(1+\|Z'\|_{L^\infty_TL^\infty})^2\|Y'\|_{\mathcal{G}_T}\\
				&\quad\quad\quad\times(1+\kappa(T)+\sup_{t\in[0,T]}(\|Z'(t)\|_{L^\infty}^{-1}))^5\\
				&\overset{\eqref{tildeV}}\lesssim |\beta|^{\mu+a}(1+\kappa(T)+\sup_{t\in[0,T]}(\|Z'(t)\|_{L^\infty}^{-1}))^5(1+\|Z'\|_{L^\infty_TL^\infty}+\|Y'\|_{\mathcal{G}_T})^3\|Y'\|_{\mathcal{G}_T}^2.
			\end{align*}
			Then we estimate $P_{3,2}$. By \eqref{H1} we have 
			\begin{align*}
				\|H(\tilde\Delta_\alpha X(t))-H(\tilde\Delta_\alpha Z(t))\|_{L^\infty}&\lesssim \|\tilde\Delta_\alpha Y(t)\|_{L^\infty}(\kappa(t)+\|Z'(t)\|_{L^\infty}^{-1})^2\lesssim \|Y'(t)\|_{L^\infty}(\kappa(t)+\|Z'(t)\|_{L^\infty}^{-1})^2.
			\end{align*}
			Moreover, one has
$$
				\int\|E^\alpha Z\|_{L^\infty}\|\delta_\alpha\delta_\beta Y\|_{L^\infty}\frac{d\alpha}{|\alpha|^2}\overset{\eqref{zzzzz}}\lesssim |\beta|^{\mu+a} \| Y'\|_{L^\infty}\|Z'\|_{L^\infty}.$$
			Then 
			\begin{align*}
				\|t^\mu P_{3,2}\|_{L^\frac{1}{a}_TL^\infty}&\lesssim|\beta|^{\mu+a} (\kappa(T)+\sup_{t\in[0,T]}(\|Z'(t)\|_{L^\infty}^{-1}))^2\|Z'\|_{L^\infty_TL^\infty}\|Y'\|_{L^\frac{1}{2a}_TL^\infty}^2\\
				&\overset{\eqref{hhhh}}\lesssim|\beta|^{\mu+a} (\kappa(T)+\sup_{t\in[0,T]}(\|Z'(t)\|_{L^\infty}^{-1}))^2\|Z'\|_{L^\infty_TL^\infty}\|Y'\|_{\mathcal{G}_T}^2.
			\end{align*}
			Hence we conclude that 
			\begin{align}\label{P2}
				\|t^\mu P_3\|_{L^\frac{1}{a}_TL^\infty}\lesssim|\beta|^{\mu+a} (1+\|Z'\|_{L^\infty_TL^\infty}+\|Y'\|_{\mathcal{G}_T})^3\|Y'\|_{\mathcal{G}_T}^2\Big(1+\kappa(T)+\sup_{t\in[0,T]}(\|Z'(t)\|_{L^\infty}^{-1})\Big)^{5}.
			\end{align}
			Note that $P_4$ can be estimated similarly as $P_2$. Combining \eqref{errorter}, \eqref{P5}, \eqref{P13} and \eqref{P2}, we obtain the result.
		\end{proof}
		\vspace{0.3cm}\\
		Recall the definition \eqref{defQ}, denote $Q(t)=Q_X(t)$. We have the following results.
		\begin{proposition} \label{propQT}
			Let $X\in \mathcal{G}_T^1$ be a solution of \eqref{peskin} on $[0,T]$, there holds
			\begin{align*}
				Q(T)\lesssim (1+\kappa(T))^4\|X'\|_{\mathcal{G}_T}(1+\|X'\|_{\mathcal{G}_T})^2.
			\end{align*}
		\end{proposition}
		\begin{proof}
			Note that 
			\begin{align*}
				\frac{1}{|\Delta_\alpha X(t,\cdot)(s)|}-\frac{1}{|\Delta_\alpha X_{0}(s)|}&=\int_0^t \partial_t\left(\frac{1}{|\Delta_\alpha X(\tau,\cdot)(s)|}\right)d\tau\\
				&\leq \kappa(t)^2 \int_0^t |\Delta_\alpha \Lambda X(\tau,\cdot)(s)|+|\Delta_\alpha N(X(\tau,\cdot))(s)|d\tau.
			\end{align*}
			By H\"{o}lder's inequality we obtain
			\begin{align*}
				\sup_{\alpha,t}\frac{|\alpha|^{\varepsilon'}}{t^{\varepsilon'}}\int_0^t |\Delta_\alpha \Lambda X(\tau,\cdot)(s)|d\tau\lesssim \sup_{\alpha}\frac{\|\delta_\alpha \Lambda X\|_{L_T^\frac{1}{1-\varepsilon'}L^\infty}}{|\alpha|^{1-\varepsilon'}}\lesssim
				\|X'\|_{\mathcal{G}_T}.
			\end{align*}
			Moreover, Proposition \ref{propdeltaN} implies
			\begin{align*}
				\sup_{\alpha,t}\frac{|\alpha|^{\varepsilon'}}{t^{\varepsilon'}}\int_0^t|\Delta_\alpha N(X(\tau,\cdot))(s)|d\tau&\lesssim\sup_{\alpha}\frac{\|\delta_\alpha N\|_{L_T^\frac{1}{1-\varepsilon'}L^\infty}}{|\alpha|^{1-\varepsilon'}}\lesssim\left(1+\kappa(T)\right)^2\|X'\|_{\mathcal{G}_T}^2(1+\|X'\|_{\mathcal{G}_T}).
			\end{align*}
			Hence one obtain 
			\begin{align*}
				Q(T)\lesssim (1+\kappa(T))^4\|X'\|_{\mathcal{G}_T}(1+\|X'\|_{\mathcal{G}_T})^2.
			\end{align*}
			This completes the proof.
		\end{proof}
		
		\subsection{Smoothing effect}
		Let $X\in \mathcal{G}_T^1$ be a solution to \eqref{eqpeskin}, $(Y,Z)\in \mathcal{G}_T^1\times C^2_{t,x}$ be a solution to \eqref{eqglo}. We prove that for any  $t\in(0,1)$ and $\gamma\in[10\varepsilon',\theta-10\varepsilon']$, there holds
		\begin{align}\label{smoresultX}
			&\|X(t)\|_{\dot C^{1+\gamma}}\lesssim t^{-\gamma}\left( \|X_0'\|_{\tilde {\mathcal{G}}_t}+(1+\kappa(t))^2 \|X'\|_{\mathcal{G}_t}^2(1+\|X'\|_{\mathcal{G}_t})^2\right),\\
			&\|Y(t)\|_{\dot C^{1+\gamma}}\lesssim t^{-\gamma}\Big\{\big(1+ \|Z'\|_{L^\infty_tL^\infty}+\|Y'\|_{\mathcal{G}_t}\big)^{5}\|Y'\|_{\mathcal{G}_t}^2\label{smoresultY}\\
			&\quad\quad\quad\quad\quad\quad\quad\quad\times(1+\kappa(t)+\sup_{\tau\in[0,t]}(\|Z'(\tau)\|_{L^\infty}^{-1}))^{5}+\|Y_0'\|_{\tilde {\mathcal{G}}_t}\Big\}.\nonumber
		\end{align}
		Denote $X'(t,s)=\partial_s X(t,s)$. We have the formula
		\begin{align*}
			X'(t,s)&=\int K(t,s-y)X_0'(y)dy+\int_0^t\int \mathcal{H}\Lambda^{1-\theta}K(t-\tau,s-y) \Lambda^{\theta} N (\tau,y)dyd\tau\\
			&:=X'_L(t,s)+X'_N(t,s).
		\end{align*}
		We write 
		$$
		X'_L(t,s)=\int K(t-\tau,s-y)X_L'(\tau,y)dy, \ \ \tau\in(0,\frac{t}{2}).
		$$
		By $\|K(t-\tau,\cdot)\|_{L^1}=1$, we obtain 
		$$
		\|\delta_\alpha X'_L(t,\cdot)\|_{L^\infty}\lesssim \|\delta_\alpha X_L'(\tau,\cdot)\|_{L^\infty}.
		$$
		Take $L^{\frac{1}{\gamma}}$ for $\tau\in(0,\frac{t}{2})$ we obtain 
		$$
		t^{\gamma}\|\delta_\alpha X'_L(t,\cdot)\|_{L^\infty}\lesssim \|\delta_\alpha X_L'\|_{L^{1/\gamma}_{t/2}L^\infty},
		$$
		which leads to 
		$$
		t^{\gamma}\|X'_L(t,\cdot)\|_{\dot C^\gamma}\lesssim \sup_\alpha\frac{\|\delta_\alpha X_L'\|_{L_{t/2}^{1/\gamma}L^\infty}}{|\alpha|^\gamma}\lesssim \|X_L'\|_{\mathcal{G}_{t}}\lesssim \|X_0'\|_{\tilde {\mathcal{G}_t}}.
		$$
		On the other hand, we have 
		$$
		\|X_N'(t)\|_{\dot C^\gamma}\overset{\eqref{fractionald}}\lesssim \int_0^t \frac{1}{(t-\tau)^{\gamma+1-\theta}}\|\Lambda^\theta N(\tau)\|_{L^\infty}d\tau.
		$$
		Denote $\tilde \gamma=\gamma+1-\theta$, we have 
		\begin{align*}
			t^{\tilde \gamma}\|X(t)\|_{\dot C^\gamma}\lesssim t^{1-\theta}\|X_0'\|_{\tilde {\mathcal{G}_t}}+\int_0^t\frac{\tau^{\tilde \gamma}+(t-\tau)^{\tilde \gamma}}{(t-\tau)^{\tilde \gamma} }\int\|\delta_\beta N (\tau,\cdot)\|_{L^\infty}\frac{d\tau d\beta}{|\beta|^{1+\theta}}.
		\end{align*}
		Combining Lemma \ref{lempropsmoothe} with Proposition \ref{propdeltaN}  we obtain 
	$$\int_0^t\frac{\tau^{\tilde \gamma}+(t-\tau)^{\tilde \gamma}}{(t-\tau)^{\tilde \gamma} }\int\|\delta_\beta N (\tau,\cdot)\|_{L^\infty}\frac{d\beta d\tau}{|\beta|^{1+\theta}}\lesssim t^{1-\theta} \left(1+\kappa(t)\right)^2\|X'\|_{\mathcal{G}_t}^2(1+\|X'\|_{\mathcal{G}_t})^2.$$
		Then one obtains $$
		t^{\tilde \gamma}\|X(t)\|_{\dot C^\gamma }\lesssim t^{1-\theta}\left(\|X_0'\|_{\tilde {\mathcal{G}_t}}+(1+\kappa(t))^2 \|X'\|_{\mathcal{G}_t}^2(1+\|X'\|_{\mathcal{G}_t})^2\right),
		$$
		which yields \eqref{smoresultX}. 
		Similarly, recalling \eqref{formulaY} we have the formula 
		\begin{align*}
			Y'(t,s)&=\int K(t,s-y)\tilde {\mathcal{O}}(s-y)Y_0'(y)dy\\
			&\quad\quad+\int_0^t\int \mathcal{H}\Lambda^{1-\theta}K(t-\tau,s-y)\tilde {\mathcal{O}}(s-y) \Lambda^{\theta} (\Pi\mathfrak{N}(X(\tau,y))) dyd\tau.
		\end{align*}
		Following above estimates, 
		Lemma \ref{lempropsmoothe} and Proposition \ref{propparabolicY} yield \eqref{smoresultY}.
		\subsection{Higher regularity}
		In the following lemma, we suppress the time variable. 
		\begin{lemma}\label{lemhigh}
			Suppose $\kappa(X)\leq C_*$, then	 we have for any $m\in\mathbb{Z}^+$\\
			\textbf{1)} If $X'\in {\dot C^{m-\frac{1}{4}}}$, then 
			$N\in\dot C^{m+\frac{1}{2}}$. In particular,		\begin{align}
				&\|\partial_s^mN\|_{\dot C^\frac{1}{2}}\lesssim \|X'\|_{\dot C^{\frac{1}{4}}}^{4m+2}+\|X'\|_{\dot C^{m-\frac{1}{4}}}^\frac{4m+2}{4m-1}, \ \ \  \label{firh}
			\end{align}
			\textbf{2)} If $X'\in{\dot C^{m}}$, then 
			$N\in{\dot C^{m+\frac{7}{8}}}$. In particular,					
			\begin{align*}
				&\|\partial_s^mN\|_{\dot C^\frac{7}{8}}\lesssim \|X'\|_{\dot C^{\frac{1}{8}}}^{8m+7}+\|X'\|_{\dot C^{m}}^\frac{8m+7}{8m}, \ \ \  
			\end{align*}
			where the implicit constants only depend on $C_*$ and $m$.
		\end{lemma}
		\begin{proof} We note that the proof is essentially an analogy of the proof of Proposition \ref{propdeltaN}. The main difference is that we ignore the time variable in this lemma.  
			
			For simplicity, we only prove \textbf{1)}. The second one can be done similarly. Recall that 
			\begin{equation*}
				N(X(s))=\sum\int H(\tilde\Delta_\alpha X(s)) E^\alpha X_i(s)  \delta_\alpha X_j'(s)\frac{d\alpha}{\alpha}.
			\end{equation*}
			We have 
			\begin{align*}
				\delta_\beta\partial_s^mN(X(s))=&\sum\int \delta_\beta\partial_s^m(H(\tilde\Delta_\alpha X(s))) E^\alpha X_i(s)  \delta_\alpha X_j'(s)\frac{d\alpha}{\alpha}\\
				&+\sum\int H(\tilde\Delta_\alpha X(s)) \delta_\beta E^\alpha \partial_s^mX_i(s)  \delta_\alpha X_j'(s)\frac{d\alpha}{\alpha}\\
				&+\sum\int H(\tilde\Delta_\alpha X(s)) E^\alpha X_i(s)  \delta_\beta \partial_s^mX_j'(s-\alpha)\frac{d\alpha}{\alpha}+R_0\\
				&:=S_1+S_2+S_3+R_0,
			\end{align*}
			where we denote $R_0$ the lower order remainder terms. For the first term, we have 
			\begin{align*}
				&|\delta_\beta\partial_s^m(H(\tilde\Delta_\alpha X(s)))|\lesssim \sum_{k=1}^m(|\tilde \Delta_\alpha\partial_s X|^{m-k}|\delta_\beta\tilde \Delta_\alpha\partial_s^k X|)\\
				&\lesssim |\beta|^\frac{1}{2}|\alpha|^{-\frac{1}{2}}(\min\{\|X'\|_{\dot C^{1-\frac{1}{4m}}}|\alpha|^{-\frac{1}{4m}},\|X'\|_{\dot C^{1-\frac{3}{4m}}}|\alpha|^{-\frac{3}{4m}}\}^m+\|X'\|_{\dot C^{m-\frac{1}{4}}}|\alpha|^{-\frac{1}{4}}),
			\end{align*} 
			and \begin{align*}
				|E^\alpha X_i(s)  \delta_\alpha X_j'(s)|\lesssim\min\{ \|X'\|_{\dot C^\frac{1}{2}}^2|\alpha|,\|X'\|_{\dot C^\frac{1}{4}}^2|\alpha|^\frac{1}{2}\}.
			\end{align*}
			Hence 
			\begin{align*}
				|S_1|
				&\lesssim|\beta|^\frac{1}{2} \int_{|\alpha|\leq \lambda} (\|X'\|_{\dot C^{1-\frac{1}{4m}}}^m+\|X'\|_{\dot C^{m-\frac{1}{4}}}) \frac{d\alpha}{|\alpha|^\frac{3}{4}}\|X'\|_{\dot C^\frac{1}{2}}^2\\
				&\quad\quad+|\beta|^\frac{1}{2}\int_{|\alpha|\geq \lambda}(\|X'\|_{\dot C^{1-\frac{3}{4m}}}^m+\|X'\|_{\dot C^{m-\frac{3}{4}}})\frac{d\alpha}{|\alpha|^\frac{7}{4}}\|X'\|_{\dot C^\frac{1}{4}}^2\\
				&\lesssim |\beta|^\frac{1}{2}\lambda^\frac{1}{4}(\|X'\|_{\dot C^{1-\frac{1}{4m}}}^m+\|X'\|_{\dot C^{m-\frac{1}{4}}})\|X'\|_{\dot C^\frac{1}{2}}^2\\
				&\quad\quad\quad+|\beta|^\frac{1}{2}\lambda^{-\frac{3}{4}}(\|X'\|_{\dot C^{1-\frac{3}{4m}}}^m+\|X'\|_{\dot C^{m-\frac{3}{4}}})\|X'\|_{\dot C^\frac{1}{4}}^2.
			\end{align*}
			Taking $\lambda=[(\|X'\|_{\dot C^{1-\frac{1}{4m}}}^m+\|X'\|_{\dot C^{m-\frac{1}{4}}})\|X'\|_{\dot C^\frac{1}{2}}^2]^{-1}(\|X'\|_{\dot C^{1-\frac{3}{4m}}}^m+\|X'\|_{\dot C^{m-\frac{3}{4}}})\|X'\|_{\dot C^\frac{1}{4}}^2$, and applying interpolation inequality and Young's inequality, we obtain 
			\begin{align*}
				\|S_1\|_{L^\infty}\lesssim |\beta|^\frac{1}{2}(\|X'\|_{\dot C^{\frac{1}{4}}}^{4m+2}+\|X'\|_{\dot C^{m-\frac{1}{4}}}^\frac{4m+2}{4m-1}).
			\end{align*}
			Then we deal with $S_2$, observe that 
			\begin{align*}
				\delta_\beta E^\alpha\partial_s^m X(s)=\alpha \partial_\alpha(\Delta_\alpha \delta_\beta \partial_s^m X(s))+\delta_\alpha\delta_\beta \partial_s^m X(s)\left(\frac{1}{\alpha}-\frac{1}{\tilde \alpha}\right).
			\end{align*}
			Hence we have
			\begin{align*}
				S_2=&\sum\int H(\tilde\Delta_\alpha X(s)) \partial_\alpha(\Delta_\alpha \delta_\beta \partial_s^m X_i(s)) \delta_\alpha X_j'(s)d\alpha\\
				&+\sum\int H(\tilde\Delta_\alpha X(s)) \delta_\alpha\delta_\beta \partial_s^m X_i(s)\left(\frac{1}{\alpha}-\frac{1}{\tilde \alpha}\right) \delta_\alpha X_j'(s)\frac{d\alpha}{\alpha}\\
				:=&S_{2,1}+S_{2,2}.
			\end{align*}
			Applying Lemma \ref{lemI12} with $f_s(\alpha)=H(\tilde\Delta_\alpha X(s))\delta_\alpha X_j'(s)$, $g(s)=\delta_\beta \partial_s^m X_i(s)$, and $\sigma=\frac{4}{5}$, we have 
			\begin{align*}
				\|S_{2,1}\|_{L^\infty}&\lesssim \|X'\|_{\dot C^{\frac{4}{5}+\varepsilon'}}^\frac{1}{2}\|X'\|_{\dot C^{\frac{4}{5}-\varepsilon'}}^\frac{1}{2}\|\delta_\beta \partial_s^m X_i(s)\|_{\dot C^{\frac{1}{5}+\varepsilon'}}^\frac{1}{2}\|\delta_\beta \partial_s^m X_i(s)\|_{\dot C^{\frac{1}{5}-\varepsilon'}}^\frac{1}{2}\\
				&\lesssim |\beta|^\frac{1}{2}\|X'\|_{\dot C^{\frac{4}{5}+\varepsilon'}}^\frac{1}{2}\|X'\|_{\dot C^{\frac{4}{5}-\varepsilon'}}^\frac{1}{2}\| \partial_s^m X_i(s)\|_{\dot C^{\frac{7}{10}+\varepsilon'}}^\frac{1}{2}\| \partial_s^m X_i(s)\|_{\dot C^{\frac{7}{10}-\varepsilon'}}^\frac{1}{2}.
			\end{align*}
			By interpolation inequality and Young's inequality we obtain 
			\begin{align*}
				\|S_{2,1}\|_{L^\infty}\lesssim |\beta|^\frac{1}{2}(\|X' \|_{\dot C^{\frac{1}{4}}}^{4m+2}+\|X'\|_{\dot C^{m-\frac{1}{4}}}^\frac{4m+2}{4m-1}).
			\end{align*}
			For $S_{2,2}$. Recall \eqref{z2}, one has 
			\begin{align*}
				|S_{2,2}|&\lesssim \int \left|H(\tilde\Delta_\alpha X(s)) \delta_\alpha\delta_\beta \partial_s^m X_i(s) \delta_\alpha X_j'(s)\right|\frac{d\alpha}{|\alpha|^2}\\
				&\lesssim |\beta|^\frac{1}{2}\left(\int_{|\alpha|\leq \lambda} \|X'\|_{\dot C^{m-\frac{1}{4}}}\|X'\|_{\dot C^1}\frac{d\alpha}{|\alpha|^\frac{3}{4}}+\int_{|\alpha|\geq \lambda} \|X'\|_{\dot C^{m-\frac{1}{2}}}\|X'\|_{\dot C^\frac{1}{2}}\frac{d\alpha}{|\alpha|^\frac{3}{2}}\right)\\
				&\lesssim |\beta|^\frac{1}{2}\left(\lambda^\frac{1}{4}\|X'\|_{\dot C^{m-\frac{1}{4}}}\|X'\|_{\dot C^1}+\lambda^{-\frac{1}{2}} \|X'\|_{\dot C^{m-\frac{1}{2}}}\|X'\|_{\dot C^\frac{1}{2}}\right).
			\end{align*}
			Let $\lambda =\left(\|X'\|_{\dot C^{m-\frac{1}{2}}}\|X'\|_{\dot C^\frac{1}{2}}\right)^\frac{4}{3}\left(\|X'\|_{\dot C^{m-\frac{1}{4}}}\|X'\|_{\dot C^1}\right)^{-\frac{4}{3}}$. Applying interpolation inequality and Young's inequality again we obtain 
			\begin{align*}
				\|S_{2,2}\|_{L^\infty}\lesssim |\beta|^\frac{1}{2}(\|X'\|_{\dot C^{\frac{1}{4}}}^{4m+2}+\|X'\|_{\dot C^{m-\frac{1}{4}}}^\frac{4m+2}{4m-1}).
			\end{align*}
			Finally, observe that 
			\begin{align*}
				\delta_\beta \partial_s^mX_j'(s-\alpha)=\alpha \partial_\alpha (\Delta_\alpha 	\delta_\beta \partial_s^mX_j')+\Delta_\alpha \delta_\beta \partial_s^mX_j'.
			\end{align*}
			We can estimate $S_3$ and $R_0$ similarly as we did for $S_2$.  We conclude that 
			\begin{align*}
				\|\delta_\beta\partial_s^mN(X(\cdot))\|_{L^\infty}\lesssim |\beta|^\frac{1}{2}(\|X'\|_{\dot C^{\frac{1}{4}}}^{4m+2}+\|X'\|_{\dot C^{m-\frac{1}{4}}}^\frac{4m+2}{4m-1}),
			\end{align*}
			which leads to \eqref{firh}.
		\end{proof}\vspace{0.3cm}\\	
		In the following, we prove that for any $m\in \mathbb{Z}^+$,
		\begin{align}
			&t^{m-\frac{1}{4}}\|X'(t)\|_{\dot C^{m-\frac{1}{4}}}\lesssim 1 \quad\Rightarrow\quad\quad	t^{m}\|X'(t)\|_{\dot C^m}\lesssim 1,\label{step1}\\
			&t^{m}\|X'(t)\|_{\dot C^m}\lesssim 1 \quad\quad\quad\ \Rightarrow\quad\quad 	t^{m+\frac{3}{4}}\|X'(t)\|_{\dot C^{m+\frac{3}{4}}}\lesssim 1.\label{step2}
		\end{align}
		We first prove \eqref{step1}.
		We have the formula
		\begin{align*}
			\partial_s^m	X'(t,s)&=\int \partial_sK(t/2,s-y) \partial_s ^{m}X(t/2,y)dy+\int_{t/2}^t\int \partial_sK(t-\tilde \tau,s-y)\partial_s^mN(\tilde \tau,y)dyd\tilde \tau\\
			&:=(\partial_s^m	X')_L(t,s)+(\partial_s^m	X')_N(t,s).
		\end{align*}
		For the linear part, we have 
		\begin{align*}
			\left\|(\partial_s^m	X')_L(t,\cdot)\right\|_{L^\infty}&\leq\sup_s\int\left| \partial_sK(t/2,s-y) (\partial_s^m	X(t/2,y)-\partial_s^m	X(t/2,s))\right|dy\\
			&\lesssim \int |\partial_sK(t/2,y)||y|^\frac{3}{4}dy\|X'(t/2)\|_{\dot C^{m-\frac{1}{4}}}\lesssim t^{-m},
		\end{align*}
		where in the last inequality we used the fact that 
		\begin{align}\label{kkk}
			\int |\partial_sK(t,y)||y|^{\sigma}dy\lesssim t^{-{(1-\sigma)}}, \ \ \ \forall \sigma\in[0,1).
		\end{align}
		For the nonlinear part, we have 
		\begin{align*}
			\left\|(\partial_s^m	X')_N(t,\cdot)\right\|_{L^\infty}&\leq \sup_s \int_{t/2}^t\int\left| \partial_sK(t-\tilde \tau,s-y)(\partial_s^mN(\tilde \tau,y)-\partial_s^mN(\tilde \tau,s))\right|dyd\tilde \tau\\
			&\lesssim\int_{t/2}^t \int |\partial_sK(t-\tilde \tau,y)|y|^\frac{1}{2}dy \|\partial_s^m N(\tilde \tau)\|_{\dot C^\frac{1}{2}}d\tilde \tau.
		\end{align*}
		By Lemma \ref{lemhigh} we have 
		\begin{align*}
			\|\partial_s^m N(\tilde \tau)\|_{\dot C^\frac{1}{2}}\lesssim \|X'\|_{\dot C^{\frac{1}{4}}}^{4m+2}+\|X'\|_{\dot C^{m-\frac{1}{4}}}^\frac{4m+2}{4m-1}\lesssim t^{-m-\frac{1}{2}}.
		\end{align*}
		Combining this with \eqref{kkk}, we obtain 
		\begin{align*}
			\left\|(\partial_s^m	X')_N(t,\cdot)\right\|_{L^\infty}&	\lesssim t^{-(m+\frac{1}{2})}\int_{t/2}^t (t-\tilde \tau)^{-\frac{1}{2}}d\tau \lesssim t^{-m}.
		\end{align*}
		Hence we obtain
		\begin{align*}
			t^m\|\partial_s^m	X'(t)\|_{L^\infty}\lesssim 1.
		\end{align*}
		Then we prove \eqref{step2}. Note that 
		\begin{align*}
			\left\|(\partial_s^m	X')_L(t,\cdot)\right\|_{\dot C^\frac{3}{4}}&\leq\sup_s\int\left| \Lambda^\frac{3}{4}K(t/2,s-y) \partial_s^m	X'(t/2,y)\right|dy\\
			&\lesssim \int |\Lambda^\frac{3}{4}K(t/2,y)|dy\|\partial_s^mX'(t/2)\|_{L^\infty}\overset{\eqref{fractionald}}\lesssim t^{-m-\frac{3}{4}}.
		\end{align*}
		Moreover, we have
		\begin{align*}
			|\delta_\alpha(\partial_s^m	X')_N(t,s)|&\leq  \int_{t/2}^t\int\left| \partial_sK(t-\tilde \tau,s-y)(\delta_\alpha\partial_s^mN(\tilde \tau,y)-\delta_\alpha\partial_s^mN(\tilde \tau,s))\right|dyd\tilde \tau\\
			&\lesssim|\alpha|^\frac{1}{2}\int_{t/2}^t \int |\partial_sK(t-\tilde \tau,y)|y|^\frac{3}{4}dy \|\partial_s^m N(\tilde \tau)\|_{\dot C^\frac{7}{8}}d\tilde \tau\\
			&\lesssim |\alpha|^\frac{1}{2}t^{-(m+\frac{7}{8})}\int_{t/2}^t (t-\tilde \tau)^{-\frac{1}{4}}d\tau \lesssim|\alpha|^\frac{1}{2} t^{-m-\frac{3}{4}},
		\end{align*}
		which implies 
		$$
		\left\|(\partial_s^m	X')_N(t,\cdot)\right\|_{\dot C^\frac{3}{4}}\lesssim t^{-m-\frac{3}{4}}.
		$$
		This completes the proof of \eqref{step2}.
		
		Combining \eqref{step1}, \eqref{step2} with \eqref{smoresultX}, we obtain the following result
		\begin{lemma}\label{lemhighre}
			Let $X\in\mathcal{G}_T^1$ be a solution to \eqref{eqpeskin} with initial data $X_0\in\tilde{\mathcal{G}}_T^1$. Then  there holds
			\begin{align*}
				t^k \|X'(t)\|_{\dot C^k}\lesssim 1, \quad\quad\quad \forall k\in\mathbb{Z}^+, \ t\in[0,T].
			\end{align*}
		\end{lemma}

		\section{Proof of the main theorems}
		We first state the  existence theorems for smooth initial data (see \cite{LinTongSolvability2019,MoriWell2019}).
		\begin{theorem}\label{localexistence} Let $\gamma>1$.
			Consider $X_0\in C^{\gamma}$ satisfying $\kappa(X_0)\leq r_0$ for some constant $r_0$. Then there exists $T=T(\|X_0\|_{ C^{\gamma}},r_0)>0$  such that the problem \eqref{peskin} admits a solution $X\in C([0,T]; C^{\gamma-\varepsilon})$ for any   $\varepsilon\in (0,1)$ and $X(t)\in C^\infty$ for any $t\in (0,T]$.
		\end{theorem}
		\begin{theorem}\label{globallemma} Let $\gamma>1$.
			There exists a constant $\rho_0>0$ such that, for any $X_0\in  C^{\gamma}$, if $\|\Pi X_0\|_{\dot C^{\gamma}}\leq \rho_0\|\mathcal{P}X_0\|_{\dot C^1}$, then the solution to the Peskin problem  \eqref{peskin} exists for all time and converges to a circle $Z_\infty\in \tilde{\mathcal{V}}$. More precisely, for any $0<\varepsilon<1$ and $t\geq 0$ there holds
			\begin{align*}
				\|\Pi X(t)\|_{C^{\gamma-\varepsilon}}\leq C\|\Pi X_0\|_{C^{\gamma-\varepsilon}}e^{-\frac{t}{4}},\quad \|X(t)-Z_\infty\|_{C^{\gamma-\varepsilon}}\leq C\|\Pi X_0\|_{C^{\gamma-\varepsilon}}e^{-\frac{t}{4}}.
			\end{align*}
		\end{theorem}\vspace{0.3cm}
		\begin{proof}[Proof of Theorem \ref{thmlocal}] 
			Note that 	$\kappa_0=\liminf_{\vartheta\rightarrow 0}\kappa(X_0\ast\rho_\vartheta)<+\infty$, hence there exists a sequence $\{\vartheta_m\}_m$ such that $\lim_{m\rightarrow+\infty}\vartheta_m=0 $ and $\sup_m\kappa(X_0\ast\rho_{\vartheta_m})\leq 2\kappa_0$. 
			Denote $X_{0,m}=X_0\ast \rho_{\vartheta_m}$, where $\rho_\vartheta$ is the standard mollifier. By Theorem \ref{localexistence}, there exist $T_1=T_1(\|X_0\|_{L^\infty}, \kappa_0, \vartheta_m)>0$ and a solution $X_m\in C([0,T_1];\dot C^\frac{3}{2})$ with initial data $X_{0,m}$. Without loss of generality, let $T_2=\sup\left\{\tau:X_{m}\in C([0,\tau];\dot C^\frac{3}{2}) \right\}$. Let $T_0=\min\{T^*,T_2\}$. Denote  $\kappa_m(t)=\kappa_{X_{m}}(t)$. 
			By Lemma \ref{estMf22} and Proposition \ref{propdeltaN}, we have 
			for any $t\in[0,T_0]$
			\begin{align*}
				\|X_m'\|_{\mathcal{G}_t}\leq& C\left(1+\kappa_m(t)\right)^2 \|X_m'\|_{\mathcal{G}_t}^2\left(1+\|X_m'\|_{\mathcal{G}_t}\right)^2+\|X_{0,m}'\|_{\tilde {\mathcal{G}}_{T_0}}.
			\end{align*}
			Fix $0<\xi_0\leq10^{-10}(2+\tilde{C}+\kappa_0+\kappa_0^{-1})^{-10}$ for $\tilde{C}\geq C$.
			Denote
			\begin{align*}
				\tilde T_0=\sup\left\{t:t\leq T_0, \|X_m'\|_{\mathcal{G}_t}\leq 3\xi_0,\kappa_m(t)\leq 3\kappa_0\right\}.
			\end{align*}
			Note that $\|X_{0,m}'\|_{\tilde {\mathcal{G}}_{T_0}}\leq \frac{3}{2}\|X_0'\|_{\tilde {\mathcal{G}}_{T^*}}\leq \frac{3}{2}\xi_0$.	By continutity we have $\tilde T_0>0$.
			We want to prove that $	\tilde T_0=T_0$. By the standard bootstrap argument, it suffices to prove 
			\begin{align}\label{goal1}
				\|X_m'\|_{\mathcal{G}_{\tilde T_0}}	\leq 2\xi_0,~~ \kappa_m(\tilde T_0)\leq 2\kappa_0.
			\end{align}
			Then for any $t\in[0,\tilde T_0]$, we have 
			\begin{align*}
				\|X_m'\|_{\mathcal{G}_t}\leq&4C \left(1+3\kappa_0\right)^{2} \xi_0^2(1+2\xi_0)^2+\frac{3}{2}\xi_0\leq 2\xi_0.
			\end{align*}
			It remains to estimate $\kappa_m(t)$. We first obtain from Proposition \ref{propQT} that $t\in[0,\tilde T_0]$
			\begin{align*}
				Q_{X_m}(t)\leq C_1(1+\kappa_m(t))^4\xi_0\leq C_1(1+3\kappa_0)^4\xi_0.
			\end{align*}
			By \eqref{smoresultX} we obtain $t\in[0,\tilde T_0]$
			\begin{align*}
				\|X_m(t)\|_{\dot C^{\frac{3}{2}}}&\leq C_2 t^{-\frac{1}{2}}\left( \|X_0'\|_{\tilde {\mathcal{G}}_t}+(1+\kappa_m(t))^2 \|X_m'\|_{\mathcal{G}_t}^2(1+\|X_m'\|_{\mathcal{G}_t})^2\right)\\&\leq C_2 t^{-\frac{1}{2}}\left( \frac{3}{2}\xi_0+9(1+3\kappa_0)^2 \xi_0^2(1+3\xi_0)^2\right)\leq 2C_2\xi_0 t^{-\frac{1}{2}}.
			\end{align*} We fix $\tilde{C}\geq 10( C+C_1+C_2)$ in the defintion of $\xi_0$.
			Applying Lemma \ref{lemkappa} with $\varepsilon	=2(C_1+C_2)(1+3\kappa_0)^4\xi_0$. Then we obtain
			\begin{align*}
				\kappa_m(t)\leq 2\kappa_0,   \ \ \forall t\in[0,\tilde T_0],
			\end{align*}
			which completes the proof of \eqref{goal1}. Hence for any $t\in[0,\tilde T_0]$, there holds $\|X_m'\|_{\mathcal{G}_t}\leq 2\xi_0,~~ \kappa_m(t)\leq  2\kappa_0.$ Hence $T_0=\tilde T_0$.
			
			Next we prove that $T_0=T^*$. If this is not true, then $T_0=T_2$. Then we have $\|X'_m\|_{\mathcal{G}_{T_2}}	\leq 2\xi_0,\kappa_m(T_2)\leq 2\kappa_0$.
			Moreover, by the smoothing effect \eqref{smoresultX},  one has 
			$$
			\|X_m(T_2)\|_{\dot C^{\frac{3}{2}+\varepsilon'}}\lesssim T_2^{-\frac{1}{2}-\varepsilon'}.
			$$
			Then applying Theorem \ref{localexistence}, there exists $\delta>0$ such that $X_m\in C([0,T_2+\delta];\dot C^\frac{3}{2})$, which contradicts the definition of $T_2$. Hence $X_m\in C([0,T^*];\dot C^\frac{3}{2})$. Moreover, we have 
			\begin{equation*}
				||X_m||_{L^\infty_{t}L^\infty}\leq  C\left(1+\kappa_m(t)\right)^2 \|X_m'\|_{\mathcal{G}_t}^2\left(1+\|X_m'\|_{\mathcal{G}_t}\right)^2+C||X_0||_{L^\infty}.
			\end{equation*} Hence $||X_m||_{L^\infty_{T^*}L^\infty}\lesssim ||X_0||_{L^\infty}+1$.  By standard compactness argument, we are able to pass to the limit $m\rightarrow +\infty$ to get a solution $X\in C((0,T^*],\dot C^\frac{3}{2})$ of the Cauchy problem \eqref{peskin} with initial data $X_0$. The solution also satisfies the above estimates. Moreover, Lemma \ref{lemhighre} also implies that 	$$t^k \|X'(t)\|_{\dot C^k}\lesssim_k \xi_0, \quad\quad\quad \forall k\in\mathbb{Z}^+, \ t\in[0,T^*].$$ Then we complete the proof of the theorem.
		\end{proof}		\vspace{0.3cm}\\
		In the following, we give a proof of Theorem \ref{thmglobal}.	\vspace{0.3cm}\\
		\begin{proof}[Proof of Theorem \ref{thmglobal}] There exists a sequence $\{\vartheta_m\}_m$ such that $\lim_{m\rightarrow+\infty}\vartheta_m=0 $ and $\sup_m\kappa(Y_0\ast \rho_{\vartheta_m}+Z_0)\leq \frac{2}{c_0}$. Denote $Y_{0,m}=Y_0\ast \rho_{\vartheta_m}$.  Then $X_{0,m}=Y_{0,m}+Z_0$ is smooth. By Theorem \ref{localexistence}, there exist $T_1>0$ and a solution $X_m=Y_m+Z_m\in C([0,T_1];\dot C^\frac{3}{2})$. Without loss of generality, let $T_2=\sup\left\{\tau:X_m\in C([0,\tau];\dot C^\frac{3}{2}) \right\}$. Denote  $\kappa_m(t)=\kappa_{X_m}(t)$. 	For any $t\in(0,1)$, Lemma  \ref{propdeltaNY} and Proposition \ref{propparabolicY} yield
			\begin{align}
				&\|Y_m'\|_{\mathcal{G}_t}\leq C\Big(1+\sup_{\tau\in[0,t]}(\|Z'_m(\tau)\|_{L^\infty}^{-1})+\kappa_m(t)\Big)^{5}\nonumber\\
				&\quad\quad\quad\quad\quad\times(1+ \|Z'_m\|_{L^\infty_tL^\infty}+\|Y_m'\|_{\mathcal{G}_t})^{5}\|Y_m'\|_{\mathcal{G}_t}^2+\|Y_{0,m}'\|_{\tilde {\mathcal{G}}_t}.\label{estGY}
			\end{align}
			Moreover, 
			\begin{equation}\label{estGX}
				\|X_m'\|_{\mathcal{G}_t}\leq 	\|Y_m'\|_{\mathcal{G}_t}+Ct^{\varepsilon'}||Z'_m||_{L^\infty_{t}L^\infty},
			\end{equation}
			\begin{equation}\label{linf'}
				||X_m||_{L^\infty_{t}L^\infty}\leq  C\left(1+\kappa_m(t)\right)^{2}(1+\|X_m'\|_{\mathcal{G}_t})^{2}\|X_m'\|_{\mathcal{G}_t}^2+||X_0||_{L^\infty}.
			\end{equation}
			Recalling Proposition \ref{propQT} one has
			\begin{equation}\label{estQX}
				Q_{X_m}(t)\leq  C(1+\kappa_m (t))^4\|X_m'\|_{\mathcal{G}_t}(1+\|X_m'\|_{\mathcal{G}_t})^2.
			\end{equation}
			Since
			$
			\partial_t Z_m=\mathcal{P}\mathfrak{N}(Y_m+Z_m)$, 
			so for any $t\in(0,1)$, there holds
			$$
			\left|	\|Z_m (t)\|_{\dot C^\frac{1}{2}}-	\|Z_0\|_{\dot C^\frac{1}{2}}\right|\lesssim \sup_{\alpha}\frac{\|\delta_\alpha \mathfrak{N}(Y_m +Z_m )\|_{L^2_tL^\infty}}{|\alpha|^\frac{1}{2}}.$$
			Recall that $Z_m =\mathcal{P}X_m \in\mathcal{V}=\operatorname{span}\{e_r,e_t,e_x,e_y\}$ belongs to a finite dimensional space. Hence all the norms are equivalent. Specially, we have $\|Z\|_{\dot C^\sigma}=c_\sigma \|Z'\|_{L^\infty}$ for any $Z\in\mathcal{V}$. Combining this with Proposition \ref{propparabolicY}, we obtain
			\begin{align}\label{estZ}
				&\sup_{\tau\in[0,t]}\left|	\|Z'_m (\tau)\|_{L^\infty}-\|Z_0'\|_{L^\infty}\right|\\
				\leq& C \Big(1+\sup_{\tau\in[0,t]}(\|Z'_m (\tau)\|_{L^\infty}^{-1})+\kappa_m (t)\Big)^{5}\Big(1+ \|Z'_m \|_{L^\infty_tL^\infty}+\|Y_m'\|_{\mathcal{G}_t}\Big)^{5}\|Y_m'\|_{\mathcal{G}_t}^2.\nonumber
			\end{align}	 
			Fix $0<\xi_1\leq (10+\tilde{C}+c_0^{-1})^{-\frac{100}{\varepsilon'}}$ for $\tilde{C}\geq C$ and let  $T_0=\min\{T_2, \xi_1^\frac{2}{3}\}$. Define
			\begin{align*}
				\tilde T_0=\sup&\Big\{t:t\leq T_0, \|Y_m'\|_{\mathcal{G}_t}\leq 3\xi_1,  \kappa_m (t)\leq 5c_0^{-1},\\
				&\quad\quad\quad\quad\quad\quad\sup_{\tau\in[0,t]}\left|\|Z'_m (\tau)\|_{L^\infty}-\|Z_0'\|_{L^\infty}\right|\leq 2\xi_1\Big\}.
			\end{align*}
			Since
			$
			\kappa_m (0)\leq2c_0^{-1}
			$, we have $\tilde T_0>0$.	We want to prove that $	\tilde T_0=T_0$. By the standard bootstrap argument, it suffices to prove for any $t\in[0,\tilde T_0]$ 
			\begin{equation}\label{goal2}
				\|Y_m'\|_{\mathcal{G}_t}\leq 2\xi_1, ~\kappa_m (t)\leq 4c_0^{-1},~ \sup_{\tau\in[0,t]}\left|\|Z'_m (\tau)\|_{L^\infty}-\|Z_0'\|_{L^\infty}\right|\leq \xi_1.
			\end{equation}
			Note that $\|Y_{0,m}'\|_{\tilde {\mathcal{G}}_t}\leq \frac{3}{2}\|Y_{0}'\|_{\tilde {\mathcal{G}}}\leq \frac{3}{2}\xi_1$. Then \eqref{estGY} implies
			$$
			\|Y_m'\|_{\mathcal{G}_t}\leq9 C^3\left(2+7c_0^{-1}\right)^{5}(1+2c_0^{-1}+3\xi_1)^{5}\xi_1^2+\xi_1\leq 2\xi_1,~~~\forall t\in[0,\tilde T_0].
			$$
			By  \eqref{estZ} and using the fact that $\|Z_0'\|_{L^\infty} \in [c_0,\frac{1}{c_0}]$, we have for any $t\in[0,\tilde T_0]$ 
			\begin{align*}
				\sup_{\tau\in[0,t]}\left|	\|Z'_m (\tau)\|_{L^\infty}-\|Z_0'\|_{L^\infty}\right|\leq 9C^3\left(2+7c_0^{-1}\right)^{5}(1+2c_0^{-1}+3\xi_1)^{5}\xi_1^2\leq \xi_1.
			\end{align*}
			It remains to estimate $\kappa_m (t)$. By \eqref{estGX}
			and \eqref{estQX}, we have for any $t\in[0,\tilde T_0]$ 
			\begin{align*}
				\quad Q_{X_m} (t)&\leq C (1+\kappa_m (t))^4\left(\|Y_m'\|_{\mathcal{G}_t}+Ct^{\varepsilon'}||Z'_m||_{L^\infty_{t}L^\infty}\right)\left(1+\|Y_m'\|_{\mathcal{G}_t}+C||Z'_m||_{L^\infty_{t}L^\infty}\right)^2\\
				&\leq C\left(1+5c_0^{-1}\right)^4\left(3\xi_1+Ct^{\varepsilon'}(1+c_0^{-1})\right)\left(1+C+3\xi_1+Cc_0^{-1}
				\right)^{2}\leq \xi_1^{\frac{\varepsilon'}{2}}.
			\end{align*}
			Moreover, by \eqref{smoresultY} we have for any $t\in[0,\tilde T_0]$ and $\varepsilon_1\in [0,\varepsilon']$
			\begin{align}\nonumber
				\|Y_m (t)\|_{\dot C^{\frac{3}{2}+\varepsilon_1}}&\leq C_1 t^{-\frac{1}{2}-\varepsilon_1} \Big\{\Big(1+ \|Z'\|_{L^\infty_tL^\infty}+\|Y'\|_{\mathcal{G}_t}\Big)^{5}\|Y'\|_{\mathcal{G}_t}^2\\
				&\quad\quad\quad\quad\quad\quad\times\Big(1+\kappa(t)+\sup_{\tau\in[0,t]}(\|Z'(\tau)\|_{L^\infty}^{-1})\Big)^{5}+\|Y_0'\|_{\tilde {\mathcal{G}}_t}\Big\}\nonumber\\
				&\leq 2C_1\xi_1t^{-\frac{1}{2}-\varepsilon_1}.\label{esym}
			\end{align}
			This leads to 
			\begin{align*}
				\|X_m (t)\|_{\dot C^\frac{3}{2}}\leq 2C_1\xi_1t^{-\frac{1}{2}}+C_2\|Z_0'\|_{L^\infty}\leq \xi_1^{\frac{\varepsilon'}{2}}t^{-\frac{1}{2}},
			\end{align*}
			for any $t\in[0,\tilde T_0]$, 	provided $\tilde{C}\geq C+C_1+C_2$.	Hence we can apply Lemma \ref{lemkappa} with $\varepsilon=\xi_1^{\frac{\varepsilon'}{2}}$, which yields
			\begin{align*}
				\kappa_m (t)\leq 2\kappa_m (0)\leq 4c_0^{-1},
			\end{align*}
			which completes the proof of \eqref{goal2}. We obtain $\tilde T_0=T_0$.\medskip\\
			We claim that $T_0=\xi_1^\frac{2}{3}$. If this is not true, then $T_0=T_2$. One has $\kappa_m (T_2)\leq 4c_0^{-1}$. Moreover, by the smoothing effect \eqref{smoresultY}, we have
			$$\|Y_m (T_2)\|_{\dot C^{\frac{3}{2}+\varepsilon'}}\leq 2C_1\xi_1T_2^{-\frac{1}{2}-\varepsilon'}.$$ Then $\|X_m (T_2)\|_{\dot C^{\frac{3}{2}+\varepsilon'}}\overset{\eqref{goal2}}\lesssim T_2^{-\frac{1}{2}-\varepsilon'}+c_0^{-1}+1$.
			Applying Theorem \ref{localexistence}, there exists $\delta>0$ such that $X_m \in C([0,T_2+\delta];\dot C^\frac{3}{2})$. This contradicts the definition of $T_2$. Hence we have $T_0=\xi_1^\frac{2}{3}$ and $X_m \in C([0,T_0];\dot C^\frac{3}{2})$. Note that $T_0$ is independent of $m $. Moreover, by \eqref{linf'} one has, $||X_m||_{L^\infty_{T_0}L^\infty}\lesssim ||X_0||_{L^\infty}+1$.  By standard compactness argument, we are able to pass to the limit $m \rightarrow\infty$ to get a solution $X\in C((0,T_0],\dot C^\frac{3}{2})$ of the Cauchy problem \eqref{peskin} with initial data $X_0$. The solution also satisfies the above estimates. Then we obtain \textbf{1)} and \textbf{2)} in Theorem \ref{thmglobal}.	Moreover, by \eqref{esym}
			\begin{align*}
				\|Y(T_0)\|_{\dot C^{\frac{3}{2}+\varepsilon'}}\leq 2C_1\xi_1T_0^{-\frac{1}{2}-\varepsilon'}\leq \xi_1^\frac{1}{4}.
			\end{align*}
			We can further choose $\xi_1$ such that $\xi_1^\frac{1}{4}\leq \rho_0\|Z_0'\|_{L^\infty}$. Applying Theorem \ref{globallemma} one gets \textbf{3)} in Theorem \ref{thmglobal}. This completes the proof of Theorem \ref{thmglobal}. 
		\end{proof}	\vspace{0.3cm}\\
	
		\begin{proof}[Proof of Proposition \ref{Uniqueness}]
			Denote $W=X-\bar X$, we have 
			\begin{align*}
				\partial_t W+\frac{1}{4}\Lambda W&=\sum \int \left(H(\tilde\Delta_\alpha X)E^\alpha X_{i} \delta_\alpha X'_{j}-H(\Delta_\alpha \bar X)E^\alpha \bar X_{i} \delta_\alpha\bar X'_{j}\right)\frac{d\alpha}{\tilde\alpha}:=\tilde N(X,\bar X).
			\end{align*}
			We have 
			\begin{align*}
				\tilde N(X,\bar X)(t,s)=&\int H(\tilde\Delta_\alpha X(s))E^\alpha X_i(s)\delta_\alpha W'_j(s)\frac{d\alpha}{\alpha}+\int H(\tilde\Delta_\alpha \bar X(s))E^\alpha W_i(s) \delta_\alpha \bar X_j'(s)\frac{d\alpha}{\alpha}\\
				&+\int [H(\tilde\Delta_\alpha X)-H(\tilde\Delta_\alpha \bar X)]E^\alpha \bar X_i(s)\delta_\alpha \bar X'_j(s) \frac{d\alpha}{\alpha}.
			\end{align*}
			Following the estimates in Proposition \ref{propparabolicY}, we obtain for any $0<t<T$,
			\begin{align*}
				&\sup_{0\leq\mu\leq\frac{2}{3}\atop\theta-\varepsilon'\leq \mu+a\leq1-\varepsilon'}\sup_\beta \frac{\|t^\mu\delta_\beta \tilde N\|_{L^\frac{1}{a}_TL^\infty}}{|\beta|^{\mu+a}}\\
				&\quad\lesssim\left(1+\kappa_X(T)+\kappa_{\bar X}(T)\right)^{2}\|W'\|_{\mathcal{G}_T}(\|X'\|_{\mathcal{G}_T}+\|\bar X'\|_{\mathcal{G}_T})(1+\|X'\|_{\mathcal{G}_T}+\|\bar X'\|_{\mathcal{G}_T})^5.
			\end{align*}
		Hence one has
			\begin{align*}
				\|W'\|_{\mathcal{G}_T}\leq& \|W_0'\|_{\tilde{\mathcal{G}}_T}+C\|W'\|_{\mathcal{G}_T}\left(1+\kappa_X(T)+\kappa_{\bar X}(T)\right)^{2}\\
				&\quad\quad\quad\quad\quad\times(\|X'\|_{\mathcal{G}_T}+\|\bar X'\|_{\mathcal{G}_T})(1+\|X'\|_{\mathcal{G}_T}+\|\bar X'\|_{\mathcal{G}_T})^5.
			\end{align*}
			By $\left(1+\kappa_X(T)+\kappa_{\bar X}(T)\right)^{2}(\|X'\|_{\mathcal{G}_T}+\|\bar X'\|_{\mathcal{G}_T})(1+\|X'\|_{\mathcal{G}_T}+\|\bar X'\|_{\mathcal{G}_T})^5\ll 1$, we obtain the result.
		\end{proof}\vspace{0.1cm}
		\section{Appendix}
	We will  prove Lemma \ref{lemI12}. First, we need the following lemma: 
		\begin{lemma}	\label{lemnew} 
			Denote $$\Phi_0(x)=\frac{1}{|x|^{\sigma_1}}\mathbf{1}_{|x|\leq1}+\frac{1}{|x|^{\sigma_2}}\mathbf{1}_{|x|>1}$$
			for some $\sigma_1\in(-1,1)\backslash\{0\}, \sigma_2\in (0,1)$.
			Let $\Phi:\mathbb{R}\to (0,\infty) $ satisfy 
			\begin{align}\label{condition}
				\sum_{m=0}^3	|x|^m\left|\frac{d^m}{dx^m} \Phi(x)\right|\lesssim \Phi_0(x), ~~~\forall x\in\mathbb{R}\backslash\{0\}.
			\end{align}
			Then there holds
			$$
			|(\mathcal{F}\Phi)(\xi)|+	|\xi|\left|\frac{d}{d\xi}(\mathcal{F}\Phi)(\xi)\right|\lesssim  \frac{\Phi_0(1/|\xi|)}{|\xi|}.$$
		\end{lemma}
		\begin{proof} We use the idea in \cite[Proof of Lemma 3.8]{HungNguyen}.
			Let $1=\sum_{n=-\infty}^{+\infty}\chi_n(x)$ be the standard partition of unity in $\mathbb{R}\backslash\{0\}$, where  $\text{supp} \chi_n \subset \{|x|\sim 2^{-n}\}$ and $|\frac{d^m}{dx^m}\chi_n(x)|\lesssim 2^{nm}$ for any $m\in \mathbb{N}$. Then we have
			$$
			(\mathcal{F}\Phi)(\xi)=\sum_{n=-\infty}^{+\infty}\int \chi_n(x) \Phi(x)e^{-ix\xi} dx.$$
			By \eqref{condition}, integrate by parts one obtains
			\begin{align*}
				|\xi|^k	\left|\int \chi_n \Phi(x)e^{-ix\xi} dx\right|\lesssim  \int |\partial^k(\chi_n \Phi)| dx\lesssim 2^{(k-1)n}\Phi_0(2^{-n}),
			\end{align*}
			for $k=0,1,2$.
			Hence we obtain
			\begin{equation*}
				\left|\int \chi_n \Phi(x)e^{-ix\xi} dx\right|\lesssim \Phi_0(2^{-n})2^{-n}\min\left\{1,\frac{2^{n}}{|\xi|}\right\}^2.
			\end{equation*}
			Then by the definition of $\Phi_0$, it is easy to check that$$
			|\mathcal{F} \Phi(\xi)|\lesssim \sum_{n=-\infty}^\infty \Phi_0(2^{-n})2^{-n}\min\left\{1,\frac{2^{n}}{|\xi|}\right\}^2\lesssim  \frac{\Phi_0(1/|\xi|)}{|\xi|}.$$
			Similarly, we have 
			$$
			\left|\frac{d}{d\xi}(\mathcal{F}\Phi)(\xi)\right|\lesssim  \frac{\Phi_0(1/|\xi|)}{|\xi|^2},$$
			this completes the proof.
		\end{proof}	\vspace{0.3cm}\\
		\begin{proof}[Proof of Lemma \ref{lemI12}]
			Let $\phi(\xi)=\phi(|\xi|)$ be a smooth  function in $\mathbb{R}\backslash\{0\}$ satisfying 
			\begin{align*}
				\phi(|\xi|)=\begin{cases}
					|\xi|^{1-\sigma-\varepsilon},~~~|\xi|\leq 1,\\
					|\xi|^{1-\sigma+\varepsilon},~~~|\xi|\geq 2.
				\end{cases}
			\end{align*}
			Let $g_1=\Lambda^\phi g$, where $\Lambda^\phi$ is the Fourier multiplier with
			symbol $\phi(\xi)$, i.e. $(\mathcal{F}g_1)(\xi)=\phi(\xi)(\mathcal{F}g)(\xi)$. Then we have 
			$g=G\ast g_1,$
			where $G=\frac{1}{(2\pi)^2}\mathcal{F}\left(\frac{1}{\phi}\right)$. We can choose $\phi$ such that the function $\frac{1}{\phi}$ satisfies condition \eqref{condition} in Lemma \ref{lemnew} with $\Phi=\Phi_0=\frac{1}{\phi}$. Applying Lemma \ref{lemnew} we obtain 
			$$
			\left|G(x)\right|+|x|\left|\frac{d}{dx}G(x)\right|\lesssim \frac{1}{\phi(1/|x|)|x|},$$
			which implies 
			\begin{align}\label{estofG}
				\left|\frac{d^k}{dx^k}G(x)\right|\lesssim\min\left\{|x|^{2\varepsilon},1\right\}|x|^{-k-\sigma-\varepsilon}, ~k=0,1.
			\end{align}
			Let  $f_1=\Lambda^\sigma f$,  then $f=\frac{c_\sigma}{|.|^{1-\sigma}}\ast f_1$. 
			We have 
			\begin{align}
				\label{1}	&\left|\int f(\alpha)\partial_\alpha \tilde g(\alpha)d\alpha\right|\lesssim \left|\iint f_1(z) \partial_\alpha\left(\frac{1}{|\alpha-z|^{1-\sigma}}\right)\Delta_\alpha (G\ast g_1)(0)d\alpha dz\right|\\
				&\lesssim ||f_1||_{L^\infty} ||g_1||_{L^\infty}\iint \left|	\int \partial_\alpha\left(\frac{1}{|\alpha-z|^{1-\sigma}}\right)\left(G(\alpha-y)-G(y)\right)\frac{d\alpha}{\alpha}\right|dydz.\nonumber
			\end{align}
			We first prove that 
			\begin{align}\label{estM}
				E=	\iint \left|	\int \partial_\alpha\left(\frac{1}{|\alpha-z|^{1-\sigma}}\right)\left(G(|\alpha-y|)-G(|y|)\right)\frac{d\alpha}{\alpha}\right|dydz<\infty.
			\end{align}
			Let $\chi(x)$ be a smooth positive symmetric function such that $\mathbf{1}_{|x|\leq 1/2}\leq \chi\leq \mathbf{1}_{|x|\leq 1}$. Then
			\begin{align*}
				1&=\chi(4|\alpha|/|z|)+\left(\chi(|\alpha|/(4|z|))-\chi(4|\alpha|/|z|)\right)+\left(1-\chi(|\alpha|/(4|z|))\right)\\
				&:=\chi_1+\chi_2+\chi_3,\\
				1&=[\chi(|\alpha|/(4|y|))-\chi(4|\alpha|/|y|)]+[1-\chi(|\alpha|/(4|y|))+\chi(4|\alpha|/|y|)]\\
				&:=\psi_1+\psi_2.
			\end{align*}
			One has
			\begin{align*}
				E_1&=	\iint \left|	\int \chi_1\partial_\alpha\left(\frac{1}{|\alpha-z|^{1-\sigma}}\right)\left(G(|\alpha-y|)-G(|y|)\right)\frac{d\alpha}{\alpha}\right|dydz\\&\leq   	\iiint \mathbf{1}_{|\alpha|\leq |z|/2} \frac{1}{|z|^{2-\sigma}}|G(|\alpha-y|)-G(|y|)|\frac{d\alpha dydz}{|\alpha|}.
			\end{align*}
			Note that 
			\begin{align}
				|G(\alpha-y)-G(y)|\lesssim&\min\left\{
				\frac{|\alpha|}{|y|},1\right\}\min\left\{|y|^{2\varepsilon},1\right\}|y|^{-\sigma-\varepsilon}\nonumber\\
				&\quad+\mathbf{1}_{|\alpha-y|\leq \frac{1}{2} |y|}\min\left\{|\alpha-y|^{2\varepsilon},1\right\}|\alpha-y|^{-\sigma-\varepsilon}.\label{z1}
			\end{align}
			Thus, we have
			\begin{align*}
				E_1&\lesssim   	\iiint_{|\alpha|\leq |z|/2} \frac{1}{|z|^{2-\sigma}}\min\left\{
				\frac{1}{|y|},\frac{1}{|\alpha|}\right\}\min\left\{|y|^{2\varepsilon},1\right\}|y|^{-\sigma-\varepsilon}d\alpha dydz\\
				&\quad\quad\quad\quad+ 	\iiint\mathbf{1}_{|y|\lesssim |z|} \mathbf{1}_{|\alpha|\lesssim  |y|} \frac{1}{|z|^{2-\sigma}}\frac{ 1}{|y|}\min\left\{|\alpha|^{2\varepsilon},1\right\}|\alpha|^{-\sigma-\varepsilon}d\alpha dydz\\
				&\lesssim   	\iint\frac{1}{|\alpha|^{1-\sigma}}\min\left\{
				\frac{1}{|y|},\frac{1}{|\alpha|}\right\}\min\left\{|y|^{2\varepsilon},1\right\}|y|^{-\sigma-\varepsilon} d\alpha dy\\
				&\quad\quad\quad\quad+ 	\iint_{|y|\lesssim |z|}\frac{1}{|z|^{2-\sigma}}\min\left\{|y|^{2\varepsilon},1\right\}|y|^{-\sigma-\varepsilon} dydz\\
				&\lesssim   	\int \min\left\{|y|^{2\varepsilon},1\right\} \frac{dy}{|y|^{1+\varepsilon}}\lesssim 1.
			\end{align*}
			Moreover, integrate by parts  we have
			\begin{align*}
				E_2&=	\iint \left|	\int \chi_2\partial_\alpha\left( \frac{1}{|\alpha-z|^{1-\sigma}}\right)\left(G(|\alpha-y|)-G(|y|)\right)\frac{d\alpha}{\alpha}\right|dydz\\&\leq   	\iiint   \left|\frac{1}{|\alpha-z|^{1-\sigma}} -\frac{1}{|y-z|^{1-\sigma}}\right| \left|\partial_\alpha\left[\frac{\chi_2\psi_1}{\alpha}\left(G(|\alpha-y|)-G(|y|)\right)\right]\right|d\alpha dydz\\&\quad\quad+	\iiint  \frac{1}{|\alpha-z|^{1-\sigma}} \left|\partial_\alpha\left[\frac{\chi_2\psi_2}{\alpha}\left(G(|\alpha-y|)-G(|y|)\right)\right]\right|d\alpha dydz\\ &:=E_{2,1}+E_{2,2}.
			\end{align*}
			By \eqref{estofG} we have
			\begin{align*}
				&\left|\partial_\alpha\left[\frac{\chi_2\psi_1}{\alpha}\left(G(|\alpha-y|)-G(|y|)\right)\right]\right|\lesssim \frac{\mathbf{1}_{|\alpha|\sim |z|\sim |y|}}{|z||\alpha-y|^{1+\sigma+\varepsilon}}\min\left\{|\alpha-y|^{2\varepsilon},1\right\},\\
				&\left|\partial_\alpha\left[\frac{\chi_2\psi_2}{\alpha}\left(G(|\alpha-y|)-G(|y|)\right)\right]\right|\lesssim \frac{\mathbf{1}_{|\alpha|\sim |z|}}{|z|}\frac{|y|^{-\sigma-\varepsilon}}{|y|+|\alpha|}\min\left\{|y|^{2\varepsilon},1\right\}.
			\end{align*}
			Moreover, there holds for any $\sigma,\gamma\in(0,1)$ $$
			\left|\frac{1}{|\alpha-z|^{1-\sigma}} -\frac{1}{|y-z|^{1-\sigma}}\right|\lesssim |\alpha-y|^\gamma\left(\frac{1}{|\alpha-z|^{1-\sigma+\gamma}}+\frac{1}{|y-z|^{1-\sigma+\gamma}}\right).$$
			Let $\gamma>0$ such that $\sigma-\varepsilon<\gamma<\sigma$. Denote $\gamma_1=1-\sigma+\gamma$, we have
			$$
			E_{2,1}\lesssim  	\iiint_{|\alpha|\sim |z|\sim |y|}  \left(\frac{1 }{|\alpha-z|^{\gamma_1}}+\frac{1 }{|y-z|^{\gamma_1}}\right) \frac{\min\left\{|\alpha-y|^{2\varepsilon},1\right\}}{|\alpha-y|^{1+\sigma+\varepsilon-\gamma}}\frac{d\alpha dydz}{|z|}.$$
			Hence 
			\begin{align*}
				E_{2,1}&\lesssim  	\iint_{|z|\sim |y|} \min\{|z|^{2\varepsilon},1\}\frac{dydz}{|z|^{2+\varepsilon}}+	\iint_{ |z|\sim |y|}  \frac{\min\{|z|^{2\varepsilon},1\}}{|y-z|^{\gamma_1}} \frac{ dydz}{|z|^{1-\gamma+\sigma+\varepsilon}}\\
				&\lesssim\int \min\{|z|^{2\varepsilon},1\}\frac{dz}{|z|^{1+\varepsilon}}+\iint_{ |y|\lesssim |z|}  \frac{1}{|y|^{\gamma_1}} \min\{|z|^{2\varepsilon},1\}\frac{ dydz}{|z|^{1-\gamma+\sigma+\varepsilon}}\\
				&\lesssim \int \min\{|z|^{2\varepsilon},1\}\frac{dz}{|z|^{1+\varepsilon}}
				\lesssim  1.
			\end{align*}
			Moreover, we have
			\begin{align*}
				E_{2,2}&\lesssim	\iiint   \frac{\min\left\{|y|^{2\varepsilon},1\right\}}{|\alpha-z|^{1-\sigma}} \frac{\mathbf{1}_{|\alpha|\sim |z|}}{|z|}\frac{|y|^{-\sigma-\varepsilon}}{|y|+|z|}d\alpha dydz\\
				&\lesssim \iint   \frac{\min\left\{|y|^{2\varepsilon},1\right\}}{|z|^{1-\sigma}}\frac{|y|^{-\sigma-\varepsilon}}{|y|+|z|} dydz\lesssim \int \min\left\{|y|^{2\varepsilon},1\right\} \frac{dy}{|y|^{1+\varepsilon}}\lesssim1.
			\end{align*}
			Finally we estimate
			\begin{align*}
				E_3&=	\iint \left|\int \chi_3\partial_\alpha \left(\frac{1}{|\alpha-z|^{1-\sigma}}\right)\left(G(|\alpha-y|)-G(|y|)\right)\frac{d\alpha}{\alpha}\right|dydz
				\\&\lesssim 	\iiint_{|\alpha|\geq2|z|} \frac{1}{|\alpha|^{3-\sigma}}|G(|\alpha-y|)-G(|y|)|d\alpha dydz.
			\end{align*}
			By \eqref{z1} we have 
			\begin{align*}
				E_3\lesssim &	\iiint_{|\alpha|>2|z|}\min\left\{
				\frac{|\alpha|}{|y|},1\right\}\min\left\{|y|^{2\varepsilon},1\right\}|y|^{-\sigma-\varepsilon}\frac{d\alpha dy dz}{|\alpha|^{3-\sigma}}\\
				&\quad\quad+\iiint_{|\alpha|>2|z|}\mathbf{1}_{|\alpha-y|\leq \frac{1}{2} |y|}\min\left\{|\alpha-y|^{2\varepsilon},1\right\}|\alpha-y|^{-\sigma-\varepsilon}\frac{d\alpha dy dz}{|\alpha|^{3-\sigma}}\\
				\lesssim &\iint\min\left\{
				\frac{|\alpha|}{|y|},1\right\}\frac{\min\left\{|y|^{2\varepsilon},1\right\}}{|y|^{\sigma+\varepsilon}}\frac{d\alpha dy }{|\alpha|^{2-\sigma}}+\iint_{|\alpha|\lesssim |y|}\frac{\min\left\{|\alpha|^{2\varepsilon},1\right\}}{|\alpha|^{\sigma+\varepsilon}}\frac{d\alpha dy }{|y|^{2-\sigma}}\\
				\lesssim& \int \min\left\{|y|^{2\varepsilon},1\right\}\frac{dy}{|y|^{1+\varepsilon}}\lesssim 1.
			\end{align*}
			This  completes the proof of \eqref{estM}.\medskip\\ We claim that  
			\begin{align}\label{zzzz}
				\|g_1\|_{L^\infty}\lesssim \|\Lambda^{1-\sigma+2\varepsilon}g\|_{L^\infty}+\|\Lambda^{1-\sigma-2\varepsilon}g\|_{L^\infty}.
			\end{align}
			To prove this, observe that 
			$$
			\phi(\xi)=\frac{\chi(\xi)\phi(\xi)}{|\xi|^{1-\sigma-2\varepsilon}}|\xi|^{1-\sigma-2\varepsilon}+\frac{(1-\chi(\xi))\phi(\xi)}{|\xi|^{1-\sigma+2\varepsilon}}|\xi|^{1-\sigma+2\varepsilon}.$$
			Denote $ \phi_1(\xi)=\frac{\chi(\xi)\phi(\xi)}{|\xi|^{1-\sigma-2\varepsilon}}$ and $ \phi_2(\xi)=\frac{(1-\chi(\xi))\phi(\xi)}{|\xi|^{1-\sigma+2\varepsilon}}$. There holds
			\begin{align*}
				\mathcal{F}g_1(\xi)=\phi(\xi)\mathcal{F}g(\xi)= \phi_1(\xi)|\xi|^{1-\sigma-2\varepsilon}\mathcal{F}g(\xi)+ \phi_2(\xi)|\xi|^{1-\sigma+2\varepsilon}\mathcal{F}g(\xi).
			\end{align*}
			Hence 
			\begin{align}\label{zxc}
				\|g_1\|_{L^\infty}\lesssim \|\mathcal{F}\phi_1\|_{L^1}\|\Lambda^{1-\sigma-2\varepsilon}g\|_{L^\infty}+\|\mathcal{F}\phi_2\|_{L^1}\|\Lambda^{1-\sigma+2\varepsilon}g\|_{L^\infty}.
			\end{align}
			Set $$
			\phi_0(\xi)=|\xi|^\varepsilon\mathbf{1}_{|\xi|\leq 1}+|\xi|^{-\varepsilon}\mathbf{1}_{|\xi|> 1}.
			$$
			Then we have 
			$$
			\sum_{m=0}^3	|\xi|^m\left|\frac{d^m}{d\xi^m} \phi_1(\xi)\right|+\sum_{m=0}^3	|\xi|^m\left|\frac{d^m}{d\xi^m} \phi_2(\xi)\right|\lesssim \phi_0(\xi), ~~~\forall \xi\in\mathbb{R}\backslash\{0\}.$$
			By Lemma \ref{lemnew} we obtain
			\begin{align*}
				|(\mathcal{F}\phi_k)(x)|\lesssim \frac{\phi_0(1/|x|)}{|x|}, ~~~~k=1,2.
			\end{align*}
			Then we obtain
			\begin{align*}
				&\|\mathcal{F}\phi_1\|_{L^1}+\|\mathcal{F}\phi_2\|_{L^1}\lesssim \int_{|x|\geq 1}\frac{dx}{|x|^{1+\varepsilon}}+ \int_{|x|\leq 1}\frac{dx}{|x|^{1-\varepsilon}}\lesssim 1.
			\end{align*}
			This together with \eqref{zxc} implies \eqref{zzzz}.
			Combining \eqref{1}, \eqref{estM} and \eqref{zzzz}  we obtain
			\begin{align*}
				\left|\int f(\alpha)\partial_\alpha \tilde g(\alpha)d\alpha\right|\lesssim \|\Lambda^\sigma f\|_{L^\infty}(\|\Lambda^{1-\sigma-2\varepsilon}g\|_{L^\infty}+\|\Lambda^{1-\sigma+2\varepsilon}g\|_{L^\infty}).
			\end{align*}
			For any $\lambda>0$, denote $f_\lambda(\alpha)=f(\lambda\alpha)$ and $g_\lambda(s)=g(\lambda s)$. Then 
			\begin{align*}
				\left|\int f(\alpha)\partial_\alpha \tilde g(\alpha)d\alpha\right|&=\frac{1}{\lambda}\left|\int f_\lambda(\alpha)\partial_\alpha \tilde g_\lambda(\alpha)d\alpha\right|\\
				&\lesssim \frac{1}{\lambda}\|\Lambda^\sigma f_\lambda\|_{L^\infty}(\|\Lambda^{1-\sigma-2\varepsilon}g_\lambda\|_{L^\infty}+\|\Lambda^{1-\sigma+2\varepsilon}g_\lambda\|_{L^\infty}).
			\end{align*}
			It is easy to check that for any $\gamma\in(0,1)$
			$$\|\Lambda^\gamma f_\lambda\|_{L^\infty}= \lambda^\gamma \|\Lambda^\gamma f\|_{L^\infty}.$$
			Hence we get $$
			\left|\int f(\alpha)\partial_\alpha \tilde g(\alpha)d\alpha\right|\lesssim \|\Lambda^\sigma f\|_{L^\infty}\left(\lambda^{-2\varepsilon}\|\Lambda^{1-\sigma-2\varepsilon}g\|_{L^\infty}+\lambda^{2\varepsilon}\|\Lambda^{1-\sigma+2\varepsilon}g\|_{L^\infty}\right).$$
			Take $\lambda=\left(\|\Lambda^{1-\sigma-2\varepsilon}g\|_{L^\infty}\|\Lambda^{1-\sigma+2\varepsilon}g\|_{L^\infty}^{-1}\right)^\frac{1}{4\varepsilon}$, we obtain the result.
		\end{proof}	\vspace{0.3cm}\\
		\begin{lemma}\label{propparabolic} Let  $f(s)=h'(s)$ for some function $h$. 
			Then for any $\eta\in(0,1)$ and $T>0$, we have
			\begin{align*}
				\|f\|_{\tilde {\mathcal{G}}_T}\lesssim\|f-f\ast \rho_\eta\|_{\tilde {\mathcal{G}}_T}+(T^{\varepsilon'}+T)\eta^{-3}||h||_{L^\infty},
			\end{align*}
		where $\rho_\eta(x)=\eta^{-1}\rho(\frac{x}{\eta})$ is the standard mollifier.
		\end{lemma} 
		\begin{proof}
			For simplicity, denote $f_\eta(t,s)=f\ast\rho_\eta$. 
			Fix $b,\mu$ such that $0\leq\mu\leq\frac{2}{3}$, $ 2\varepsilon'\leq b\leq\theta-\mu-\varepsilon'$.
			We have 
			\begin{align*}
				\|\delta_\alpha \Lambda^{b-\varepsilon'} (K(t,\cdot)\ast f_\eta)\|_{L^\infty}\lesssim \eta^{-3}|\alpha|^{\mu+\varepsilon'}\|h\|_{L^\infty}.
			\end{align*}
		Hence 
			\begin{align*}
				\sup_{\alpha\in\mathbb{R}}\frac{\|t^\mu\delta_\alpha \Lambda^{b-\varepsilon'} (K(t,\cdot)\ast f_\eta)\|_{L_T^\frac{1}{b}L^\infty}}{|\alpha|^{\mu+\varepsilon'}}\lesssim \eta^{-3}T^{\mu+b}\|h\|_{L^\infty},
			\end{align*}
			which implies $\|f\ast \rho_\eta \|_{\tilde {\mathcal{G}}_T}\lesssim (T^{\varepsilon'}+T)\eta^{-3}||h||_{L^\infty}$. We complete the proof.
		\end{proof}	\vspace{0.3cm}\\
	For initial data $X_0\in  (C^2)^{\dot B^1_{\infty,\infty}}\cap L^\infty$, we have $\lim_{\eta\to 0}\|X_0'-X_0'\ast\rho_\eta\|_{\tilde{\mathcal{G}}}=0$.
	For any $\xi_0>0$, we can choose $\eta$ small enough and $T^*=T^*(\eta, \|X_0\|_{L^\infty})$ small enough such that 
$$	\|X_0'-X_0'\ast\rho_\eta\|_{\tilde{\mathcal{G}}}\leq \frac{\xi_0}{3}, \ \ \ \text{and}\ \ \  (T^*)^{\varepsilon'}\eta^{-3}||X_0||_{L^\infty}\leq \frac{\xi_0}{3}.
$$		By Lemma \ref{propparabolic},   we have $\|X_0'\|_{\tilde{\mathcal{G}}_{T^*}}\leq \xi_0$. Applying this to Theorem \ref{thmlocal} we obtain Corollary \ref{corvmo}.
		\vspace{0.5cm}\\
		\noindent	\textbf{Acknowledgments:} This project is supported by the ShanghaiTech University startup fund,Academy of Mathematics and Systems Science,
Chinese Academy of Sciences startup fund, and the National Natural Science Foundation of China (12050410257).
		
	\end{document}